\documentclass[11pt, oneside]{article}

\usepackage{epsfig}
\usepackage{tikz}
\usepackage{url}
\usepackage{amsmath,amstext,amssymb,amsfonts,amsthm}
\usepackage{mathrsfs}
\usepackage{xcolor}
\usepackage{mathtools}
\usepackage{graphicx}

\usetikzlibrary{fadings,shapes,arrows.meta,positioning}
\usepackage{hhline}

\usepackage{booktabs}
\usepackage{multicol}
\usepackage{multirow}
\usepackage[labelfont=bf]{caption}
\usepackage{subcaption}
\usepackage{siunitx}
\usepackage[title]{appendix}

\usepackage[authormarkup=none]{changes}
\definechangesauthor[name=reviewer0, color=red]{reviewer0}
\definechangesauthor[name=reviewer1, color=blue]{reviewer1}
\definechangesauthor[name=reviewer2, color=brown]{reviewer2}
\newtheorem{thm}{Theorem}[section]

\newtheorem{rem}{Remark}[section]

\newtheorem{exmp}{Example}

\newtheorem{coro}{Corollary}[section]
\numberwithin{equation}{section}
\numberwithin{figure}{section}

\newcommand{\jl}{[\![}
\newcommand{\jr}{]\!]}

\newcommand{\dd}{\mathrm{d}}
\newcommand{\ii}{\mathrm{i}}
\newcommand{\quand}{\quad \text{and} \quad}
\title{Stability and Time-Step Constraints of Exponential Time Differencing Runge--Kutta Discontinuous Galerkin Methods for Advection-Diffusion Equations}
\author{Ziyao Xu\footnote{Department of Applied and Computational Mathematics and Statistics,
University of Notre Dame, Notre Dame, IN 46556, USA. E-mail: zxu25@nd.edu
}, Zheng Sun\footnote{Department of Mathemtatics, The University of Alabama, Tuscaloosa, AL 35487, USA. E-mail: zsun30@ua.edu. Research of Z. Sun
is partially supported by NSF DMS-2208391.}, 
and Yong-Tao Zhang\footnote{Department of Applied and Computational Mathematics and Statistics,
University of Notre Dame, Notre Dame, IN 46556, USA. E-mail: yzhang10@nd.edu. Research of Y.-T. Zhang is partially supported by Simons Foundation MPS-TSM-00007854.
}}
\date{}

\begin{document}

\maketitle

\textbf{Abstract.} 
In this paper, we investigate the stability and time-step constraints for solving advection-diffusion equations using exponential time differencing (ETD) Runge–Kutta (RK) methods in time and discontinuous Galerkin (DG) methods in space. 
We demonstrate that the resulting fully discrete scheme is stable when the time-step size is upper bounded by a constant. 
More specifically, when central fluxes are used for the advection term, the schemes are stable under the time-step constraint $\tau \leq \tau_0 {d}/{a^2}$, while when upwind fluxes are used, the schemes are stable if $\tau \leq \max\left\{\tau_0 {d}/{a^2}, c_0 h/a\right\}$.
Here, $\tau$ is the time-step size, $h$ is the spatial mesh size, and $a$ and $d$ are constants for the advection and diffusion coefficients, respectively. 
The constant $c_0$ is the CFL constant for the explicit RK method for the purely advection equation, and $\tau_0$ is a constant that depends on the order of the ETD-RK method. 
These stability conditions are consistent with those of the implicit-explicit RKDG method. 
The time-step constraints are rigorously proved for the lowest-order case and are validated through Fourier analysis for higher-order cases. 
Notably, the constant $\tau_0$ in the fully discrete ETD-RKDG schemes appears to be determined by the stability condition of their semidiscrete (continuous in space, discrete in time) ETD-RK counterparts and is insensitive to the polynomial degree and the specific choice of the DG method.
Numerical examples, including problems with nonlinear convection in one and two dimensions, are provided to validate our findings.

\bigskip
\bigskip

\noindent{\bf Key Words:}
Exponential time differencing Runge--Kutta methods;
Discontinuous Galerkin methods;
Stability;
Advection-diffusion equations

\section{Introduction}

In this paper, we study the stability and time-step constraints of discontinuous Galerkin (DG) methods combined with exponential time differencing (ETD) Runge--Kutta (RK) methods for the linear advection-diffusion equation:
\begin{equation}\label{eq:AdvDiffEq}
u_t+au_x=du_{xx}
\end{equation}
with the periodic boundary condition.
Here, $a, d$ are constants representing advection and diffusion coefficients, respectively. We require $d > 0$ and, without loss of generality, assume $a>0$. %The uniform mesh grids are assumed in the analysis. In the lowest-order DG schemes, where $\mathbb{P}^0$ elements are employed, the spatial discretization reduces to commonly used finite difference methods. 

The DG method is a class of finite element methods that employs discontinuous piecewise polynomial spaces. It was first proposed by Reed and Hill in the 1970s for solving steady-state transport equations \cite{reed1973triangular} and was later extended to other types of equations, including hyperbolic conservation laws \cite{cockburn2001runge}, elliptic equations \cite{arnold2002unified}, and equations involving higher-order derivatives \cite{cheng2008discontinuous,xu2010local}. The DG method offers several advantages, such as high-order accuracy, preservation of local conservation, flexibility in handling complex geometries, ease of implementing $h$-$p$ adaptivity, and high parallel efficiency. These features have made the method widely applicable across various fields \cite{cockburn2012discontinuous}. There are many different versions of the DG method, varying in formulation and choice of numerical flux. For the advection term, we consider DG methods with upwind numerical flux and central flux in this paper. For the diffusion term, we use the local DG (LDG) method \cite{shu2009discontinuous} and the interior penalty DG (IPDG) method for discretization \cite{riviere2008discontinuous}.

When solving time-dependent and convection-dominated problems, the DG method is typically coupled with explicit RK methods for time marching. For purely convection problems, the resulting fully discrete schemes are generally stable under the standard Courant--Friedrichs--Lewy (CFL) condition \(\tau \leq C h / a\) \cite{cockburn2001runge,sun2019strong,xu20192}, where $\tau$ is the time-step size, $h$ is the spatial mesh size, and $a$ is the wave speed. However, for problems with strong second-order diffusion terms, explicit RK methods suffer from a restrictive time-step constraint $\tau \leq C h^2 / d$ and are therefore less favorable. To circumvent this time-step restriction while avoiding a fully nonlinear implicit scheme, a popular approach is to use implicit-explicit (IMEX) RK methods \cite{ascher1997implicit,kennedy2003additive} for time marching, treating the convection term explicitly and the diffusion term implicitly. 
It has been proven that, for advection-diffusion equations, the fully discrete method is stable when the time-step size is bounded above by a constant $\tau\leq C{d}/{a^2}$
\cite{wang2015stability,wang2023uniform}. Here, \(C\) is a constant independent of \(a\), \(d\), and \(h\), but it may depend on the specific IMEX method, the polynomial degree \(k\), and the mesh regularity parameter, etc. 
We refer to \cite{tan2021stability} for similar discussions involving finite difference spatial discretization and to \cite{hunter2024stability,wang2024analysis} for studies on IMEX methods applied to advection-dispersion equations.

An alternative to IMEX time-marching methods is the use of exponential integrators. These methods employ an exponential integrating factor to handle the dominant linear stiff term and remove the severe time-step size restriction. We refer to \cite{hochbruck2010exponential} for a general review, to \cite{luan2013exponential,luan2014explicit,isherwood2018strong,huang2018bound,liao2024average} for several later works, and to \cite{jiang2016krylov,lu2016krylov,xu2024high} for applications to advection-diffusion equations. In this paper, we are particularly interested in the ETD methods and especially the ETD-RK methods. The ETD methods were first introduced for computational electrodynamics and were then systematically developed in \cite{Beylk,cox2002exponential}, with their stability analyzed for ordinary differential equations \cite{cox2002exponential}, diffusion-reaction equations \cite{du2004stability}, and gradient flow problems \cite{fu2022energy,cao2024exponential}. These methods offer advantages such as relatively small numerical errors, good steady-state preservation properties, and the ability to preserve the maximum principle, among others. Due to these advantages, ETD methods have been widely applied in various fields, particularly in phase field models and gradient flow problems, for preserving the energy decay law \cite{ju2018energy,fu2022energy,fu2024higher,cao2024exponential} and the maximum bound principle \cite{du2019maximum,du2021maximum}.

In this paper, we use the Fourier method, also known as the von Neumann analysis, to analyze the ETD-RKDG methods for the advection-diffusion equation \eqref{eq:AdvDiffEq}. 
We find that a similar time-step constraint applies to the ETD-RKDG method as to the IMEX-RKDG method. In particular, using central flux for the advection term, we find that the ETD-RKDG methods are stable under the time-step constraint
\begin{equation}\label{eq:CFLcentral}
\tau \leq \tau_0 \frac{d}{a^2},
\end{equation}
where \(\tau_0 > 0\) is a constant that depends on the order of the {ETD-RK} temporal discretization. 
The values of \(\tau_0\) for the ETD-RK1 to ETD-RK4 methods are provided either analytically or numerically (accurate to two decimal places) in Table \ref{tab:tau_0}. 
If we instead choose upwind flux for the advection term, the usual CFL conditions for advection equations also apply:
\begin{equation}\label{eq:CFLupwind}
\tau \leq \max\left\{\tau_0 \frac{d}{a^2}, c_0 \frac{h}{a}\right\},
\end{equation}
where $c_0 > 0$ is the standard CFL constant for explicit RKDG methods applied to the purely advection equation \cite{cockburn2001runge}. The time-step constraints \eqref{eq:CFLcentral} and \eqref{eq:CFLupwind} are rigorously proved for the lowest-order case and are verified numerically for higher-order cases with the Fourier method. 
More importantly, for the methods tested in this paper, \emph{the constant \(\tau_0\) appears to be determined by the semidiscrete (continuous in space, discrete in time) ETD-RK formulations discussed in Section \ref{sec:semi-ETD} and is therefore insensitive to the DG spatial discretization.}  
Similar behavior is observed for the time-step constraints of the IMEX-RK methods, as detailed in Appendix \ref{app:IMEX}. 
Notably, the insensitivity of \(\tau_0\) to the polynomial degree stands in sharp contrast to explicit RKDG methods for purely advection equations, where the CFL constant \(c_0\) decreases significantly with increasing polynomial order \cite{cockburn2001runge}.

The rest of the paper is organized as follows.
We begin by presenting the preliminaries on the DG methods, the ETD-RK methods, and the dimensionless form of \eqref{eq:AdvDiffEq} in Section \ref{sec:Preliminaries}.
In Section \ref{sec:LowestOrder}, we use the Fourier method to analyze the stability of the lowest-order discretization for the advection-diffusion equation. These results provide a prototype for the time-step constraints needed for stability.
In Section \ref{sec:HigherOrder}, we study the stability of high-order ETD-RK time-stepping methods combined with high-order DG discretizations.
Numerical experiments are presented in Section \ref{sec:Numerics} to validate the analysis and highlight the stability advantages of ETD-RK methods in time marching.
Finally, we conclude with closing remarks in Section \ref{sec:Remarks}.
Appendix A presents a similar discussion on IMEX-RKDG methods.

\section{Preliminaries}\label{sec:Preliminaries}

In this section, we first provide a brief review of the DG and ETD-RK methods. Then we derive the dimensionless form of the linear advection-diffusion equation, which would simplify the analyses on stability and time-step constraints in later sections.  
%We begin with an introduction to various types of DG methods, including the central and upwind fluxes for the advection term, and the local discontinuous Galerkin (LDG) and interior penalty discontinuous Galerkin (IPDG) methods for the discretization of the diffusion term.
%We then proceed to discuss ETD methods, which are specifically designed to address the stiffness arising from the discretization of diffusion term by integrating it exactly using exponential integrators.

\subsection{DG methods}\label{sec:Preliminaries:DG}

Consider the computational domain $\Omega=[0,2\pi]$ with partition $0=x_{\frac12}<x_{\frac32}<\cdots<x_{N+\frac12}=2\pi$.
We denote by $I_j=(x_{j-\frac12},x_{j+\frac12})$ the $j$th cell, with the cell center $x_j=\frac12(x_{j-\frac12}+x_{j+\frac12})$ and the cell length $\Delta x_{j}=x_{j+\frac12}-x_{j-\frac12}$.
If the grid is uniform, which is assumed throughout the analysis, we have $x_{j+\frac12}=jh$ for $j=0,1,\ldots, N$, where $h=\frac{2\pi}{N}$ is the grid size.

The DG space $\mathbb{V}_h^k$ is define on the grid as follows:
\begin{equation*}
\mathbb{V}_{h}^{k}=\{v_h\in L^2(\Omega):v_h|_{I_j}\in\mathbb{P}^k(I_j), 1\leq j\leq N\},
\end{equation*}
where $\mathbb{P}^k(I_j)$ is the space of polynomials of degree at most $k$ on the interval $I_j$.
To facilitate the DG formulation, we adopt the notations \cite{riviere2008discontinuous}
\begin{align*}
    \{v_h\}_{j+\frac12}=\frac 12 (v_{h,j+\frac12}^{-}+v_{h,j+\frac12}^{+}), \quad \jl v_h\jr_{j+\frac12}=v_{h,j+\frac12}^{-}-v_{h,j+\frac12}^{+}, \quad \text{ for } v_h\in \mathbb{V}_{h}^{k},
\end{align*}
where $v_{h,j+\frac12}^{\pm}=\lim_{\gamma\rightarrow 0^+}v_h(x_{j+\frac12}\pm\gamma)$ represents the left or right limit of $v_h$ at $x_{j+\frac12}$.

\begin{itemize}
\item For the advection equation:
\begin{equation*}
u_t+au_x=0, \quad x\in\Omega,
\end{equation*}
the semidiscrete DG formulation is given as follows \cite{shu2009discontinuous}: Find $u_h(t)\in \mathbb{V}_{h}^{k}$, such that,
\begin{equation}\label{eq:DG}
\int_{I_j}(u_h)_t v_h\dd x-\int_{I_j}au_h(v_h)_x\dd x+a\hat{u}_{j+\frac12}v_{h,j+\frac12}^{-}-a\hat{u}_{j-\frac12}v_{h,j-\frac12}^{+}=0\quad \forall v_h\in\mathbb{P}^k(I_j),
\end{equation}
for $j=1,2,\ldots,N$, where $\hat{u}_{j+\frac12}$ is the numerical flux defined on cell interfaces.
In this paper, we consider both the central flux
\begin{equation}\label{eq:CentralFlux}
\hat{u}_{j+\frac12}=\{u_h\}_{j+\frac12}
\end{equation}
and the upwind flux
\begin{equation}\label{eq:UpwindFlux}
\hat{u}_{j+\frac12}=u_{h,j+\frac12}^{-}.
\end{equation}

We denote the resulting linear system by $\mathbf{u}_t+A\mathbf{u}=\mathbf{0}$, where $\mathbf{u}$ is the vector of degrees of freedom (DoFs) of the numerical solution $u_h$, which will be specified in later sections.

\item For the diffusion equation:
\begin{equation*}
u_t=du_{xx},\quad x\in\Omega,
\end{equation*}
the semidiscrete LDG method is formulated as follows \cite{shu2009discontinuous}: Find $u_h, p_h\in \mathbb{V}_{h}^{k}$, such that,
\begin{subequations}\label{eq:LDG}
\begin{align}
\int_{I_j}(u_h)_tv_h\dd x=&\sqrt{d}\left(-\int_{I_j}p_h(v_h)_x\dd x+\hat{p}_{j+\frac12}v_{h,j+\frac12}^{-}-\hat{p}_{j-\frac12}v_{h,j-\frac12}^{+}\right)\quad \forall v_h\in\mathbb{P}^{k}(I_j),\\
\int_{I_j}p_hq_h\dd x=&\sqrt{d}\left(-\int_{I_j}u_h(q_h)_x\dd x+\check{u}_{j+\frac12}q_{h,j+\frac12}^{-} -\check{u}_{j-\frac12}q_{h,j-\frac12}^{+}\right)\quad \forall q_h\in\mathbb{P}^{k}(I_j),
\end{align}
\end{subequations}
for $j=1,2\ldots,N$, where $\check{u}_{j+\frac12}$ and $\hat{p}_{j+\frac12}$ are numerical fluxes defined on cell interfaces. 
For example, the alternating flux is given by
\begin{equation}\label{eq:LDGAltFlux}
\check{u}_{j+\frac12}=u_{h,j+\frac12}^{-},\quad \hat{p}_{j+\frac12}=p_{h,j+\frac12}^{+},
\end{equation}
and the central flux is defined as
\begin{equation*}
\check{u}_{j+\frac12}=\{u_h\}_{j+\frac12},\quad\hat{p}_{j+\frac12}=\{p_h\}_{j+\frac12}.
\end{equation*}
The semidiscrete IPDG methods is formulated as follows: Find $u_h\in \mathbb{V}_{h}^{k}$, such that,
\begin{equation}\label{eq:IPDG}
\begin{split}
\sum_{j=1}^{N}\int_{I_j}(u_h)_t v_h\dd x = & -\sum_{j=1}^{N}\int_{I_j} d (u_h)_x (v_h)_x \dd x+\sum_{j=1}^{N}d\{(u_h)_x\}_{j+\frac12}\jl v_h\jr_{j+\frac12}\\
&+\epsilon\sum_{j=1}^{N}d \jl u_h\jr_{j+\frac12}\{(v_h)_x\}_{j+\frac12} - \sum_{j=1}^{N}\frac{\sigma}{h}\jl u_h\jr_{j+\frac12}\jl v_h\jr_{j+\frac12} \ \forall v_h\in \mathbb{V}_{h}^{k}, 
\end{split}
\end{equation}
where $\epsilon=1,-1,0$ corresponds to the symmetric interior penalty Galerkin (SIPG), nonsymmetric interior penalty Galerkin (NIPG), and incomplete interior penalty Galerkin (IIPG) methods, respectively, and $\sigma>0$ is the penalty parameter needed in SIPG and IIPG \cite{riviere2008discontinuous}.

We denote the resulting linear system by $\mathbf{u}_t=D\mathbf{u}$, where $\mathbf{u}$ is the vector of DoFs of the numerical solution $u_h$. 
Note that lifting techniques are required in the LDG method to eliminate the DoFs of the auxiliary variable $p$ \cite{arnold2002unified}.

\item For the advection-diffusion equation \eqref{eq:AdvDiffEq}, we combine the discretization for the advection and diffusion terms to obtain the linear system:
\begin{equation*}
\mathbf{u}_t+A\mathbf{u}=D\mathbf{u}.
\end{equation*}

\end{itemize}

\subsection{ETD-RK methods}\label{sec:Preliminaries:ETD}
We consider the ODE system in the following form: 
\begin{equation}\label{eq:ODE}
\mathbf{u}_t=D\mathbf{u}+F(\mathbf{u}).
\end{equation}
Here $D$ represents the linear stiff term from the discretization of the diffusion term. $F$ is the non-stiff term from the discretization of the advection term, which, for the problem \eqref{eq:AdvDiffEq} we consider, is given by $F(\mathbf{u})=-A\mathbf{u}$.
Since we also test nonlinear convection-diffusion equations in the numerical section, we use $F(\mathbf{u})$ for generality.

Multiplying \eqref{eq:ODE} with the integrating factor $e^{-Dt}$ to absorb the stiff term, and then integrating over time from $t^n$ to $t^{n+1}$, we obtain its Volterra integral equation
\begin{equation}\label{eq:ETD_ODEs2}
\mathbf{u}(t^{n+1})=e^{\tau D}\mathbf{u}(t^n)+\int_{0}^{\tau}e^{(\tau-s)D} F(\mathbf{u}(t^n+s)) \dd s,
\end{equation}
where $\tau=t^{n+1}-t^{n}$ is the time-step size. This representation of the exact solution is also called
\emph{variation-of-constants formula} \cite{hochbruck2010exponential}.

The first-order ETD method (ETD-RK1) is obtained by the approximation $F(\mathbf{u}(t^n+s))\approx F(\mathbf{u}^n)$ in the integrand of \eqref{eq:ETD_ODEs2}: 
\begin{equation}\label{eq:ETDRK1}
\begin{split}
\mathbf{u}^{n+1}=&e^{\tau D}\mathbf{u}^{n} + D^{-1}\left(e^{\tau D}-I\right)F(\mathbf{u}^n),
\end{split}
\end{equation}
where $\mathbf{u}^n$ (or $\mathbf{u}^{n+1}$) is the approximation of the solution $\mathbf{u}(t^n)$ (or $\mathbf{u}(t^{n+1})$).

Following a similar approach to traditional RK methods, the higher-order ETD-RK methods are formulated as follows \cite{cox2002exponential}: 
\begin{itemize}
   \item ETD-RK2 (second-order):
   \begin{equation}\label{eq:ETDRK2}
   \begin{split}
       \mathbf{a}^n=&\mathbf{u}^n+\tau \varphi_1 (\tau D) (D\mathbf{u}^{n}+F(\mathbf{u}^n)),\\
       \mathbf{u}^{n+1}=&\mathbf{a}^n+\tau\varphi_{2}(\tau D)(-F(\mathbf{u}^n)+F(\mathbf{a}^n)),
   \end{split}
   \end{equation}
    \item ETD-RK3 (third-order):
    \begin{equation}\label{eq:ETDRK3}
    \begin{split}
        \mathbf{a}^n=&\mathbf{u}^n+\frac{\tau}{2}\varphi_1(\frac{\tau}{2}D)(D\mathbf{u}^n+F(\mathbf{u}^n)),\\
        \mathbf{b}^n=&\mathbf{u}^n+\tau\varphi_{1}(\tau D)(D\mathbf{u}^n-F(\mathbf{u}^n)+2F(\mathbf{a}^n)),\\
        \mathbf{u}^{n+1}=&\mathbf{u}^n+\tau\varphi_{1}(\tau D)(D\mathbf{u}^n+F(\mathbf{u}^n))\\
        &+\tau\varphi_{2}(\tau D)(-3F(\mathbf{u}^n)+4F(\mathbf{a}^n)-F(\mathbf{b}^n))\\
        &+\tau\varphi_{3}(\tau D)(4F(\mathbf{u}^n)-8F(\mathbf{a}^n)+4F(\mathbf{b}^n)),
    \end{split}
    \end{equation}
    \item ETD-RK4 (fourth-order):
    \begin{equation}\label{eq:ETDRK4}
    \begin{split}
        \mathbf{a}^n=&\mathbf{u}^n+\frac{\tau}{2}\varphi_{1}(\frac{\tau}{2}D)(D\mathbf{u}^n+F(\mathbf{u}^n)),\\
        \mathbf{b}^n=&\mathbf{u}^{n}+\frac{\tau}{2}\varphi_{1}(\frac{\tau}{2} D)(D\mathbf{u}^n+F(\mathbf{a}^n)),\\
        \mathbf{c}^n=&\mathbf{a}^n+\frac{\tau}{2}\varphi_{1}(\frac{\tau}{2} D)(D\mathbf{a}^n-F(\mathbf{u}^n)+2F(\mathbf{b}^n)),\\
        \mathbf{u}^{n+1}=&\mathbf{u}^n+\tau\varphi_{1}(\tau D)(D\mathbf{u}^n+F(\mathbf{u}^n))\\
        &+\tau\varphi_{2}(\tau D)(-3F(\mathbf{u}^{n})+2F(\mathbf{a}^n)+2F(\mathbf{b}^n)-F(\mathbf{c}^n))\\
        &+\tau\varphi_{3}(\tau D)(4F(\mathbf{u}^n)-4F(\mathbf{a}^n)-4F(\mathbf{b}^n)+4F(\mathbf{c}^n)),
    \end{split}
    \end{equation}
\end{itemize}
where the $\varphi$-functions are defined as \cite{hochbruck2010exponential}:
\begin{equation*}
\varphi_1(z)=\frac{e^z-1}{z},\quad \varphi_2(z)=\frac{e^z-1-z}{z^2},\quad \varphi_3(z)=\frac{e^{z}-1-z-\frac12 z^2}{z^3}.
\end{equation*}

\subsection{Dimensionless form}\label{sec:dimless}
The nondimensionalized analysis can significantly simplify the study of time-step constraints in later sections.
By the change of variables
\begin{equation*}
t'=\frac{a^2}{d}t\quand
x'=\frac{a}{d}x,
\end{equation*}
the advection-diffusion equation \eqref{eq:AdvDiffEq} can be transformed into its dimensionless form:
\begin{equation}\label{eq:AdvDiffEq_Dimless}
u_{t'}+u_{x'}=u_{x'x'}.
\end{equation}
If a numerical scheme for \eqref{eq:AdvDiffEq_Dimless} is stable under the time-step constraints
\begin{equation*}
\tau'\leq\tau_0\quad \text{or} \quad \tau'\leq c_0 h',
\end{equation*}
where $\tau'$ and $h'$ are dimensionless temporal and spatial grid sizes, then the scheme for the original equation \eqref{eq:AdvDiffEq} is stable under the constraints 
\begin{equation*}
\tau\leq\tau_0\frac{d}{a^2} \quad\text{or}\quad \tau\leq c_0\frac{h}{a},
\end{equation*}
respectively.
For simplicity, we will conduct all analysis using the dimensionless form \eqref{eq:AdvDiffEq_Dimless} and obtain the time-step constraints for the physical variables through the variable transformations.
However, the analysis can also be applied directly to the general coefficient equation \eqref{eq:AdvDiffEq}.
For conciseness, we will omit the prime notation from the dimensionless variables in the analysis.

\section{Stability of the lowest-order case}\label{sec:LowestOrder}
In this section, we assume uniform mesh partitions, which allows us to use the Fourier method for stability analysis. The lowest-order ($\mathbb{P}^{0}$ elements) DG discretization coincides with commonly used finite difference schemes. 
With the choice of central flux \eqref{eq:CentralFlux} in the DG method \eqref{eq:DG}, and alternating flux \eqref{eq:LDGAltFlux} in the LDG method \eqref{eq:LDG} or $\sigma=d$ in the IPDG method \eqref{eq:IPDG}, the semidiscrete scheme for \eqref{eq:AdvDiffEq_Dimless} is given by
\begin{equation}\label{eq:FD_central}
\frac{\dd u_{j}}{\dd t} + \frac{u_{j+1}-u_{j-1}}{2h} = \frac{u_{j+1}-2u_{j}+u_{j-1}}{h^2}.
\end{equation}
Alternatively, if we choose upwind flux \eqref{eq:UpwindFlux} in the DG method \eqref{eq:DG}, the semidiscrete scheme for \eqref{eq:AdvDiffEq_Dimless} becomes
\begin{equation}\label{eq:FD_upwind}
\frac{\dd u_j}{\dd t}+\frac{u_{j}-u_{j-1}}{h}=\frac{u_{j+1}-2u_j+u_{j-1}}{h^2}.
\end{equation}
In this section, we analyze the time-step constraints of \eqref{eq:FD_central} and \eqref{eq:FD_upwind} when discretized with the lowest-order ETD-RK method \eqref{eq:ETDRK1}.
The results provide a prototype for the time-step constraints needed for stability of ETD-RKDG methods.

\subsection{Central scheme with ETD-RK1 discretization}\label{sec:LowestOrder_central}

It is well-known that when discretized with the forward Euler method, the central difference scheme for the advection equation is unstable under the usual CFL conditions.
However, we shall show that the ETD-RK1 method for \eqref{eq:FD_central}, which treats the advection term explicitly, is stable under the time-step constraint $\tau\leq2$.

Let $\mathbf{u}^m = \begin{bmatrix}u_1^m, u_2^m, \cdots, u_{N}^m\end{bmatrix}^T$, $m = n,n+1$. 
Applying \eqref{eq:ETDRK1} for \eqref{eq:FD_central}, we obtain
\begin{equation}\label{eq:ETDRK1-FD}
\mathbf{u}^{n+1}=e^{\tau D}\mathbf{u}^n-D^{-1}\left(e^{\tau D}-I\right)A\mathbf{u}^n,
\end{equation}
where
\begin{equation}\label{eq:ETD-CentralFD-D}
D=\frac{1}{h^2}\begin{bmatrix}
-2 & 1 & 0 & 0 & \cdots & 0 & 1\\
1 & -2 & 1 & 0 & \cdots & 0 & 0\\
0 & 1 & -2 & 1 & \cdots & 0 & 0 \\
& & & \ddots & & & \\
1 & 0 & 0 & 0 & \cdots & 1 & -2\\
\end{bmatrix}
\end{equation}
and 
\begin{equation}\label{eq:ETD-CentralFD-A}
A=\frac{1}{2h}\begin{bmatrix}
0 & 1 & 0 & 0 & \cdots & 0 & -1\\
-1 & 0 & 1 & 0 & \cdots & 0 & 0\\
0 & -1 & 0 & 1 & \cdots & 0 & 0 \\
& & & \ddots & & & \\
1 & 0 & 0 & 0 & \cdots & -1 & 0\\
\end{bmatrix}.
\end{equation}

We analyze the stability of the scheme \eqref{eq:ETDRK1-FD}-\eqref{eq:ETD-CentralFD-A} using the Fourier method.
Consider a Fourier mode $u^{m}_{j}=\widehat{u}^m e^{\ii \omega jh}$, where $\ii=\sqrt{-1}$, $j=1,\ldots,N$, $m = n,n+1$, and $\omega=-\frac{N}{2}+1,\ldots,\frac{N}{2}$. In this case, we have 
\begin{equation}\label{eq:Fansatz}
    \mathbf{u}^m = \widehat{u}^m \begin{bmatrix}e^{\ii \omega h}, e^{\ii 2\omega h},\cdots,e^{\ii N\omega h}\end{bmatrix}^T,\quad  m = n, n+1.
\end{equation}
It is straightforward to verify that
\begin{equation*}
D\mathbf{u}^n=-\frac{4}{h^2}\sin^2(\frac{\omega h}{2})\mathbf{u}^n \quand A\mathbf{u}^n=\frac{2\ii}{h}\sin(\frac{\omega h}{2})\cos(\frac{\omega h}{2})\mathbf{u}^n,
\end{equation*}
meaning that $\mathbf{u}^n$ is an eigenvector of both $A$ and $D$. Recall that $D\mathbf{u}^n = \lambda \mathbf{u}^n$ implies $e^{\tau D}\mathbf{u}^n = e^{\tau \lambda}\mathbf{u}^n$ and $D^{-1}(e^{\tau D}-I)\mathbf{u}^n = \lambda^{-1}(e^{\tau \lambda}-1)\mathbf{u}^n$. 
Therefore, the scheme \eqref{eq:ETDRK1-FD} with the Fourier ansatz \eqref{eq:Fansatz} can be written as: 
\begin{equation}\label{eq:G_factor}
\mathbf{u}^{n+1}=\widehat{G}(\tau,h,\omega)\mathbf{u}^n.
\end{equation}
Here 
\begin{equation*}
\widehat{G}(\tau,h,\omega)=e^{-\frac{4\tau}{h^2}\sin^2(\frac{\omega h}{2})}+\ii\frac{h}{2}\cot(\frac{\omega h}{2})\left(e^{-\frac{4\tau}{h^2}\sin^2(\frac{\omega h}{2})}-1\right)
\end{equation*}
is the growth factor of the Fourier mode. 
Equivalently, we have
\begin{equation*}
\widehat{u}^{n+1}=\widehat{G}(\tau,h,\omega)\widehat{u}^n. 
\end{equation*}
As a sufficient condition, the scheme is stable if $|\widehat{G}(\tau,h,\omega)|\leq 1$ for all $\omega$.

The stability and time-step constraint are stated in the following theorem. The algebra in the proof is tedious but straightforward, and can easily be verified using symbolic computation software. Here, we present only the key computational results in the proof, omitting the intermediate steps.
\begin{thm}\label{thm:1}
The scheme \eqref{eq:ETDRK1-FD}-\eqref{eq:ETD-CentralFD-A} is stable under the time-step constraint $\tau\leq2$, with the growth factor $|\widehat{G}(\tau,h,\omega)|\leq 1$. 
\end{thm}
\begin{proof}
One can calculate that 
\begin{equation*}
\begin{split}
|\widehat{G}(\tau,h,\omega)|^2&=e^{-\frac{8\tau\eta}{h^2}}+\frac{h^2(1-\eta)}{4\eta}\left(e^{-\frac{4\tau\eta}{h^2}}-1\right)^2\\
&=:Q(\tau,h,\eta),\quad \text{where } \eta=\sin^2(\frac{\omega h}{2})\in[0,1].    
\end{split}
\end{equation*}
Note we have $Q(\tau,h,0) = \lim_{\eta\to 0} Q(\tau, h,\eta) = 1$. With direct computation, one can get
\begin{equation*}
\begin{split}
\frac{\partial Q(\tau,h,\eta)}{\partial \eta}&=-\frac{e^{\frac{-8\tau\eta}{h^2}}}{4h^2\eta^2}\left( \left( h^2(-1 + e^{\frac{4 \tau\eta}{h^2}}) - 4\tau\eta(1-\eta) \right)^2 + 16\tau\eta^2\left(2-\tau(1-\eta)^2\right)\right)\\
&\leq 0\quad\forall \eta\in[0,1], h>0, \tau \leq 2.
\end{split}
\end{equation*}
Therefore, we have $Q(\tau,h,\eta)\leq Q(\tau,h,0)=1$, which completes the proof.
\end{proof}
On the other hand, since
\begin{equation*}
\left.\frac{\partial  Q(\tau,h,\eta)}{\partial \eta}\right|_{\eta=0}=\frac{4\tau(\tau-2)}{h^2}>0, \text{ for } \tau>2,
\end{equation*}
we have $Q(\tau,h,\eta_0)>Q(\tau,h,0)=1$ for some $\eta_0\in(0,1]$. 
Thus, the condition $\tau\leq2$ is indeed necessary for the stability condition $|\widehat{G}(\tau,h,\omega)|\leq 1$, which will also be verified by numerical tests.

\medskip

For the general coefficient problem \eqref{eq:AdvDiffEq}, the time-step constraint for stability is given by $\tau \leq 2\frac{d}{a^2}$, derived through a change of variables as described in Section \ref{sec:dimless}.
\begin{coro}
    The central scheme  for \eqref{eq:AdvDiffEq},    \begin{equation*}
\frac{\dd u_{j}}{\dd t} + a\frac{u_{j+1}-u_{j-1}}{2h} = d\frac{u_{j+1}-2u_{j}+u_{j-1}}{h^2},
\end{equation*}
with ETD-RK1 time discretization is stable under the time-step constraint $\tau\leq2\frac{d}{a^2}$.
\end{coro}
\subsection{Upwind scheme with ETD-RK1 discretization}
It is known that the upwind discretization for advection equations has better stability than the central scheme. Therefore, we expect that the time-step constraint obtained for \eqref{eq:FD_central} is sufficient to ensure the stability of \eqref{eq:FD_upwind}. 
Moreover, since the diffusion term is integrated exactly, we also expect that the usual CFL condition for pure advection equations will satisfy the stability requirement.
In summary, taking all these factors into account, we shall demonstrate that the ETD-RK1 method for \eqref{eq:FD_upwind} is stable under the time-step constraint $\tau\leq\max\{2, h\}$.

Applying \eqref{eq:ETDRK1} to \eqref{eq:FD_upwind}, we obtain the scheme \eqref{eq:ETDRK1-FD} with the same definition of $D$ as in \eqref{eq:ETD-CentralFD-D}, and a different $A$ defined as follows,
\begin{equation}\label{eq:ETD-UpwindFD-A}
A=\frac{1}{h}\begin{bmatrix}
1 & 0 & 0 & 0 & \cdots & 0 & -1\\
-1 & 1 & 0 & 0 & \cdots & 0 & 0\\
0 & -1 & 1 & 0 & \cdots & 0 & 0 \\
& & & \ddots & & & \\
0 & 0 & 0 & 0 & \cdots & -1 & 1\\
\end{bmatrix}.
\end{equation}
For the Fourier mode \eqref{eq:Fansatz} with $u^{n}_{j}=\widehat{u}^n e^{\ii \omega jh}$, it is straightforward to verify that
\begin{equation*}
D\mathbf{u}^n=-\frac{4}{h^2}\sin^2(\frac{\omega h}{2})\mathbf{u}^n\quand A\mathbf{u}^n=\left(\frac{2\ii}{h}\sin(\frac{\omega h}{2})\cos(\frac{\omega h}{2})+\frac2h\sin^2(\frac{\omega h}{2})\right)\mathbf{u}^{n}.
\end{equation*}
The growth factor in \eqref{eq:G_factor} is then given by
\begin{equation*}
\begin{split}
\widehat{G}(\tau,h,\omega)=&e^{-\frac{4\tau}{h^2}\sin^2(\frac{\omega h}{2})}+\frac{h}{2}\left(e^{-\frac{4\tau}{h^2}\sin^2(\frac{\omega h}{2})}-1\right)+\ii\frac{h}{2}\cot{(\frac{\omega h}{2})}\left(e^{-\frac{4\tau}{h^2}\sin^2(\frac{\omega h}{2})}-1\right).
\end{split}
\end{equation*}

The stability and time-step constraint for \eqref{eq:FD_upwind} are stated in the following theorem. Again, we omit some of the algebra and present only the key computational results in the proof.
\begin{thm}\label{thm:2}
The scheme \eqref{eq:ETDRK1-FD}, \eqref{eq:ETD-CentralFD-D} and \eqref{eq:ETD-UpwindFD-A} is stable  under the time-step constraint $\tau\leq\max\{2,h\}$, with the growth factor $|\widehat{G}(\tau,h,\omega)|\leq 1$. 
\end{thm}
\begin{proof}
One can calculate that
\begin{equation}\label{eq:upwindG}
\begin{split}
|\widehat{G}(\tau,h,\omega)|^2&=\left(e^{-\frac{4\tau\eta}{h^2}}+\frac{h}{2}(e^{-\frac{4\tau\eta}{h^2}}-1)\right)^2 + \frac{h^2(1-\eta)}{4\eta}\left(e^{-\frac{4\tau\eta}{h^2}}-1\right)^2\\
&=:Q(\tau,h,\eta),\quad \text{where } \eta=\sin^2(\frac{\omega h}{2})\in[0,1].
\end{split}
\end{equation}

First, we show that $Q(\tau,h,\eta)\leq1$ under the time-step constraint $\tau\leq2$. Note that 
$e^{-\frac{4\tau\eta}{h^2}}>0$ and $\frac{h}{2}(e^{-\frac{4\tau\eta}{h^2}}-1)<0$. Hence the first term in \eqref{eq:upwindG} has the estimate
\begin{equation*}
    \left(e^{-\frac{4\tau\eta}{h^2}}+\frac{h}{2}(e^{-\frac{4\tau\eta}{h^2}}-1)\right)^2 \leq \max\left\{e^{-\frac{8\tau\eta}{h^2}},\frac{h^2}{4}(e^{-\frac{4\tau\eta}{h^2}}-1)^2\right\}.
\end{equation*}
As a result, by defining the auxiliary functions
\begin{equation*}
Q_1(\tau,h,\eta) = e^{-\frac{8\tau\eta}{h^2}}+\frac{h^2(1-\eta)}{4\eta}\left(e^{-\frac{4\tau\eta}{h^2}}-1\right)^2
\end{equation*}
and
\begin{equation*}
Q_2(\tau,h,\eta) = \frac{h^2}{4\eta}\left(e^{-\frac{4\tau\eta}{h^2}}-1\right)^2,
\end{equation*}
we have 
\begin{equation*}
Q(\tau,h,\eta)\leq\max\{Q_{1}(\tau,h,\eta), Q_2(\tau,h,\eta)\}.
\end{equation*}
It was proved in Theorem \ref{thm:1} that $Q_1(\tau,h,\eta)\leq1$ for $h>0$ and $\eta\in[0,1]$ when $\tau\leq2$. Hence it suffices to show that $Q_2(\tau,h,\eta)\leq 1$ under the same condition. Indeed, we define $g(x)=2(e^{-x}-1)^2-x$.
By computing the derivative $g'(x)=4(1-e^{-x})e^{-x}-1$, we find that $g(x)\leq g(\ln2)=\frac12-\ln2<\frac12-\ln\sqrt{e}=0$ for $x>0$.
Therefore, we have the inequality $\frac{\tau}{x}(e^{-x}-1)^2<1$ for $x>0$ when $\tau\leq 2$.
Substituting in $x=\frac{4\tau\eta}{h^2}$, we obtain $Q_2(\tau,h,\eta)<1$ for $h>0$, $\eta\in[0,1]$ when $\tau\leq2$.
Collecting the above results, we have
\begin{equation*}
Q(\tau,h,\eta)\leq 1, \quad \text{when } \tau\leq2.
\end{equation*}

We then show that $Q(\tau,h,\eta)\leq1$ under the time-step constraint $\tau\leq h$.
We define an auxiliary function 
\begin{equation*}
\overline{Q}(h,\eta,\theta) = \left(e^{-\frac{4\theta\eta}{h}}+\frac{h}{2\theta}(e^{-\frac{4\theta\eta}{h}}-1)\right)^2 + \frac{h^2(1-\eta)}{4\theta^2\eta}\left(e^{-\frac{4\theta\eta}{h}}-1\right)^2,
\end{equation*}
such that 
\begin{equation}\label{eq:Qbar}
\overline{Q}(h,\eta,1)=Q(h,h,\eta) \quand \overline{Q}(h,\eta,0)=\lim_{\theta\to 0}\overline{Q}(h,\eta,\theta)=1.
\end{equation}
By direct calculation, one can get
\begin{equation*}
\begin{split}
&\frac{\partial }{\partial \theta}\overline{Q}(h,\eta,\theta)\\
=&-\frac{h^2}{2\theta^3\eta}e^{-\frac{8\theta\eta}{h}}\left( \left( e^{\frac{4\theta\eta}{h}} - (1+\frac{3\theta\eta}{h}+\frac{4\theta^2\eta^2}{h^2}) \right)^2 - \frac{\theta^2\eta^2}{h^2}\left(1+\frac{4\theta\eta}{h}\right)^2 + \frac{16\theta^3\eta^2}{h^3}(1-\eta) \right)\\
=&-\frac{h^2}{2\theta^3\eta}e^{-\frac{8\theta\eta}{h}}\left( \left( \frac{\theta\eta}{h}\left( 1+\frac{4\theta\eta}{h} \right) +\sum_{m=3}^{\infty}\frac{4^m}{m!}\left(\frac{\theta\eta}{h}\right)^m \right)^2 - \frac{\theta^2\eta^2}{h^2}\left(1+\frac{4\theta\eta}{h}\right)^2 + \frac{16\theta^3\eta^2}{h^3}(1-\eta) \right)\\
\leq&0,\quad \text{ for } \eta\in[0,1], h>0,
\end{split}
\end{equation*}
where in the last inequality we have used $\left( \frac{\theta\eta}{h}\left( 1+\frac{4\theta\eta}{h} \right) +\sum_{m=3}^{\infty}\frac{4^m}{m!}\left(\frac{\theta\eta}{h}\right)^m \right)^2 \geq \frac{\theta^2\eta^2}{h^2}\left(1+\frac{4\theta\eta}{h}\right)^2$ and $\frac{16\theta^3\eta^2}{h^3}(1-\eta) \geq0$ for $\eta\in[0,1]$.
Therefore, using \eqref{eq:Qbar} and the monotonicity of 
$\overline{Q}$ with respect to its last argument, one can obtain
\begin{equation*}
Q(h,h,\eta)=\overline{Q}(h,\eta,1)\leq\overline{Q}(h,\eta,0)=1,\quad \text{ for } \eta\in[0,1], h>0.
\end{equation*}
We then define another auxiliary function 
\begin{equation*}
\begin{split}
\widetilde{Q}(\lambda,h,\eta)&=\left(e^{-\frac{4\eta}{h}}+\frac{\lambda h}{2}(e^{-\frac{4\eta}{h}}-1)\right)^2 + \frac{\lambda^2h^2(1-\eta)}{4\eta}\left(e^{-\frac{4\eta}{h}}-1\right)^2,
\end{split}
\end{equation*}
such that $\widetilde{Q}(1,h,\eta)=Q(h,h,\eta)\leq 1$.
Moreover, it is clear that $\widetilde{Q}(0,h,\eta)=e^{-\frac{8\eta}{h}}\leq 1$.
Since $\widetilde{Q}(\lambda,h,\eta)$ is a concave-up quadratic function with respect to $\lambda$, we have 
\begin{equation*}
\widetilde{Q}(\lambda,h,\eta)\leq \max\{\widetilde{Q}(0,h,\eta),\widetilde{Q}(1,h,\eta)\}\leq 1\quad \forall \eta\in[0,1], h>0, \text{ for } \lambda\in[0,1].
\end{equation*}
Note that the above inequality holds for any positive number in its second argument.  Therefore, if $\tau\leq h$, we have
\begin{equation*}
Q(\tau,h,\eta)=\widetilde{Q}(\frac{\tau}{h},\frac{h^2}{\tau},\eta)\leq \max\{\widetilde{Q}(0,\frac{h^2}{\tau},\eta),\widetilde{Q}(1,\frac{h^2}{\tau},\eta)\}\leq 1,
\end{equation*}
which completes the proof.
\end{proof}

%In the numerical experiment, we observe that the conditions in Theorem \ref{thm:2} are sufficient but not necessary. The actual CFL conditions can be less restrictive than those provided in Theorem \ref{thm:2}.
\medskip

For the general coefficient problem \eqref{eq:AdvDiffEq}, the time-step constraint for stability is given by $\tau \leq \max\left\{2\frac{d}{a^2}, \frac{h}{a}\right\}$, derived through a change of variables as described in Section \ref{sec:dimless}.
\begin{coro}
    The upwind scheme for \eqref{eq:AdvDiffEq},    \begin{equation*}
\frac{\dd u_{j}}{\dd t} + a\frac{u_{j}-u_{j-1}}{h} = d\frac{u_{j+1}-2u_{j}+u_{j-1}}{h^2},
\end{equation*}
with ETD-RK1 time discretization is stable under the time-step constraint $\tau \leq \max\left\{2\frac{d}{a^2}, \frac{h}{a}\right\}$.
\end{coro}

\section{Stability of higher-order cases}\label{sec:HigherOrder}
In this section, we examine the stability and time-step constraints of ETD-RK methods of varying orders combined with higher-order DG methods.
We demonstrate that the constants $\tau_0$ in the time-step constraint $\tau \leq \tau_0 \frac{d}{a^2}$ {for the fully discrete ETD-RKDG schemes coincide with} those of the continuous-in-space, semidiscrete ETD-RK schemes, and provide the specific values of $\tau_0$ to two decimal places for different ETD-RK methods. 

\subsection{Semidiscrete ETD-RK schemes}\label{sec:semi-ETD}
In this subsection, we study the stability of semidiscrete ETD-RK schemes of varying orders. These semidiscrete schemes are continuous in space and discrete in time, as described below:
\begin{itemize}
   \item ETD-RK1:
   \begin{equation}\label{eq:ETDRK1_semi}
   \begin{split}
    u^{n+1}=e^{\tau \partial_{xx}}u^n-\partial_{xx}^{-1}\left(e^{\tau \partial_{xx}}-1\right)\partial_{x}u^n,
   \end{split}
   \end{equation}
   \item ETD-RK2:
   \begin{equation}\label{eq:ETDRK2_semi}
   \begin{split}
       a^n=&u^n+\tau \varphi_1 (\tau \partial_{xx}) (\partial_{xx}u^{n}-\partial_{x}u^n),\\
       u^{n+1}=&a^n+\tau\varphi_{2}(\tau \partial_{xx})(\partial_{x}u^n-\partial_{x}a^n),
   \end{split}
   \end{equation}
    \item ETD-RK3:
    \begin{equation}\label{eq:ETDRK3_semi}
    \begin{split}
        a^n=&u^n+\frac{\tau}{2}\varphi_1(\frac{\tau}{2}\partial_{xx})(\partial_{xx}u^n-\partial_{x}u^n),\\
        b^n=&u^n+\tau\varphi_{1}(\tau \partial_{xx})(\partial_{xx}u^n+\partial_{x}u^n-2\partial_{x}a^n),\\
        u^{n+1}=&u^n+\tau\varphi_{1}(\tau \partial_{xx})(\partial_{xx}u^n-\partial_{x}u^n)\\
        &+\tau\varphi_{2}(\tau \partial_{xx})(3\partial_{x}u^n-4\partial_{x}a^n+\partial_{x}b^n)\\
        &+\tau\varphi_{3}(\tau \partial_{xx})(-4\partial_{x}u^n+8\partial_{x}a^n-4\partial_{x}b^n),
    \end{split}
    \end{equation}
    \item ETD-RK4:
    \begin{equation}\label{eq:ETDRK4_semi}
    \begin{split}
        a^n=&u^n+\frac{\tau}{2}\varphi_{1}(\frac{\tau}{2}\partial_{xx})(\partial_{xx}u^n-\partial_{x}u^n),\\
        b^n=&u^{n}+\frac{\tau}{2}\varphi_{1}(\frac{\tau}{2} \partial_{xx})(\partial_{xx}u^n-\partial_{x}a^n),\\
        c^n=&a^n+\frac{\tau}{2}\varphi_{1}(\frac{\tau}{2} \partial_{xx})(\partial_{xx}a^n+\partial_{x}u^n-2\partial_{x}b^n),\\
        u^{n+1}=&u^n+\tau\varphi_{1}(\tau \partial_{xx})(\partial_{xx}u^n-\partial_{x}u^n)\\
        &+\tau\varphi_{2}(\tau \partial_{xx})(3\partial_{x}u^{n}-2\partial_{x}a^n-2\partial_{x}b^n+\partial_{x}c^n)\\
        &+\tau\varphi_{3}(\tau \partial_{xx})(-4\partial_{x}u^n+4\partial_{x}a^n+4\partial_{x}b^n-4\partial_{x}c^n).
    \end{split}
    \end{equation}
\end{itemize}

We will apply the Fourier method to analyze the above ETD-RK schemes \cite{du2004stability}.
For simplicity, we assume that the solution \( u^n \) $\in L^1(\mathbb{R})\cap L^2(\mathbb{R})$ and consider the continuous Fourier transform.

We denote the Fourier transform of an integrable function $f$ on $\mathbb{R}$ by 
\begin{equation*}
\widehat{f}(\xi)=\mathcal{F}[f](\xi):=\int_{\mathbb{R}}f(x)e^{-\ii\xi x}\dd x.
\end{equation*}
It's well-known that $\mathcal{F}$ is an isometric transform on $L^2(\mathbb{R})$ (up to a constant factor), with the property,
\begin{equation*}
\mathcal{F}[\partial_x^\alpha f]=(\ii\xi)^{\alpha}\mathcal{F}[f].
\end{equation*}
Applying the Fourier transform to both sides of the semidiscrete ETD-RK1 scheme \eqref{eq:ETDRK1_semi}, we obtain
\begin{equation*}
\begin{split}
\widehat{u}^{n+1}&=e^{-\tau\xi^2}\widehat{u}^n+\frac{1}{\xi^2}(e^{-\tau\xi^2}-1)(\ii\xi)\widehat{u}^n\\
&=:\widehat{G}(\tau, \xi)\widehat{u}^{n},
\end{split}
\end{equation*}
where $\widehat{G}(\tau,\xi)=e^{-\tau\xi^2}+\frac{\ii}{\xi}(e^{-\tau\xi^2}-1)$ is the growth factor of the scheme in Fourier space.
We demonstrate that the time-step constraint for stability for the semidiscrete scheme is consistent with that of the fully discretized central scheme \eqref{eq:ETDRK1-FD}-\eqref{eq:ETD-CentralFD-A}, and the proof is very similar to that of Theorem \ref{thm:1}.
\begin{thm}
The semidiscrete ETD-RK1 scheme \eqref{eq:ETDRK1_semi} is stable under the time-step constriant $\tau\leq 2$, with the growth factor $|\widehat{G}(\tau,\xi)|\leq1$ for all $\xi$.
\end{thm}
\begin{proof}
One can calculate that
\begin{equation*}
\begin{split}
|\widehat{G}(\tau,\xi)|^2&=e^{-2\tau\eta}+\frac{1}{\eta}(e^{-\tau\eta}-1)^2\\
&=:\widehat{Q}(\tau,\eta), \quad \text{where } \eta=\xi^2\geq0.
\end{split}
\end{equation*}
Since 
\begin{equation*}
\frac{\partial \widehat{Q}(\tau,\eta)}{\partial\eta}=-\frac{e^{-2\tau\eta}}{\eta^2}\left((e^{\tau \eta}-\tau\eta-1)^2+\tau\eta^2(2-\tau)\right)\leq0\quad \forall\eta>0,
\end{equation*}
we have $\widehat{Q}(\tau,\eta)\leq\widehat{Q}(\tau,0)=\lim_{\eta\to 0} \widehat{Q}(\tau,\eta)=1$, which completes the proof.
\end{proof}
Indeed, the condition $\tau\leq2$ is also necessary for $|\widehat{G}(\tau,\xi)|\leq 1$, since 
\begin{equation*}
\left.\frac{\partial \widehat{Q}(\tau,\eta)}{\partial\eta}\right|_{\eta=0}=\tau(\tau-2)>0 \text{ for } \tau>2.
\end{equation*}

As a consequence of the theorem, we have $|\widehat{u}^{n+1}(\xi)|\leq|\widehat{u}^{n}(\xi)|$ for all $ \xi\in\mathbb{R}$, which implies $||\widehat{u}^{n+1}||_{L^2(\mathbb{R})}\leq ||\widehat{u}^{n}||_{L^2(\mathbb{R})}$ and, therefore, $||{u}^{n+1}||_{L^2(\mathbb{R})}\leq ||{u}^{n}||_{L^2(\mathbb{R})}$ due to Plancherel theorem.

Similarly, we can derive the expression $\widehat{G}(\tau,\xi)$ for higher-order semidiscrete ETD-RK methods \eqref{eq:ETDRK2_semi} -- \eqref{eq:ETDRK4_semi}, which are provided below.
\begin{itemize}
   \item ETD-RK2:
   \begin{equation*}
   \begin{split}
       \widehat{a}(\tau,\xi)=&1+\tau \varphi_1 (-\tau\xi^2) (-\xi^2-\ii\xi),\\
       \widehat{G}(\tau,\xi)=&\widehat{a}(\tau,\xi)+\tau\varphi_{2}(-\tau \xi^2)(\ii\xi-\ii\xi\widehat{a}(\tau,\xi)),
   \end{split}
   \end{equation*}
    \item ETD-RK3:
    \begin{equation*}
    \begin{split}
        \widehat{a}(\tau,\xi)=&1+\frac{\tau}{2}\varphi_1(-\frac{\tau}{2}\xi^2)(-\xi^2-\ii\xi),\\
        \widehat{b}(\tau,\xi)=&1+\tau\varphi_{1}(-\tau \xi^2)(-\xi^2+\ii\xi-2\ii\xi\widehat{a}(\tau,\xi)),\\
        \widehat{G}(\tau,\xi)=&1+\tau\varphi_{1}(-\tau \xi^2)(-\xi^2-\ii\xi)\\
        &+\tau\varphi_{2}(-\tau \xi^2)(3\ii\xi-4\ii\xi\widehat{a}(\tau,\xi)+\ii\xi\widehat{b}(\tau,\xi))\\
        &+\tau\varphi_{3}(-\tau \xi^2)(-4\ii\xi+8\ii\xi\widehat{a}(\tau,\xi)-4\ii\xi\widehat{b}(\tau,\xi)),
    \end{split}
    \end{equation*}
    \item ETD-RK4:    \begin{equation*}
    \begin{split}
        \widehat{a}(\tau,\xi)=&1+\frac{\tau}{2}\varphi_{1}(-\frac{\tau}{2}\xi^2)(-\xi^2-\ii\xi),\\
        \widehat{b}(\tau,\xi)=&1+\frac{\tau}{2}\varphi_{1}(-\frac{\tau}{2} \xi^2)(-\xi^2-\ii\xi\widehat{a}(\tau,\xi)),\\
        \widehat{c}(\tau,\xi)=&\widehat{a}(\tau,\xi)+\frac{\tau}{2}\varphi_{1}(-\frac{\tau}{2} \xi^2)(-\xi^2\widehat{a}(\xi,\tau)+\ii\xi-2\ii\xi\widehat{b}(\tau,\xi)),\\
        \widehat{G}(\tau,\xi)=&1+\tau\varphi_{1}(-\tau \xi^2)(-\xi^2-\ii\xi)\\
        &+\tau\varphi_{2}(-\tau\xi^2)(3\ii\xi-2\ii\xi\widehat{a}(\tau,\xi)-2\ii\xi\widehat{b}(\tau,\xi)+\ii\xi\widehat{c}(\tau,\xi))\\
        &+\tau\varphi_{3}(-\tau\xi^2)(-4\ii\xi+4\ii\xi\widehat{a}(\tau,\xi)+4\ii\xi\widehat{b}(\tau,\xi)-4\ii\xi\widehat{c}(\xi,\tau)).
    \end{split}
    \end{equation*}
\end{itemize}
Although the expressions are too complex for theoretical analysis, a numerical search reveals that there always exists a value $\tau_0$ such that $|\widehat{G}(\tau,\xi)|\leq1$ for all $\xi\in\mathbb{R}$ and $\tau\leq\tau_0$. 
The values of $\tau_0$ for various ETD-RK methods, along with the plots of $|\widehat{G}(\tau_0,\xi)|^2$ versus $\xi$, are shown in Figure \ref{fig:G2vsxi}.
These searched values of $\tau_0$ are valid up to the last digit shown, except for the ETD-RK1, where the value is exact. With the given $\tau_0$ values, the curves in Figure \ref{fig:G2vsxi} approach 1 from below, indicating $|\widehat{G}(\tau_0,\xi)|^2 \leq 1$. 
However, if $\tau$ is further increased, the curves in Figure \ref{fig:G2vsxi} exceed 1 (first doing so at the second peak in (b)–(d)), indicating that $|\widehat{G}(\tau,\xi)|^2 > 1$ and thus violating the stability condition.

\begin{figure}[!htbp]
 \centering
 \begin{subfigure}[b]{0.24\textwidth}
  \includegraphics[width=\textwidth]{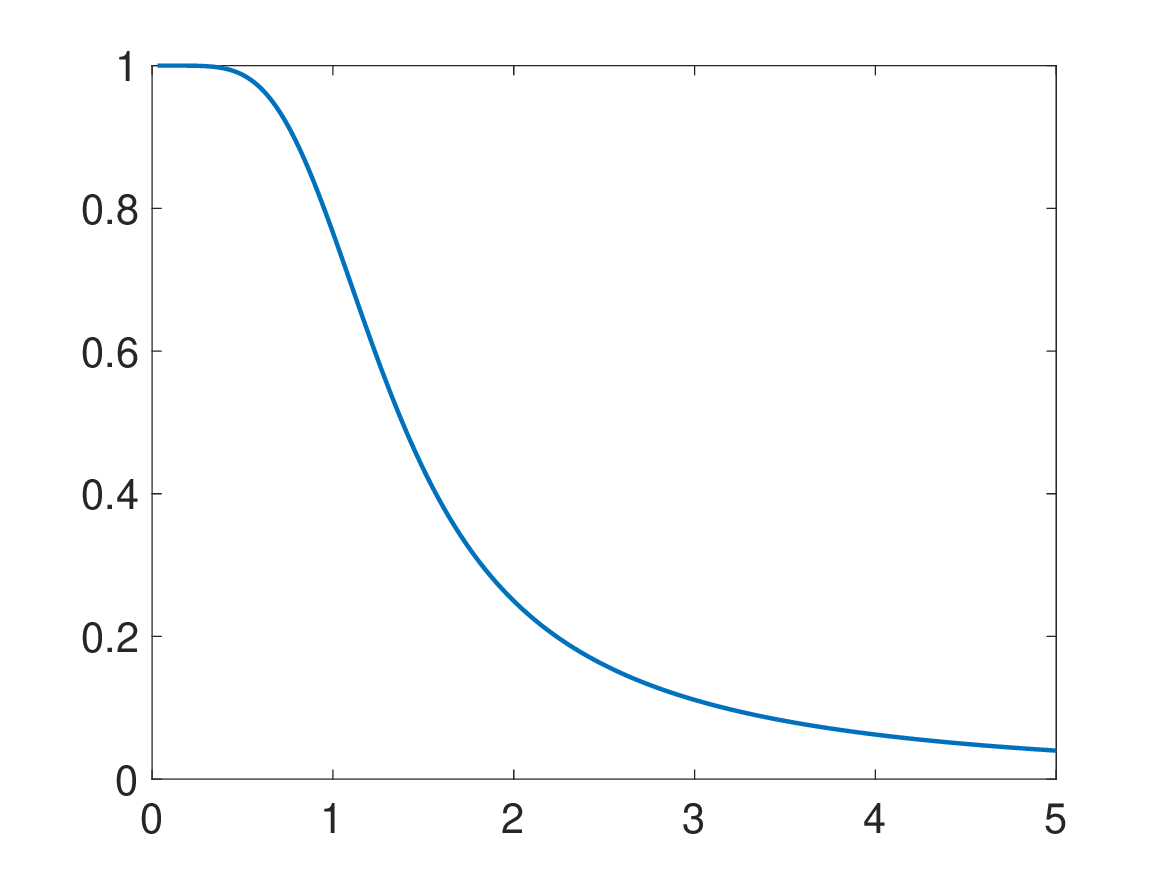}  \caption{ETD-RK1, $\tau_0=2$}
 \end{subfigure}
  \begin{subfigure}[b]{0.24\textwidth}
  \includegraphics[width=\textwidth]{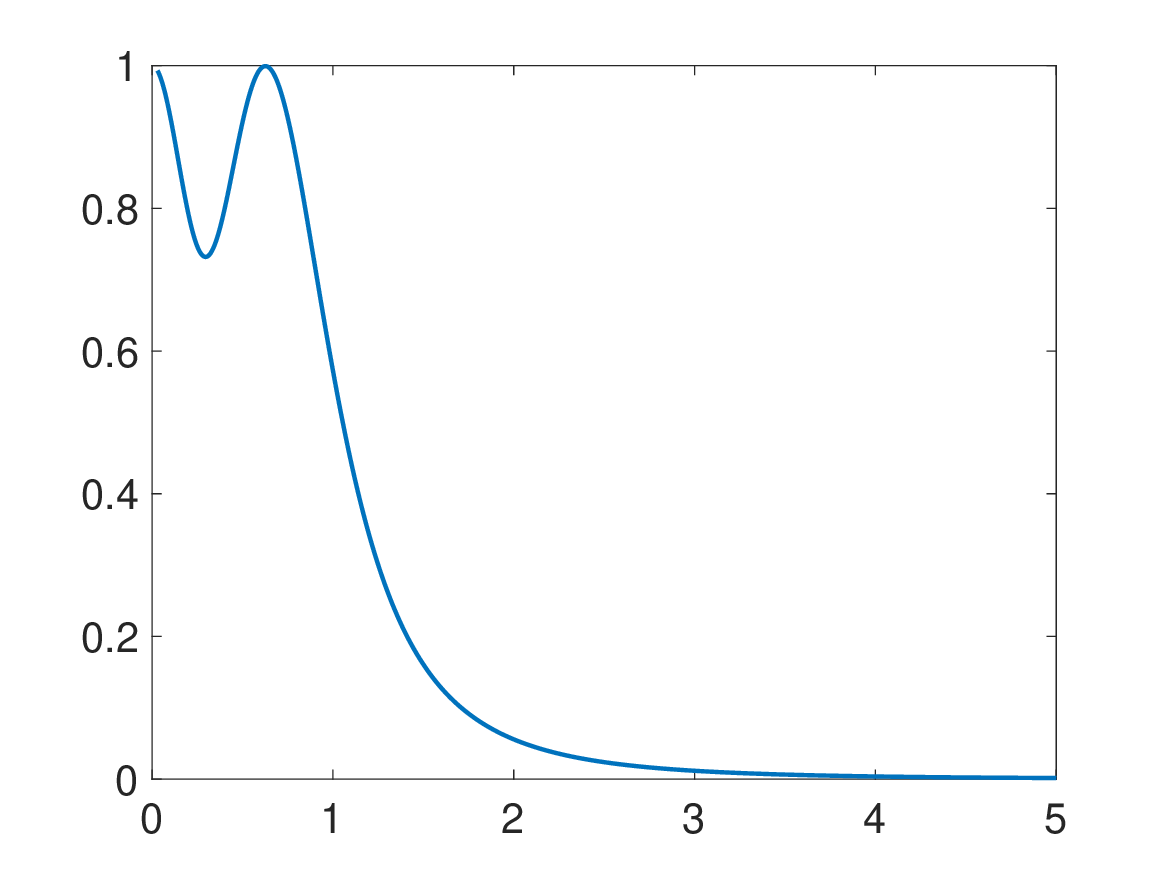}
\caption{ETD-RK2, $\tau_0=3.93$}
 \end{subfigure}
 \begin{subfigure}[b]{0.24\textwidth}
  \includegraphics[width=\textwidth]{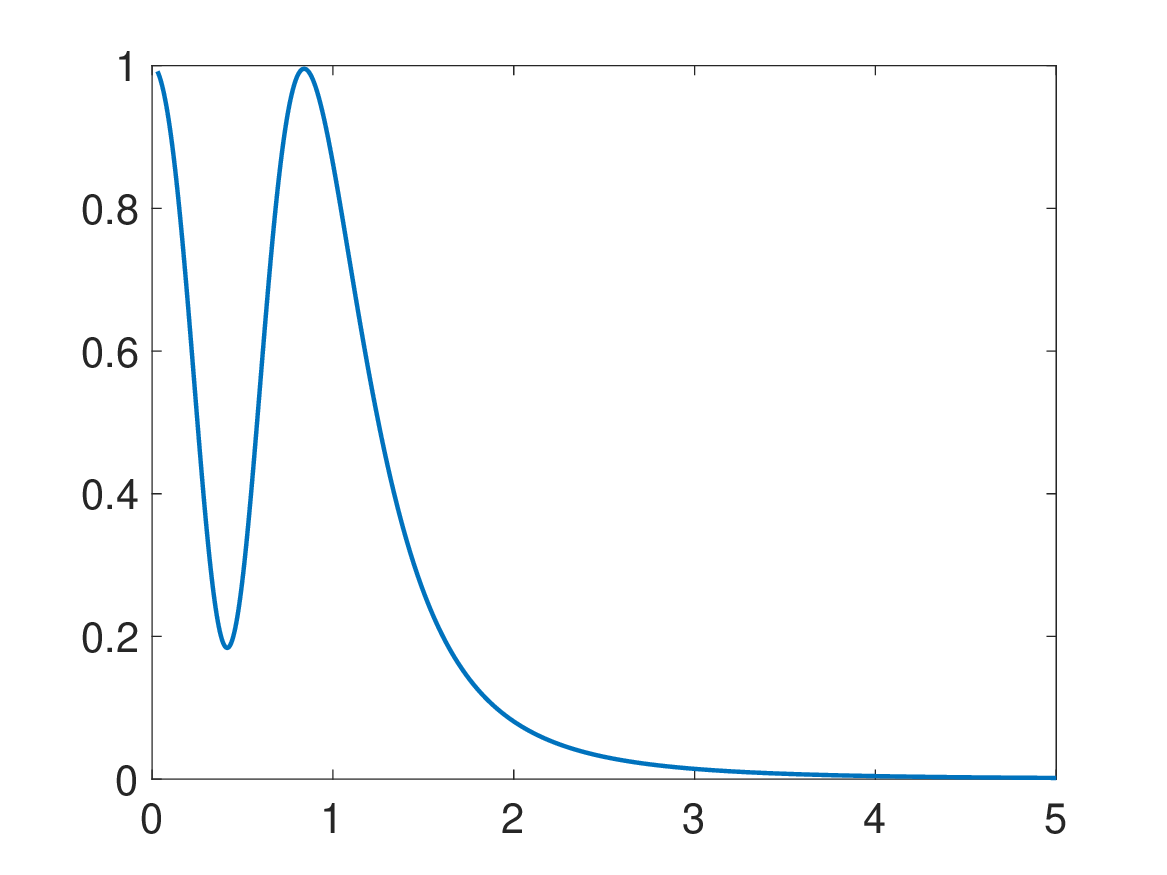}
\caption{ETD-RK3, $\tau_0=4.55$}
 \end{subfigure}
 \begin{subfigure}[b]{0.24\textwidth}
  \includegraphics[width=\textwidth]{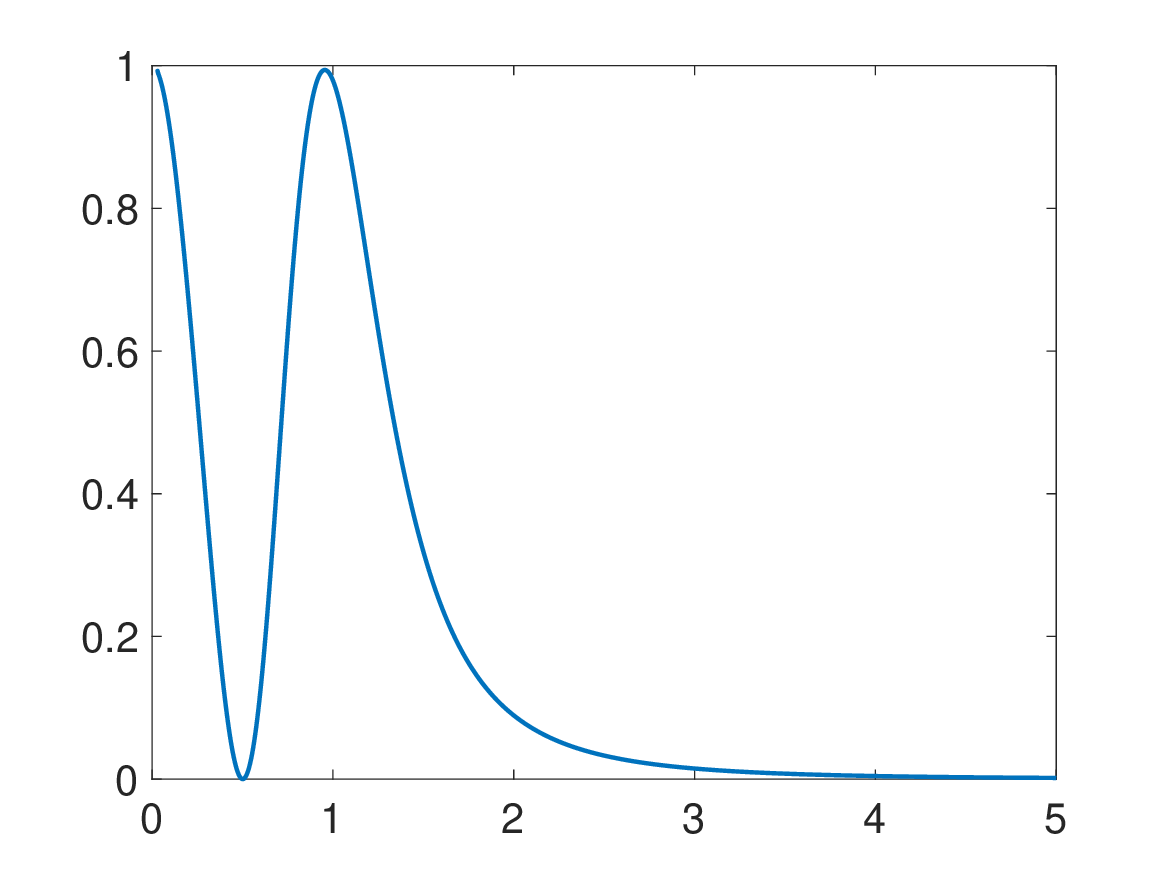}
  \caption{ETD-RK4, $\tau_0=4.81$}
 \end{subfigure}
 \caption{The square of growth factor $|\widehat{G}(\tau_0,\xi)|^2$ versus $\xi$. 
 Since $|\widehat{G}(\tau_0,\xi)|^2$ is an even function with respect to $\xi$, we present only $\xi>0$.
 The value of $\tau_0=2$ for ETD-RK1 is sharp.
 The searched values of $\tau_0$ for ETD-RK2, ETD-RK3, and ETD-RK4 are valid up to the last digit shown. For example, if $\tau_0=3.93$, then $\tau=3.93$ satisfies the stability condition, but $\tau=3.94$ does not.}
 \label{fig:G2vsxi}
\end{figure}
% For the general coefficients equation, we have $G(\tau,\xi;a,d)=e^{-\tau d\xi^2}+\frac{ai}{d\xi}(e^{-\tau d\xi^2}-1)$. Let $\eta=\xi^2$ and $Q(\tau,\eta;a,d)=|G(\tau,\xi;a,d)|^2$, we have, for $\tau\leq2\frac{d}{a^2}$,
% \begin{equation*}
% \frac{\partial Q}{\partial\eta}=-\frac{e^{-2d\tau \eta}}{\eta^2}\left((\frac{a}{d}e^{d\tau \eta}-a\tau\eta-\frac{a}{d})^2+d\tau\eta^2(2-\frac{a^2}{d}\tau)\right)\leq0,\quad \forall\eta>0,
% \end{equation*}

% main results:
% ETD-RK1,2,3,4 with LDG P0,P1,P2,P3. 
% All 16 combinations, CFL number. (upwind/central) 
% For central flux, $\tau\leq \max\tau_0\frac{d}{a^2}$.
% For upwind flux, $\tau\leq \max\{\tau_0\frac{d}{a^2},\lambda(a,d)h\}$.
% Does $\lambda$ depend on $d$?
% Do we have $\lambda(a,d)\geq \lambda(a,0)$?
% Can we obtain $\tau_0$ for the upwind-biased discretization?
% What is the value of the amplification factor for the lowest frequency (k=1)?

\subsection{ETD-RKDG schemes}
In this subsection, we study the stability of fully discrete ETD-RKDG schemes of varying orders. For further details on applying the Fourier approach for analyzing DG schemes, we refer to \cite{zhang2003analysis,guo2013superconvergence,chen2025runge}.

When the advection term is discretized using the DG method \eqref{eq:DG} with the central flux \eqref{eq:CentralFlux}, and the diffusion term is discretized using the LDG method \eqref{eq:LDG} with the alternating flux \eqref{eq:LDGAltFlux}, the resulting semidiscrete DG scheme for \eqref{eq:AdvDiffEq_Dimless} can be expressed as the following matrix equations:
\begin{equation*}
\begin{cases}
\frac{\dd\mathbf{u}_{j}}{\dd t} + A_{-1}\mathbf{u}_{j-1} + A_0\mathbf{u}_{j}+A_{1}\mathbf{u}_{j+1}=B_{0}\mathbf{p}_{j}+B_{1}\mathbf{p}_{j+1},\\
\mathbf{p}_{j}=C_{-1}\mathbf{u}_{j-1}+C_{0}\mathbf{u}_{j},
\end{cases}
\end{equation*}
where $\mathbf{u}_{j}, \mathbf{p}_{j}\in\mathbb{R}^{k+1}$ are the degrees of freedom of $u_h,p_h\in \mathbb{V}_{h}^{k}$ on cell $I_j$, for $j=1,2,\ldots, N$, and $A_{-1}, A_{0},\ldots, C_{0}$ are $(k+1)\times (k+1)$ real matrices associated with the DG spatial discretization.
The auxiliary variable $\mathbf{p}_j$ can be further eliminated to derive a more compact form of the matrix equation:
\begin{equation}\label{eq:MatrixEquation}
\frac{\dd\mathbf{u}_{j}}{\dd t} + A_{-1}\mathbf{u}_{j-1} + A_0\mathbf{u}_{j} + A_{1}\mathbf{u}_{j+1} = D_{-1}\mathbf{u}_{j-1}+D_0\mathbf{u}_{j}+D_{1}\mathbf{u}_{j+1}.
\end{equation}
The matrix equations are determined once the polynomial order $k$ and basis functions on each mesh cell are specified.
For instance, if we adopt the piecewise linear DG space ($k=1$) and use the Lagrange interpolation basis at the $k+1$ Legendre--Gauss--Lobatto points, we have $\mathbf{u}_{j}=\begin{bmatrix}u_{j-\frac12}^{+},u_{j+\frac12}^{-}\end{bmatrix}^T$, and the matrices in \eqref{eq:MatrixEquation} are given by:
\begin{equation*}
\begin{split}
&A_{-1} = \frac1h\begin{bmatrix}
    0& -2\\
    0& 1
\end{bmatrix},\quad 
A_{0} = \frac1h\begin{bmatrix}
    1& 2\\
    -2& -1
\end{bmatrix},\quad 
A_{1} = \frac1h\begin{bmatrix}
    -1& 0\\
    2& 0
\end{bmatrix},
\end{split}
\end{equation*}
and
\begin{equation*}
\begin{split}
&D_{-1}=\frac{1}{h^2}\begin{bmatrix}
    0&10\\
    0&-2
\end{bmatrix}, \quad 
D_{0}=\frac{1}{h^2}\begin{bmatrix}
    -12&10\\
    6&-20
\end{bmatrix}, \quad 
D_1=\frac{1}{h^2}\begin{bmatrix}
    -6&-2\\
    12&4
\end{bmatrix}.
\end{split}
\end{equation*}
Similarly, using the upwind flux \eqref{eq:UpwindFlux} in the DG discretization \eqref{eq:DG} for the advection term and the IPDG discretization \eqref{eq:IPDG} for the diffusion term results in the matrix equation \eqref{eq:MatrixEquation} with the following matrices:
\begin{equation*}
\begin{split}
&A_{-1} = \frac1h\begin{bmatrix}
   0 & -4\\
   0 & 2
\end{bmatrix},\quad 
A_{0} = \frac1h\begin{bmatrix}
   3 & 1\\
   -3 & 1
\end{bmatrix},\quad 
A_{1} = \begin{bmatrix}
    0 & 0\\
    0 & 0
\end{bmatrix},
\end{split}
\end{equation*}
and
\begin{equation*}
\begin{split}
D_{-1}=\frac{1}{h^2}&\begin{bmatrix}
    2 & -2+3\epsilon+4\sigma\\
    -1 & 1-3\epsilon-2\sigma
\end{bmatrix}, 
D_{0}=\frac{1}{h^2}\begin{bmatrix}
    -3-3\epsilon-4\sigma & 3+3\epsilon+2\sigma\\
    3+3\epsilon+2\sigma & -3-3\epsilon-4\sigma
\end{bmatrix}, 
D_1=\frac{1}{h^2}\begin{bmatrix}
    1-3\epsilon-2\sigma & -1\\
    -2+3\epsilon+4\sigma& 2
\end{bmatrix}.
\end{split}
\end{equation*}
The derivation of matrix equations for other choices of schemes, numerical fluxes, DG spaces, and basis functions follows a pattern similar to that of \eqref{eq:MatrixEquation} (possibly with wider stencils) and is omitted here for brevity.

Consider a Fourier mode $\mathbf{u}_j(t)=\widehat{\mathbf{u}}(t)e^{\ii \omega jh}$.
Substituting it into the matrix equation \eqref{eq:MatrixEquation}, we obtain the following equation for the mode:
\begin{equation}\label{eq:FourierODEs}
\begin{split}
\frac{\dd \widehat{\mathbf{u}}(t)}{\dd t}&=-\left( A_{-1}e^{-\ii \xi} + A_0 +A_{1}e^{\ii \xi} \right)\widehat{\mathbf{u}}(t) + \left(D_{-1}e^{-\ii \xi}+D_0+D_1e^{\ii \xi} \right)\widehat{\mathbf{u}}(t)\\
:&=\left(-\hat{A}(h,\xi)+\hat{D}(h,\xi)\right)\widehat{\mathbf{u}}(t),\quad \text{where } \xi=\omega h.
\end{split}
\end{equation}
We apply the ETD-RK schemes \eqref{eq:ETDRK1}-\eqref{eq:ETDRK4} to the modal equation \eqref{eq:FourierODEs} and numerically investigate their stability.
The ETD-RK1 scheme yields the following difference equation:
\begin{equation*}
\begin{split}
\widehat{\mathbf{u}}^{n+1} &= e^{\tau \hat{D}(h,\xi)}\widehat{\mathbf{u}}^n-\hat{D}^{-1}(h,\xi)\left(e^{\tau\hat{D}(h,\xi)}-I\right)\hat{A}(h,\xi)\widehat{\mathbf{u}}^n,\\
:&=\widehat{G}(\tau,h,\xi)\widehat{\mathbf{u}}^n,
\end{split}
\end{equation*}
and similarly for other higher-order ETD-RK schemes.
As expected, these fully discrete ETD-RKDG schemes converge to their semi-discrete ETD-RK counterparts as $h\rightarrow0$.

{We denote by $\rho(\widehat{G}(\tau,h,\xi))$ the spectral radius of the matrix growth factor $\widehat{G}(\tau,h,\xi)$ and seek a time-step constraint that ensures the stability condition 
\begin{equation}\label{eq:full_discrete_stab}
\sup_{\xi\in[-\pi,\pi]}\rho(\widehat{G}(\tau,h,\xi))\leq 1\quad\forall h>0.
\end{equation}
As suggested by the lowest-order ETD-RKDG schemes, we expect time-step constraints of the form $\tau\leq C$ for the central DG flux \eqref{eq:CentralFlux} and $\tau\leq\max\{C,\tilde{C} h\}$ for the upwind DG flux \eqref{eq:UpwindFlux}.

Since $\rho(\widehat{G}(\tau,h,\xi))$ should converge to $|\widehat{G}(\tau,\xi)|$ in Section \ref{sec:semi-ETD} in the limit $h\rightarrow0$, we anticipate that the values $\tau_0$ determined from semi-discrete ETD-RK schemes would serve as upper bounds for $C$. We sample $\xi$ in $[-\pi, \pi]$ and take different $h>0$, and compute the corresponding spectral radii.
To minimize the effects of round-off errors in floating-point computations, we use extended precision with 512-bit floating-point arithmetic (BigFloat, refer to \cite{BigFloat2025}) on the Julia Platform.
Surprisingly, extensive numerical experiments indicate that the values $\tau_0$ obtained from semi-discrete ETD-RK schemes are not only necessary but also sufficient for ensuring the stability condition \eqref{eq:full_discrete_stab}, regardless of the polynomial degree or the specific choice of the
DG method.
This remains true even when extending the truncated 
$\tau_0$ values to ten decimal places.

In summary, we establish the stability condition \eqref{eq:full_discrete_stab} under the time-step constraints $\tau\leq \tau_0$ and $\tau\leq\max\{\tau_0,c_0 h\}$ for the central \eqref{eq:CentralFlux} and upwind \eqref{eq:UpwindFlux} DG fluxes, respectively.  
The values of $\tau_0$ are provided in Table \ref{tab:tau_0} up to two decimal places, and $c_0$ represents the standard CFL constants reported in Table 2.2 of \cite{cockburn2001runge}.}
\begin{table}[h!]
\centering
\begin{tabular}{|c|c|c|c|c|}
\hline
method & ETD-RK1 & ETD-RK2 & ETD-RK3 & ETD-RK4 \\ \hline
$\tau_0$ & 2 & 3.93 & 4.55 & 4.81 \\ \hline
\end{tabular}
\caption{Stable $\tau_0$ for $\sup_{\xi\in[-\pi,\pi]}\rho(\widehat{G}(\tau,h,\xi))\leq 1 \ \forall h>0$.}\label{tab:tau_0}
\end{table}

Finally, we visualize the square of the growth factor, $\rho(\widehat{G}(\tau_0,h,\xi))^2$, as a function of $\xi$ for a specific spatial discretization setting: central DG flux for the advection combined with LDG discretization for diffusion, using a $\mathbb{P}^4$ polynomial space and $h = \frac{\pi}{10^{6}}$. 
The results for other spatial discretization choices are close.
From Figure \ref{fig:G2_DG_vsxi_ETD}, we observe that the pattern of $\rho(\widehat{G}(\tau_0,h,\xi))^2$ closely resembles their semi-discrete counterparts $|\widehat{G}(\tau_0,\xi)|^2$.
\begin{figure}[!htbp]
 \centering
 \begin{subfigure}[b]{0.24\textwidth}
  \includegraphics[width=\textwidth]{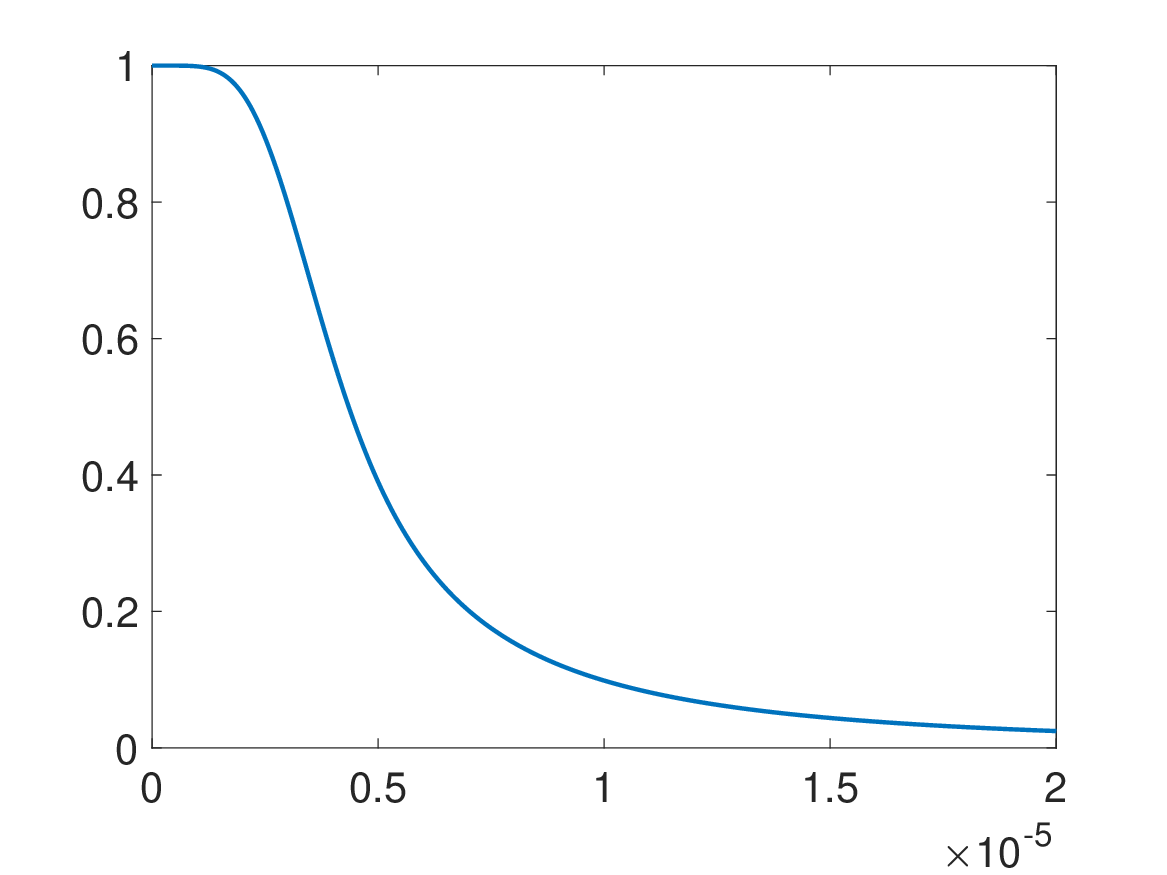}  \caption{ETD-RK1, $\tau_0=2$}
 \end{subfigure}
  \begin{subfigure}[b]{0.24\textwidth}
  \includegraphics[width=\textwidth]{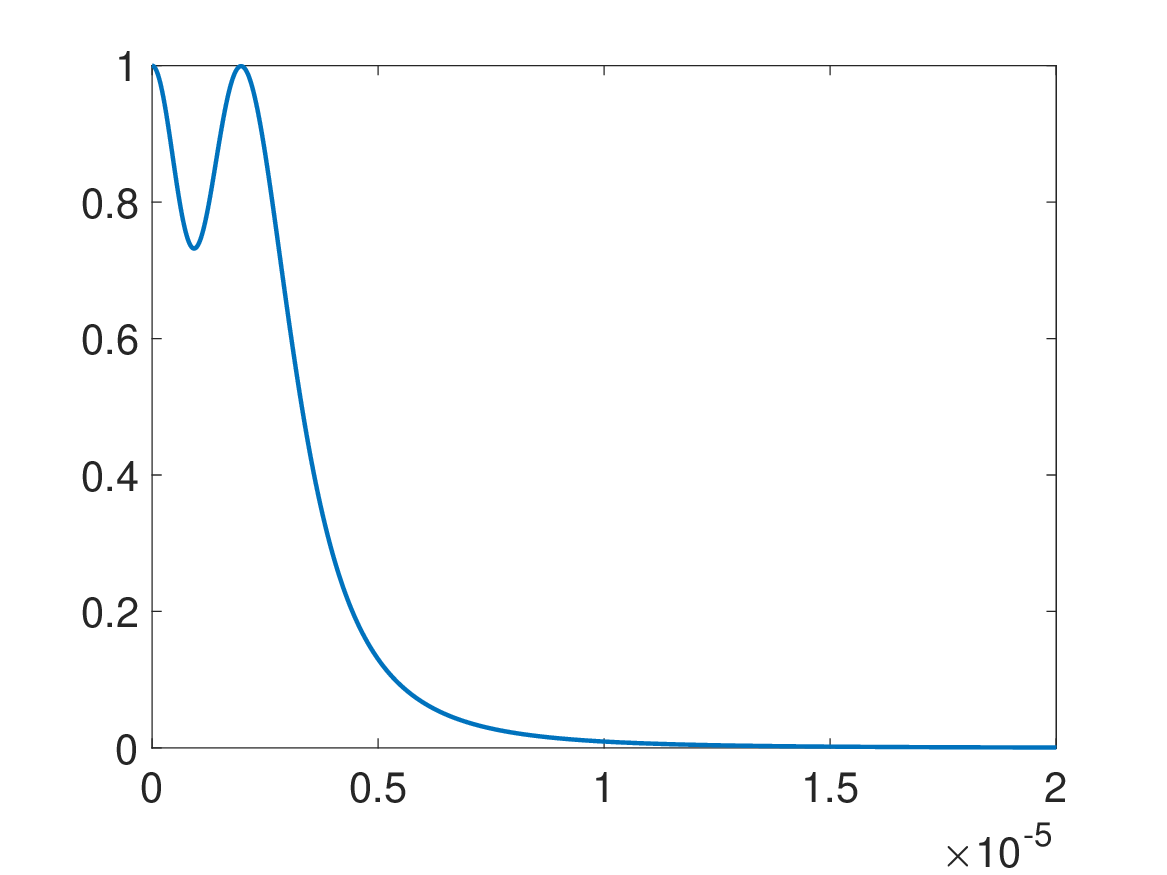}
\caption{ETD-RK2, $\tau_0=3.93$}
 \end{subfigure}
 \begin{subfigure}[b]{0.24\textwidth}
  \includegraphics[width=\textwidth]{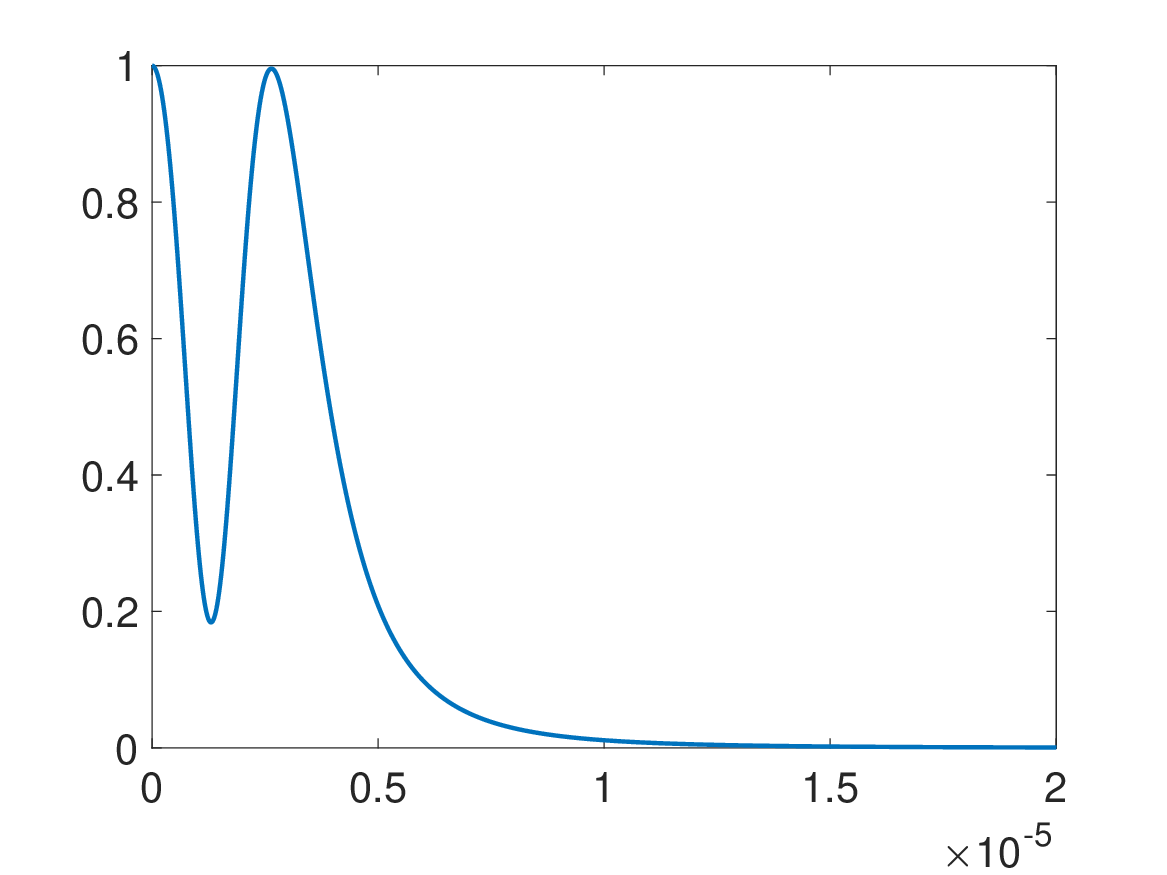}
\caption{ETD-RK3, $\tau_0=4.55$}
 \end{subfigure}
 \begin{subfigure}[b]{0.24\textwidth}
  \includegraphics[width=\textwidth]{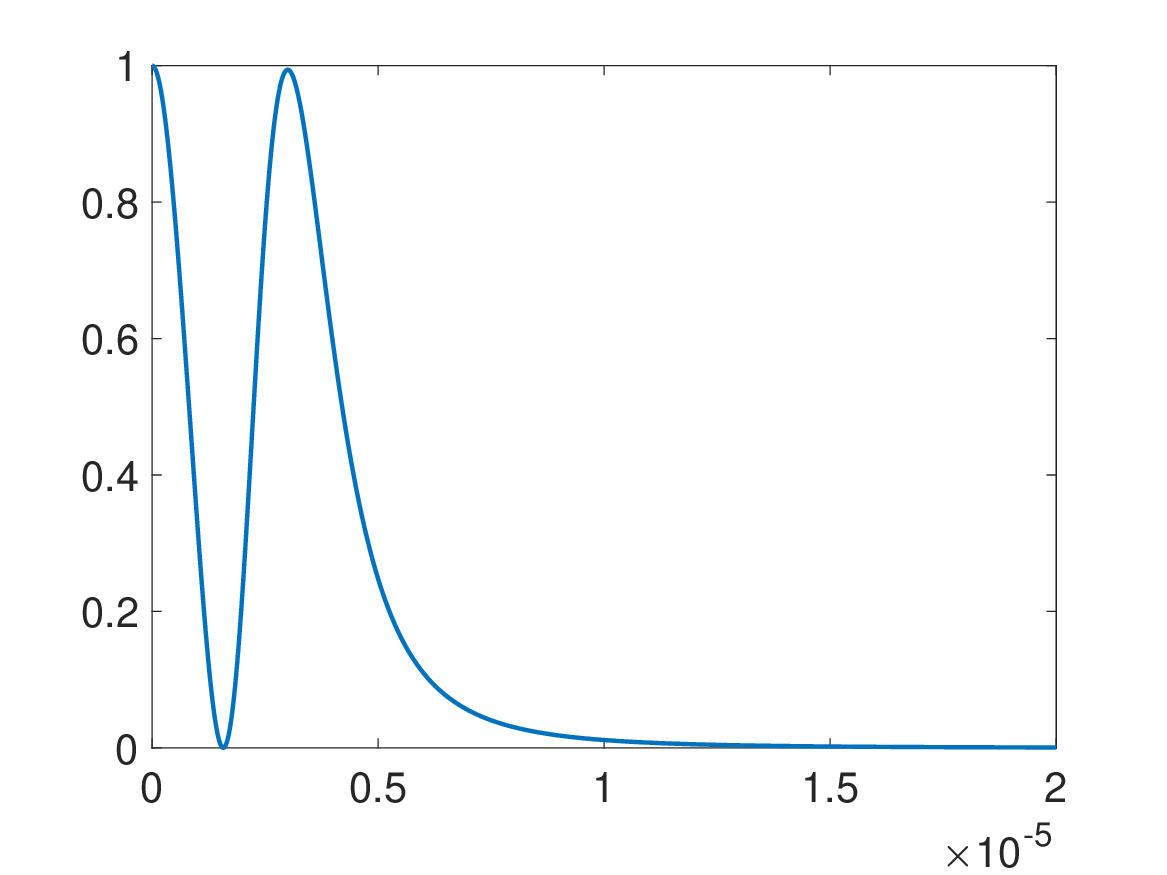}
\caption{ETD-RK4, $\tau_0=4.81$}
 \end{subfigure} 
 \caption{The square of growth factor $\rho(\widehat{G}(\tau_0,h,\xi))^2$ versus $\xi=\omega h$ for a specific spatial discretization setting: central DG flux for the advection combined with LDG discretization for diffusion, using a $\mathbb{P}^4$ polynomial space and $h = \frac{\pi}{10^{6}}$.
 The results for other spatial discretization choices are close.}
 \label{fig:G2_DG_vsxi_ETD}
\end{figure}
\begin{rem}
    Under different choices of basis functions, the DG matrices in \eqref{eq:MatrixEquation} will differ, but they are all similar matrices. As a result, the spectral radius of \( \widehat{G} \) is independent of the choice of basis functions. 
\end{rem}

\section{Numerical experiments}\label{sec:Numerics}
In this section, we numerically examine the stability and time-step constraints of the ETD-RKDG methods analyzed in the previous sections.

% Five examples are tested:
% The first example demonstrates the designed order of spatial and temporal accuracy of the ETD-RKDG methods.
% The second example verifies the stability of the ETD-RKDG methods on uniform meshes.
% Since the time-step constraints exhibit independence from the spatial discretization, we study the stability of the ETD-RKDG methods on nonuniform grids and with varying polynomial degrees in the third example.
% The fourth example examines the performance of ETD-RKDG methods on equations with nonlinear advection, focusing on one-dimensional cases.
% The final example is a two-dimensional test involving nonlinear advection terms.

The DG method \eqref{eq:DG} with the central flux \eqref{eq:CentralFlux} is used to discretize the advection terms. It is worth noting that the central flux is typically considered a less desirable choice for the explicit discretization of hyperbolic equations due to its reduced stability. However, we employ the central flux here to better illustrate the stability enhancement provided by ETD-RK methods, as this improvement is more evident when applied to a less stable flux.
The numerical results for different discretizations of the diffusion terms are very similar and do not affect the conclusions of this section. For brevity, we present only the results of the LDG method \eqref{eq:LDG} with alternating flux \eqref{eq:LDGAltFlux} for the one-dimensional examples and the SIPG method for the final two-dimensional example.

\begin{exmp}\textbf{Accuracy test}\label{ex:accuracy}
\end{exmp}
In this example, we test the accuracy of the ETD-RKDG methods.
We solve the advection-diffusion equation \eqref{eq:AdvDiffEq_Dimless} on the computational domain $\Omega=[0,2\pi]$ with the initial condition $u(x,0)=\sin(x)$ and periodic boundary conditions.
The exact solution for this problem is given by $u(x,t)=e^{-t}\sin(x-t)$.

We use combinations of DG spatial discretizations and ETD-RK methods of different orders.
Uniform meshes are employed, and the time-step size is set to $\tau = h$.
The $L^2$ errors and orders of convergence at the terminal time $T = 1$ for different methods are presented in Table \ref{tab:accuracy}.
From the error table, we observe that the overall order of convergence is determined by the lower of the spatial and temporal accuracy orders, as expected.

\begin{table}[!htbp]
\centering
\begin{tabular}{ccccccccc}
\toprule[1.5pt]
\multicolumn{9}{c}{ETD-RK1} \\
\midrule
& $\mathbb{P}^0$-DG & & $\mathbb{P}^1$-DG & & $\mathbb{P}^2$-DG & & $\mathbb{P}^3$-DG\\
\cline{2-3} \cline{4-5} \cline{6-7} \cline{8-9}  
$h$ & $L^2$ Error & Order & $L^2$ Error & Order & $L^2$ Error & Order& $L^2$ Error & Order\\
\midrule
$\pi/10$ & $1.80\times10^{-1}$ & - & $1.69\times10^{-1}$ & - & $1.72\times10^{-1}$ & - & $1.72\times10^{-1}$ &-\\
$\pi/20$ & $8.62\times10^{-2}$ & $1.06$ & $7.85\times10^{-2}$ & $1.11$ & $7.93\times10^{-2}$ & $1.12$ & $7.93\times10^{-2}$ & $1.12$\\
$\pi/40$ & $4.22\times10^{-2}$ & $1.03$ & $3.77\times10^{-2}$ & $1.06$ & $3.79\times10^{-2}$ & $1.06$ & $3.79\times10^{-2}$ & $1.06$\\
$\pi/80$ & $2.07\times10^{-2}$ & $1.03$ & $1.85\times10^{-2}$ & $1.03$ & $1.85\times10^{-2}$ & $1.03$ & $1.85\times10^{-2}$ & $1.03$\\
%$\pi/160$ & $1.03\times10^{-2}$ & $1.01$ & $9.15\times10^{-3}$ & $1.01$ & $9.16\times10^{-3}$ & $1.02$ & $9.16\times10^{-3}$ & $1.02$\\
\midrule
\multicolumn{9}{c}{ETD-RK2} \\
\midrule
& $\mathbb{P}^0$-DG & & $\mathbb{P}^1$-DG & & $\mathbb{P}^2$-DG & & $\mathbb{P}^3$-DG\\
\cline{2-3} \cline{4-5} \cline{6-7} \cline{8-9}  
$h$ & $L^2$ Error & Order & $L^2$ Error & Order & $L^2$ Error & Order& $L^2$ Error & Order\\
\midrule
$\pi/10$ & $5.96\times10^{-2}$ & - & $1.79\times10^{-2}$ & - & $1.71\times10^{-2}$ & - & $1.70\times10^{-2}$ &-\\
$\pi/20$ & $3.46\times10^{-2}$ & $0.79$ & $4.69\times10^{-3}$ & $1.93$ & $4.29\times10^{-3}$ & $1.99$ & $4.29\times10^{-3}$ & $1.99$\\
$\pi/40$ & $1.87\times10^{-2}$ & $0.89$ & $1.20\times10^{-3}$ & $1.97$ & $1.07\times10^{-3}$ & $2.01$ & $1.07\times10^{-3}$ & $2.01$\\
$\pi/80$ & $9.34\times10^{-3}$ & $1.00$ & $3.00\times10^{-4}$ & $2.00$ & $2.66\times10^{-4}$ & $2.01$ & $2.66\times10^{-4}$ & $2.01$\\
%$\pi/160$ & $4.76\times10^{-3}$ & $0.97$ & $7.53\times10^{-5}$ & $1.99$ & $6.64\times10^{-5}$ & $2.00$ & $6.64\times10^{-5}$ & $2.00$\\
\midrule
\multicolumn{9}{c}{ETD-RK3} \\
\midrule
& $\mathbb{P}^0$-DG & & $\mathbb{P}^1$-DG & & $\mathbb{P}^2$-DG & & $\mathbb{P}^3$-DG\\
\cline{2-3} \cline{4-5} \cline{6-7} \cline{8-9}  
$h$ & $L^2$ Error & Order & $L^2$ Error & Order & $L^2$ Error & Order& $L^2$ Error & Order\\
\midrule
$\pi/10$ & $6.04\times10^{-2}$ & - & $5.84\times10^{-3}$ & - & $1.23\times10^{-3}$ & - & $1.24\times10^{-3}$ &-\\
$\pi/20$ & $3.47\times10^{-2}$ & $0.80$ & $1.56\times10^{-3}$ & $1.91$ & $1.52\times10^{-4}$ & $3.02$ & $1.54\times10^{-4}$ & $3.01$\\
$\pi/40$ & $1.87\times10^{-2}$ & $0.89$ & $4.07\times10^{-4}$ & $1.94$ & $1.87\times10^{-5}$ & $3.02$ & $1.89\times10^{-5}$ & $3.02$\\
$\pi/80$ & $9.35\times10^{-3}$ & $1.00$ & $1.00\times10^{-4}$ & $2.02$ & $2.32\times10^{-6}$ & $3.01$ & $2.35\times10^{-6}$ & $3.01$\\
%$\pi/160$ & $4.77\times10^{-3}$ & $0.97$ & $2.54\times10^{-5}$ & $1.98$ & $2.89\times10^{-7}$ & $3.01$ & $2.92\times10^{-7}$ & $3.01$\\
\midrule
\multicolumn{9}{c}{ETD-RK4} \\
\midrule
& $\mathbb{P}^0$-DG & & $\mathbb{P}^1$-DG & & $\mathbb{P}^2$-DG& & $\mathbb{P}^3$-DG\\
\cline{2-3} \cline{4-5} \cline{6-7} \cline{8-9}  
$h$ & $L^2$ Error & Order & $L^2$ Error & Order & $L^2$ Error & Order & $L^2$ Error & Order \\
\midrule
$\pi/10$ & $6.05\times10^{-2}$ & - & $5.05\times10^{-3}$ & - & $1.45\times10^{-4}$ & - & $1.09\times10^{-4}$ &-\\
$\pi/20$ & $3.47\times10^{-2}$ & $0.80$ & $1.47\times10^{-3}$ & $1.78$ & $1.74\times10^{-5}$ & $3.06$ & $6.76\times10^{-6}$ & $4.00$\\
$\pi/40$ & $1.87\times10^{-2}$ & $0.89$ & $3.98\times10^{-4}$ & $1.89$ & $2.29\times10^{-6}$ & $2.93$ & $4.19\times10^{-7}$ & $4.01$\\
$\pi/80$ & $9.35\times10^{-3}$ & $1.00$ & $9.92\times10^{-5}$ & $2.00$ & $2.87\times10^{-7}$ & $3.00$ & $2.60\times10^{-8}$ & $4.01$\\
%$\pi/160$ & $4.77\times10^{-3}$ & $0.97$ & $2.53\times10^{-5}$ & $1.97$ & $3.66\times10^{-8}$ & $2.97$ & $1.62\times10^{-9}$ & $4.00$\\
\bottomrule[1.5pt]
\end{tabular}
\caption{\textbf{Example \ref{ex:accuracy}. Accuracy test.} 
The $L^2$ errors and orders of convergence for solving the linear advection-diffusion equation \eqref{eq:AdvDiffEq_Dimless} using different ETD-RKDG methods. 
Uniform meshes are employed, and the time-step size is set to $\tau = h$.}
\label{tab:accuracy}
\end{table}

\begin{exmp}\textbf{Stability test}\label{ex:stability}
\end{exmp}
In this example, we test the stability of the ETD-RKDG methods.
We solve the advection-dominated equation \eqref{eq:AdvDiffEq} with $a=1$ and $d=0.01$ on the computational domain $[0, 2\pi]$ with the initial condition $u(x,0)=\sin(x)$ and periodic boundary conditions. 
The exact solution for this problem is given by $u(x,t)=e^{-0.01t}\sin(x-t)$.

We use combinations of DG spatial discretizations and ETD-RK methods of different orders.
A uniform mesh with $h=\frac{\pi}{1000}$ is employed.
The time-step sizes are set to the critical values $\tau=\tau_0\frac{d}{a^2}$, with $\tau_0 = 2, 3.93, 4.55$, and $4.81$ for ETD-RK1, ETD-RK2, ETD-RK3, and ETD-RK4, respectively.
These time-step sizes are significantly larger than the mesh size $h$.

The growth of the maximum norm of the numerical solutions over time is presented in log scale in Figure \ref{fig:stability}, along with results obtained using slightly larger time-steps $\tau=1.1\times\tau_0\frac{d}{a^2}$. 
As shown in the figure, the numerical solutions are stable under the time-step constraint $\tau=\tau_0\frac{d}{a^2}$, but they exhibit instability with $10\%$ larger time-steps.
Interestingly, the decay of the maximum norm of the solutions for ETD-RK1 in Figure \ref{fig:stab-ETD-RK1} is almost invisible, unlike those for ETD-RK2, ETD-RK3, and ETD-RK4, which decay at the correct rate of $e^{-0.01t}$.
This is because the value $\tau_0=2$ is sharp for ETD-RK1 to maintain stability, but the temporal accuracy is relatively low. 
%\zs{Yes, it is inaccurate. But why does the norm not change?}
We also observe that most numerical solutions with $\tau=1.1\times\tau_0\frac{d}{a^2}$ blow up in the simulation, with only a few exceptions. However, these exceptions are not inherently stable, as they can still exhibit instability when a different mesh size $h$ (e.g., $h=\frac{\pi}{100}$) is used in our experiments.

\begin{figure}[!htbp]
 \centering
 \begin{subfigure}[b]{0.3\textwidth}
  \includegraphics[width=\textwidth]{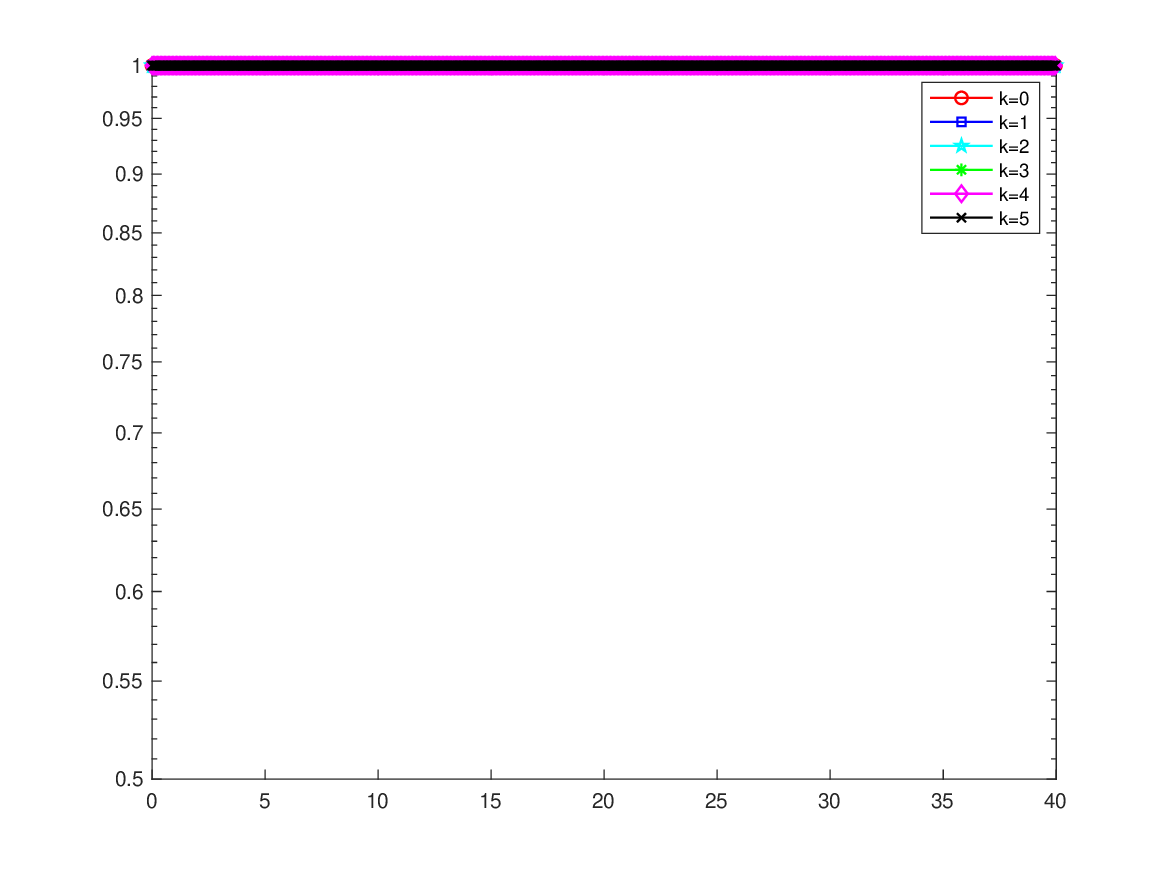}
  \caption{ETD-RK1, $\tau=\tau_0 \frac{d}{a^2}$}\label{fig:stab-ETD-RK1}
 \end{subfigure}
 \begin{subfigure}[b]{0.3\textwidth}
  \includegraphics[width=\textwidth]{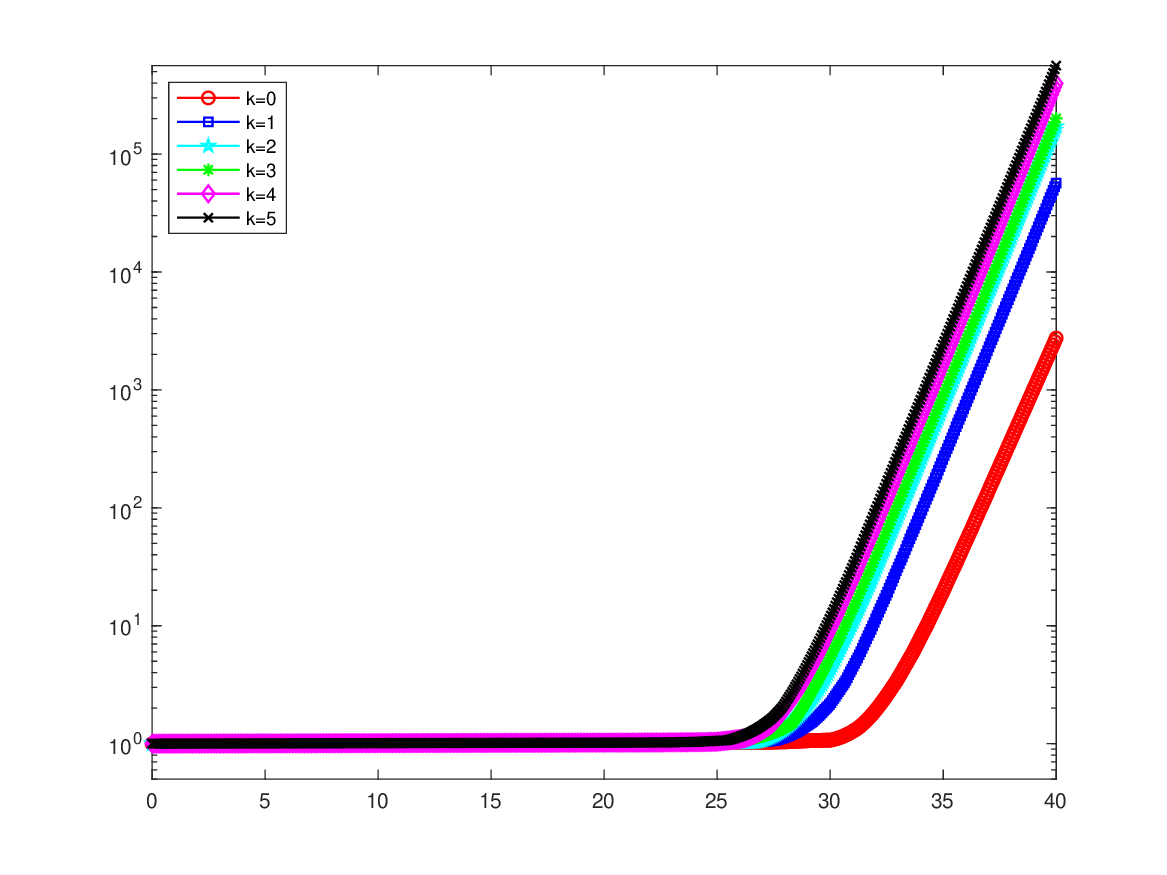}
  \caption{ETD-RK1, $\tau=1.1\times\tau_0 \frac{d}{a^2}$}
 \end{subfigure}

 \begin{subfigure}[b]{0.3\textwidth}
  \includegraphics[width=\textwidth]{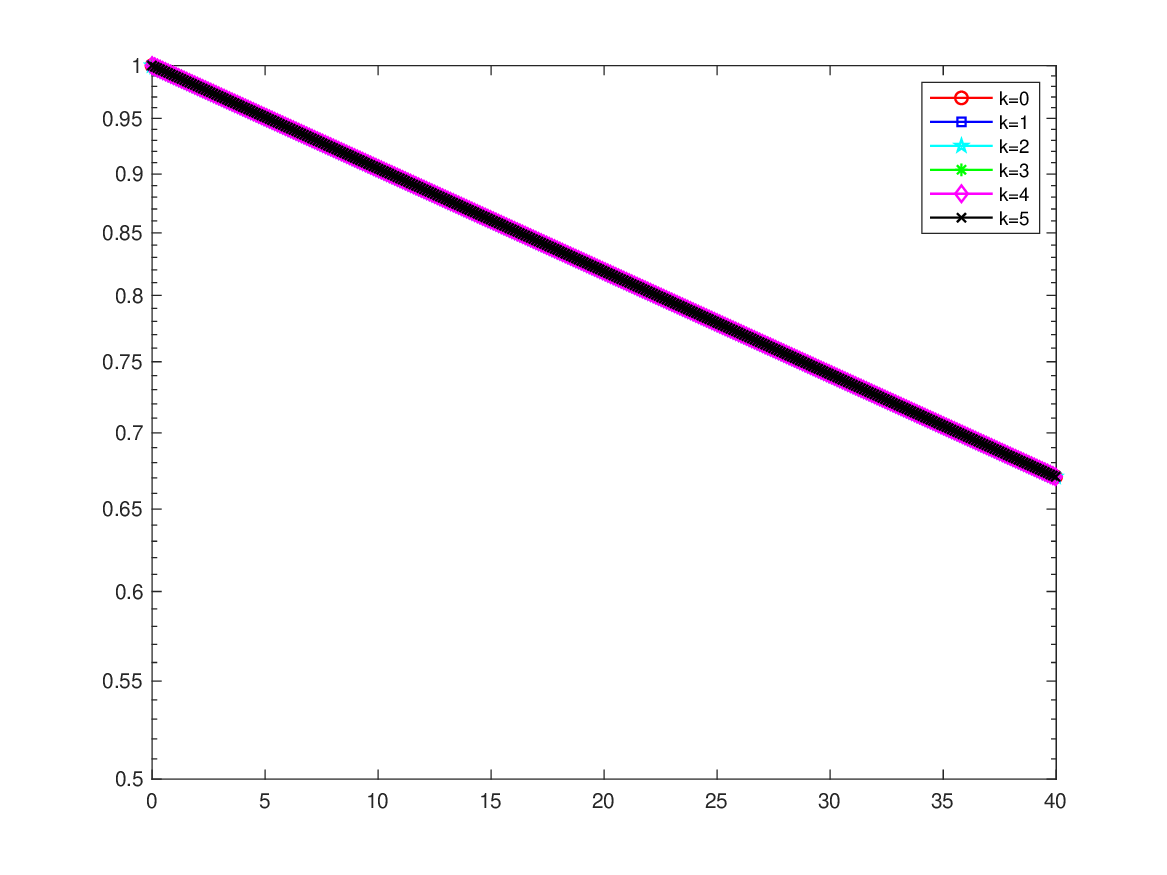}
  \caption{ETD-RK2, $\tau=\tau_0 \frac{d}{a^2}$}
 \end{subfigure}
 \begin{subfigure}[b]{0.3\textwidth}
  \includegraphics[width=\textwidth]{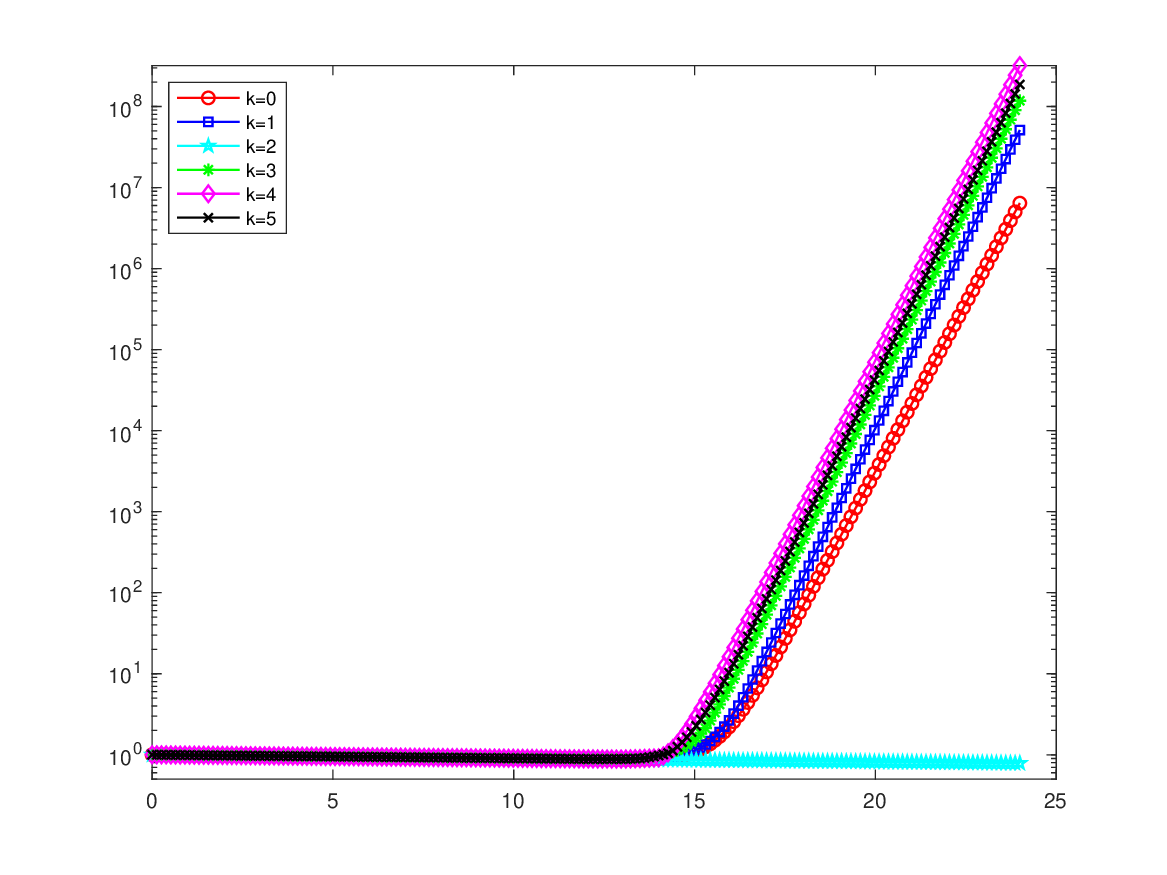}
  \caption{ETD-RK2, $\tau=1.1\times\tau_0 \frac{d}{a^2}$}
 \end{subfigure}

 \begin{subfigure}[b]{0.3\textwidth}
  \includegraphics[width=\textwidth]{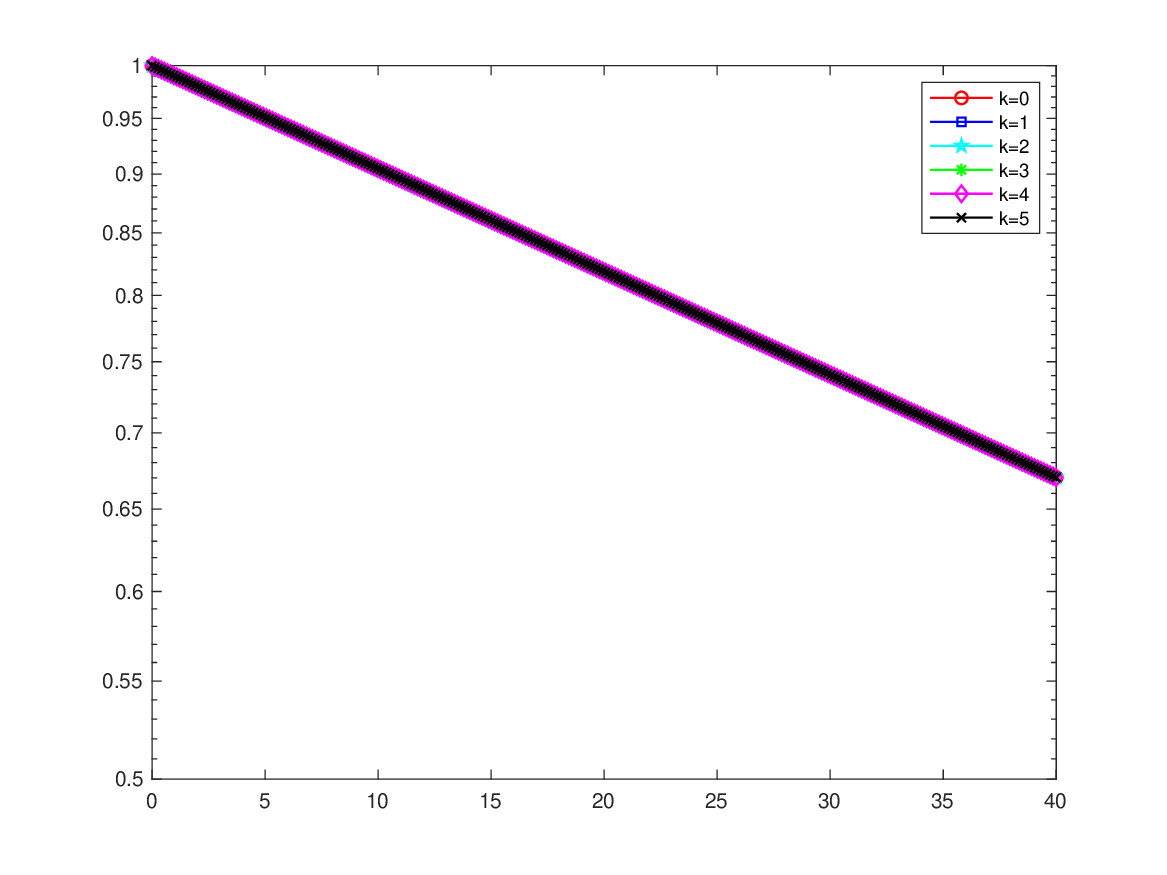}
  \caption{ETD-RK3, $\tau=\tau_0 \frac{d}{a^2}$}
 \end{subfigure}
 \begin{subfigure}[b]{0.3\textwidth}
  \includegraphics[width=\textwidth]{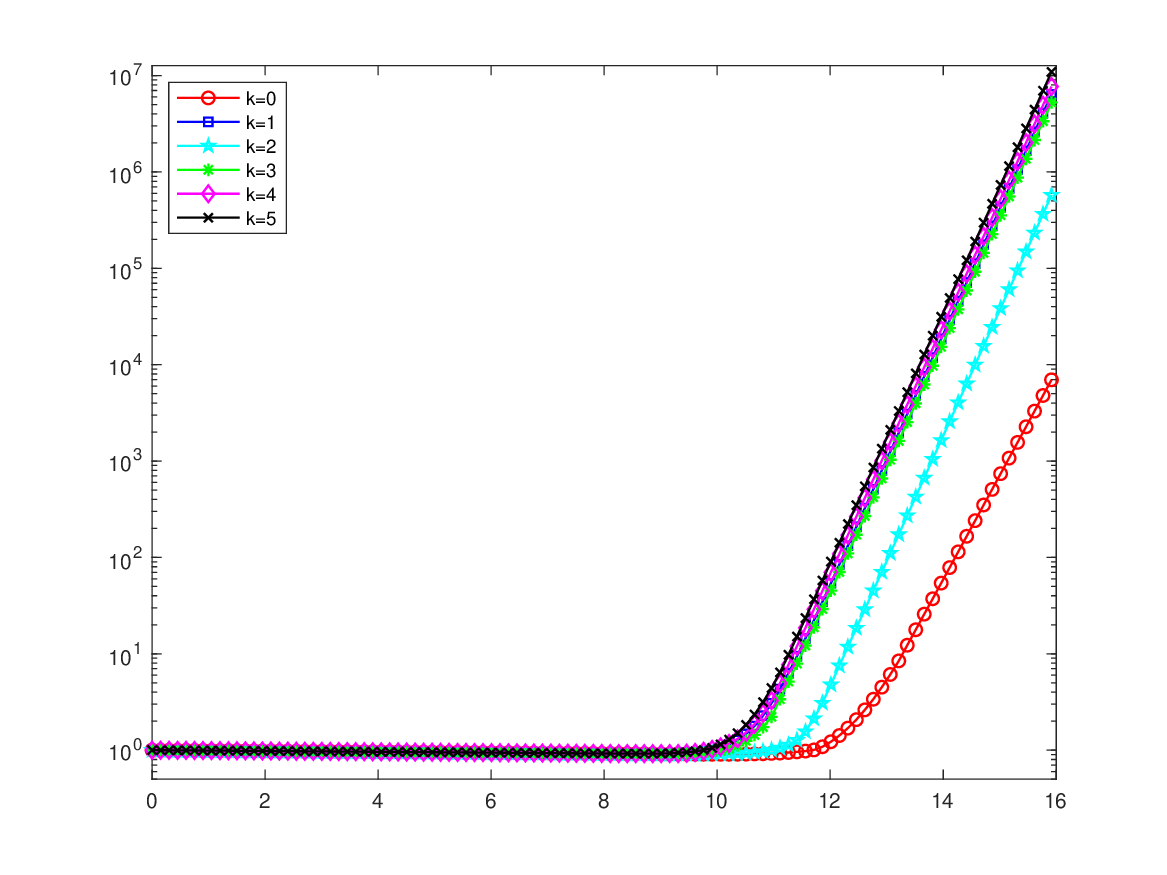}
  \caption{ETD-RK3, $\tau=1.1\times\tau_0 \frac{d}{a^2}$}
 \end{subfigure}

 \begin{subfigure}[b]{0.3\textwidth}
  \includegraphics[width=\textwidth]{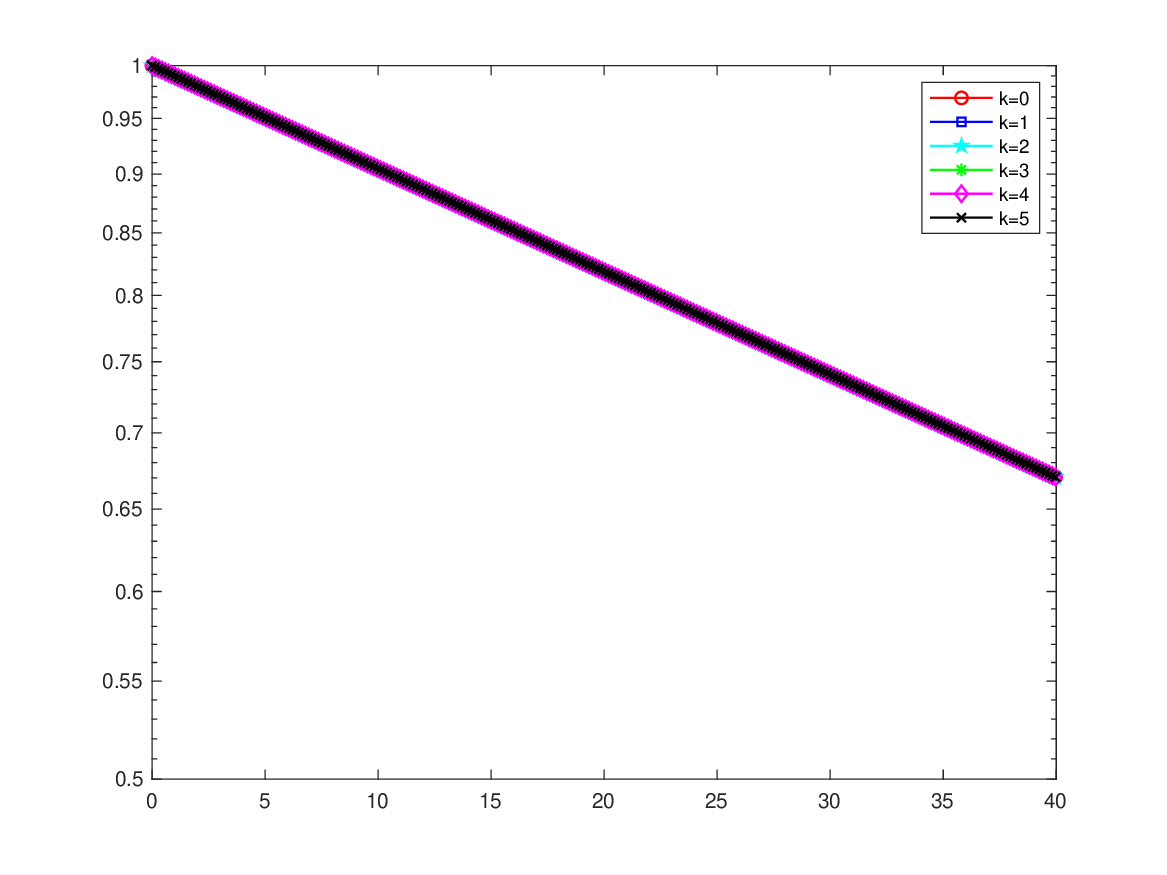}
  \caption{ETD-RK4, $\tau=\tau_0 \frac{d}{a^2}$}
 \end{subfigure}
 \begin{subfigure}[b]{0.3\textwidth}
  \includegraphics[width=\textwidth]{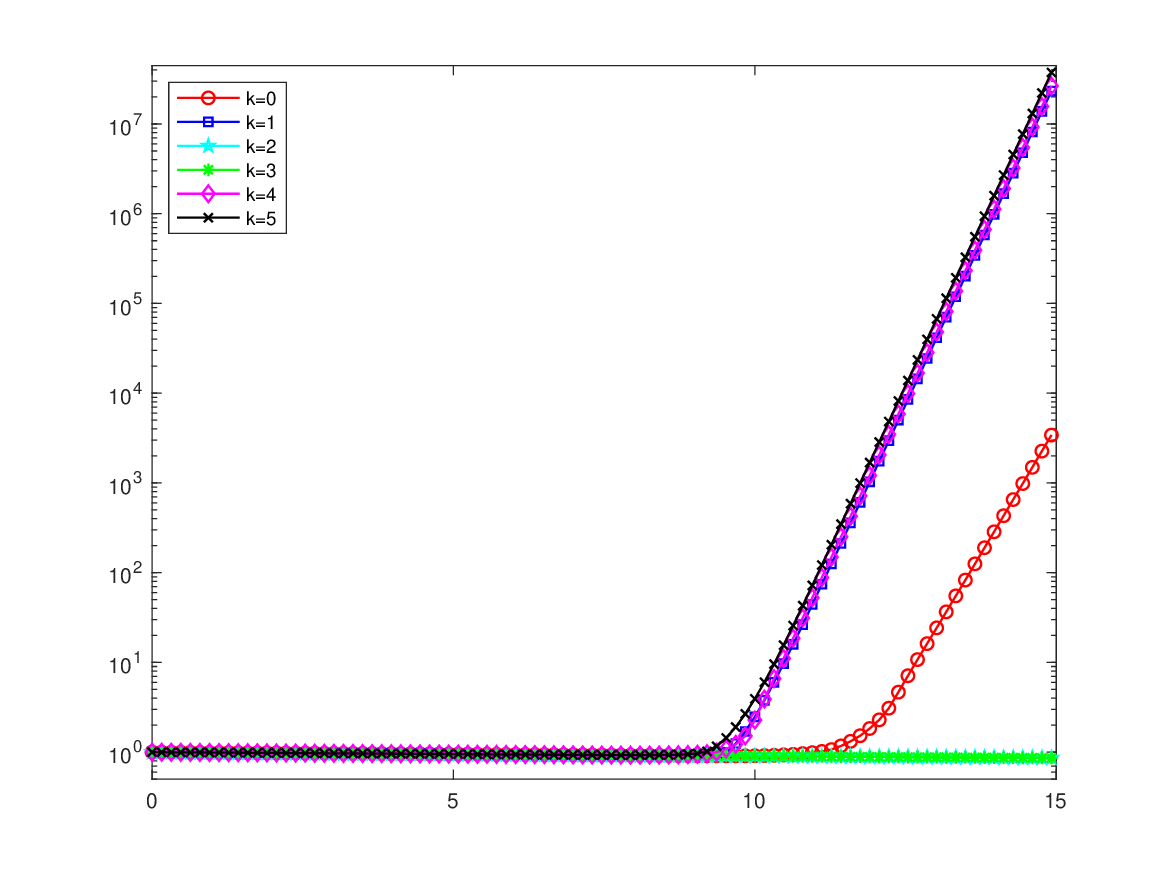}
  \caption{ETD-RK4, $\tau=1.1\times\tau_0 \frac{d}{a^2}$}
 \end{subfigure}
 \caption{\textbf{Example \ref{ex:stability}. Stability test.} 
 The growth of the maximum norm of the solutions for the advection-dominated problem over time using different ETD-RKDG methods.
 A uniform mesh with $h=\frac{\pi}{1000}$ is employed.
The time-step sizes are set to the stable values $\tau=\tau_0\frac{d}{a^2}$ in the left column and the unstable values $\tau=1.1\times\tau_0 \frac{d}{a^2}$ in the right column, where $\tau_0 = 2, 3.93, 4.55$, and $4.81$ for ETD-RK1, ETD-RK2, ETD-RK3, and ETD-RK4, respectively.}
 \label{fig:stability}
\end{figure}

\begin{exmp}\textbf{$h$-$p$ variation test}\label{ex:adaptive}
\end{exmp}
As observed in the previous analysis and experiments, the time-step constraints for the stability of ETD-RKDG methods are independent of the mesh size \(h\) and  the polynomial degree \(k\).
Therefore, the $h$-$p$ adaptive scenario, involving locally refined meshes or enhanced polynomial degrees, may be particularly suitable for the ETD-RKDG methods.
To explore the potential of the ETD-RK methods in these aspects, we test the stability of the ETD-RKDG methods using nonuniform meshes and nonuniform polynomial degrees in this example.
The same advection-dominated problem studied in Example \ref{ex:stability} is considered in this test.

We first solve the problem on a nonuniform mesh with uniform polynomial degrees. 
The mesh contains $N=2000$ cells with mesh sizes $\Delta x_{2m-1}:\Delta x_{2m}=1:9$, for $m=1,2,\ldots,1000$.
The time-step sizes are set to the critical values $\tau=\tau_0\frac{d}{a^2}$, with $\tau_0 = 2, 3.93, 4.55$, and $4.81$ for ETD-RK1, ETD-RK2, ETD-RK3, and ETD-RK4, respectively.
These time-step sizes are significantly larger than the minimal mesh size.

Next, we solve the problem on a uniform mesh with $N=2000$ and nonuniform polynomial degrees.
Polynomials of degrees $p_{2m-1}:p_{2m}=r:k$ are used on cells $I_{2m-1}$ and $I_{2m}$, respectively, for $m=1,2,\ldots,1000$ in the ETD-RK$r$ methods, where $r=1,2,3,4$ and $k=0,1,\ldots, 5$.
The time-step sizes are again set to the critical values $\tau=\tau_0\frac{d}{a^2}$, with $\tau_0 = 2, 3.93, 4.55$, and $4.81$ for ETD-RK1, ETD-RK2, ETD-RK3, and ETD-RK4, respectively.

The growth of the maximum norm of the numerical solutions over time is presented in log scale in Figure \ref{fig:h_adaptive} and Figure \ref{fig:p_adaptive}, for the nonuniform mesh and nonuniform polynomial degrees cases, respectively.
The figures also include results for the unstable counterpart with time-step sizes $\tau=1.1\times\tau_0\frac{d}{a^2}$.
From the figures, we observe that the numerical solutions are stable with the time-step sizes $\tau=\tau_0\frac{d}{a^2}$, regardless of the nonuniformity of the mesh sizes or polynomial degrees, while instability is evident with a $10\%$ increase in the time-step sizes.
\begin{figure}[!htbp]
 \centering
 \begin{subfigure}[b]{0.3\textwidth}
  \includegraphics[width=\textwidth]{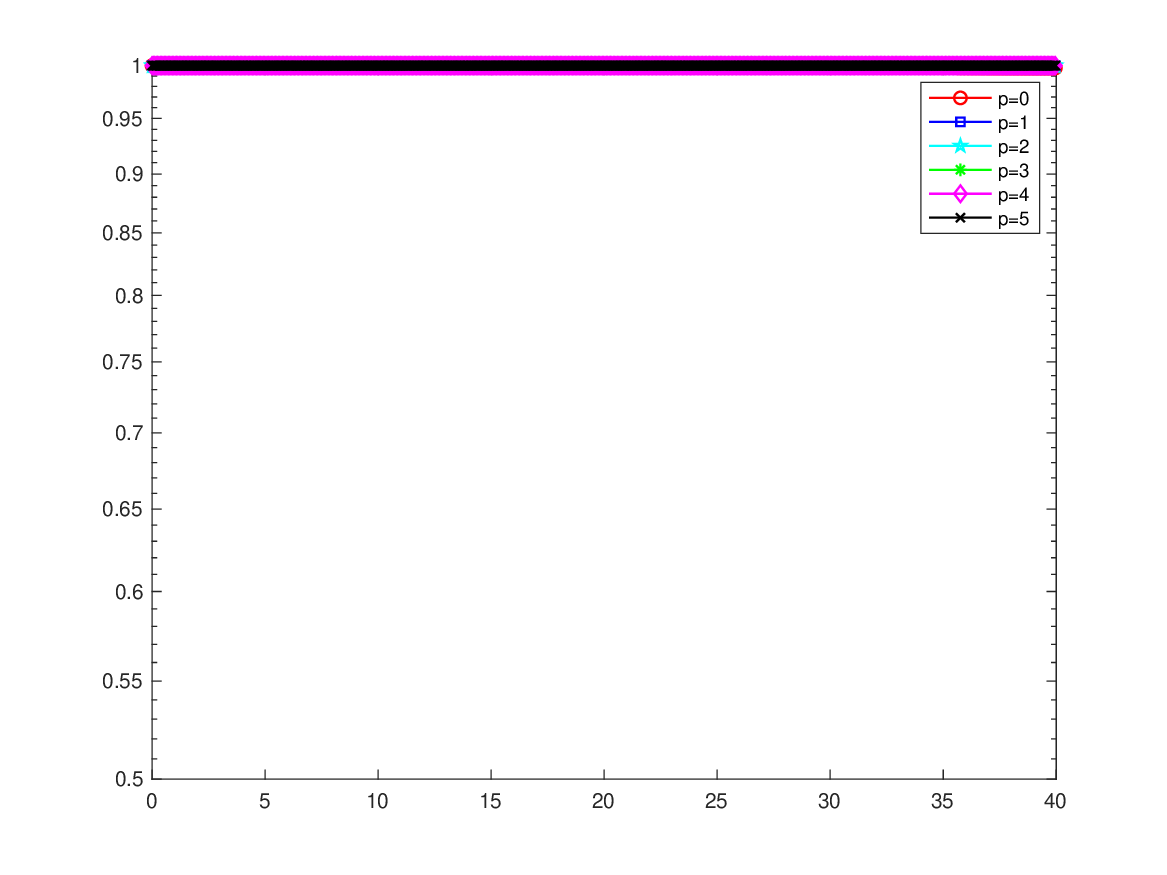}
  \caption{ETD-RK1, $\tau=\tau_0\frac{d}{a^2}$}
 \end{subfigure}
 \begin{subfigure}[b]{0.3\textwidth}
  \includegraphics[width=\textwidth]{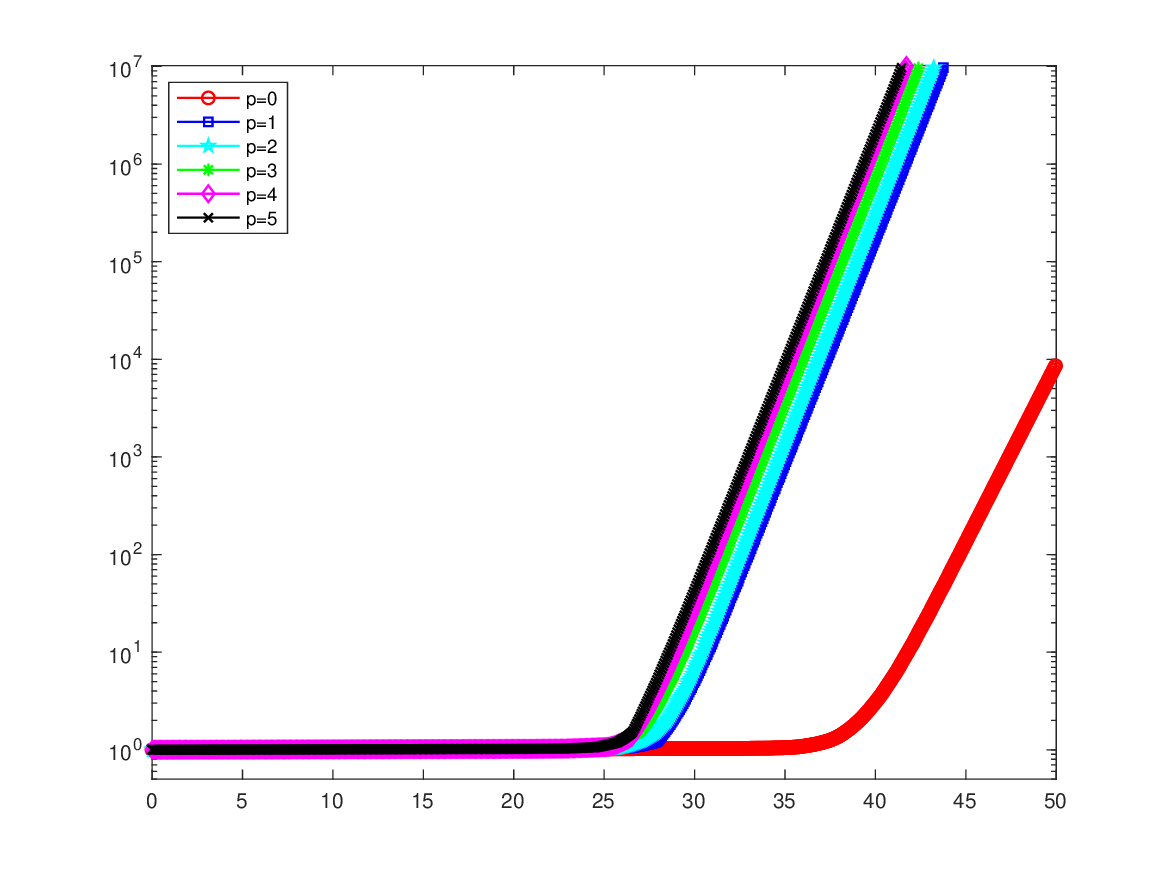}
  \caption{ETD-RK1, $\tau=1.1\times\tau_0\frac{d}{a^2}$}
 \end{subfigure}

 \begin{subfigure}[b]{0.3\textwidth}
  \includegraphics[width=\textwidth]{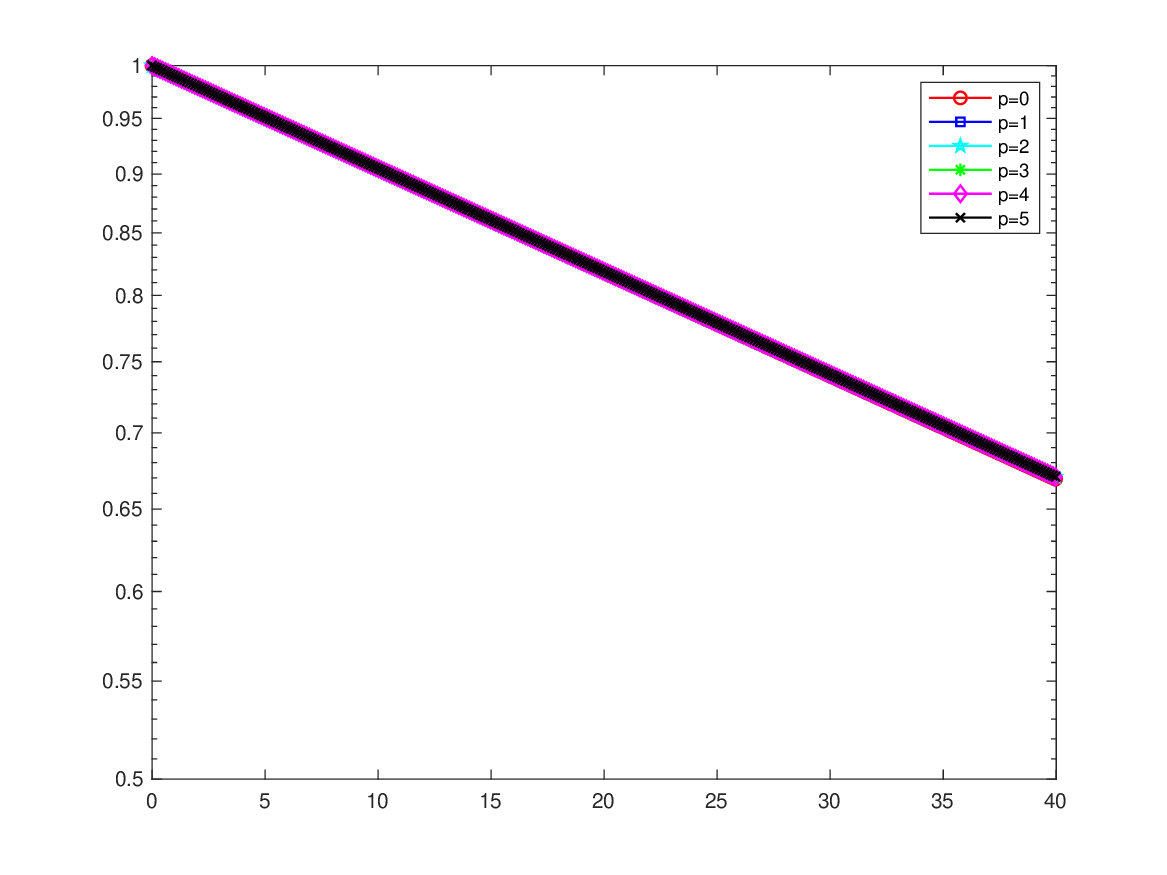}
  \caption{ETD-RK2, $\tau=\tau_0\frac{d}{a^2}$}
 \end{subfigure}
 \begin{subfigure}[b]{0.3\textwidth}
  \includegraphics[width=\textwidth]{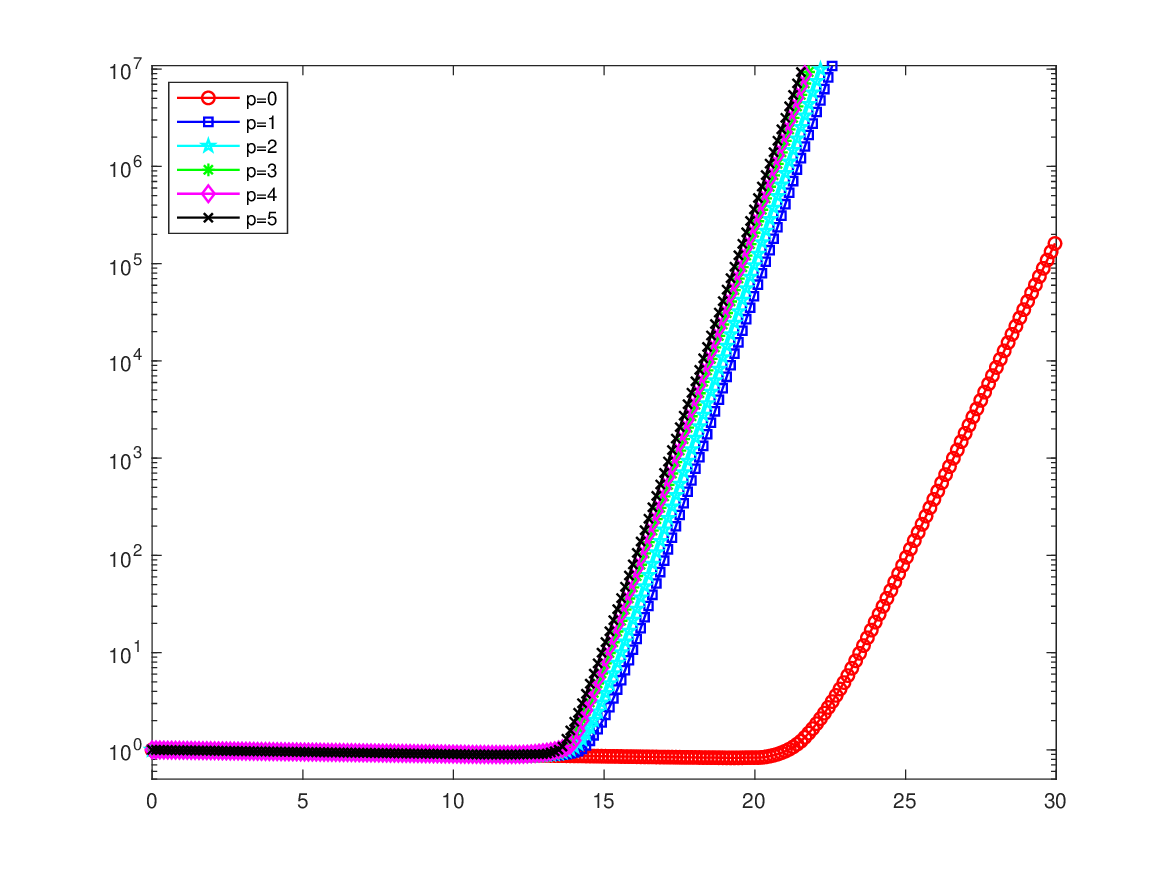}
  \caption{ETD-RK2, $\tau=1.1\times\tau_0\frac{d}{a^2}$}
 \end{subfigure}

 \begin{subfigure}[b]{0.3\textwidth}
  \includegraphics[width=\textwidth]{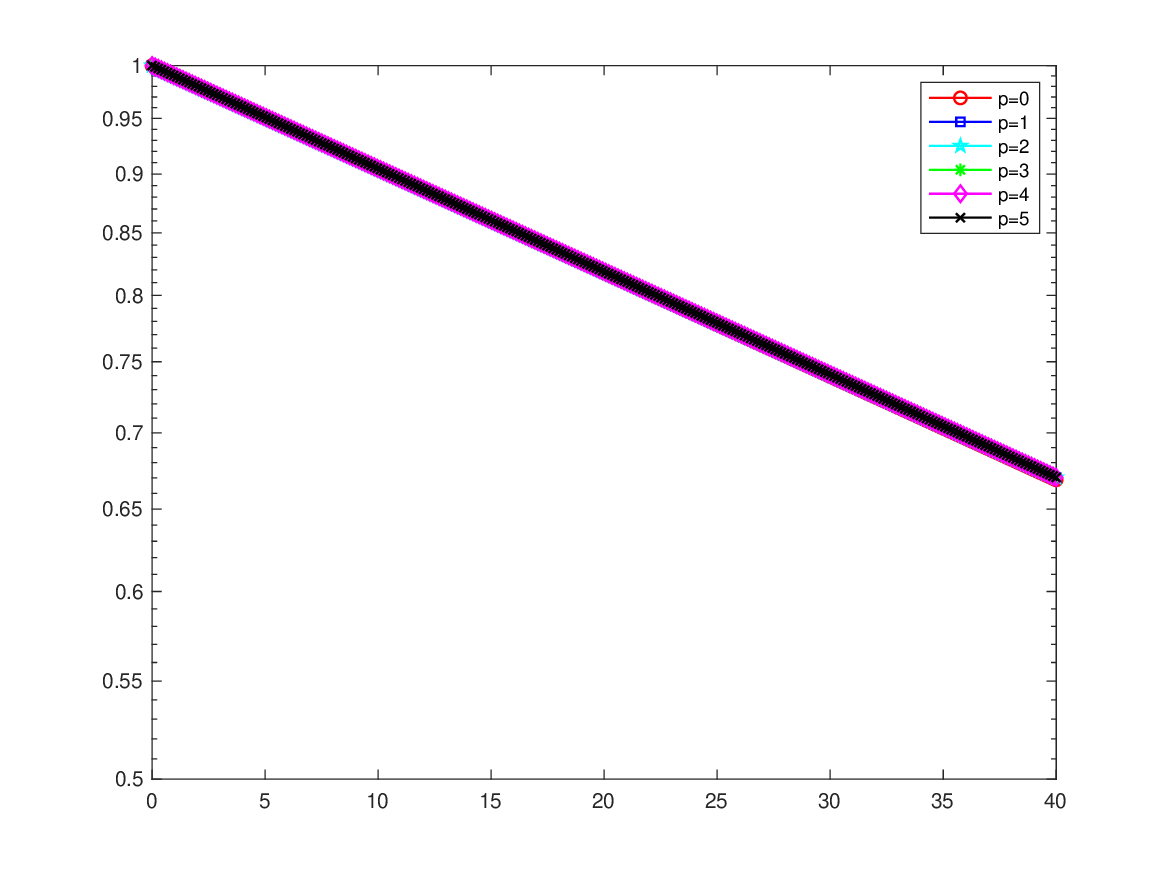}
  \caption{ETD-RK3, $\tau=\tau_0\frac{d}{a^2}$}
 \end{subfigure}
 \begin{subfigure}[b]{0.3\textwidth}
  \includegraphics[width=\textwidth]{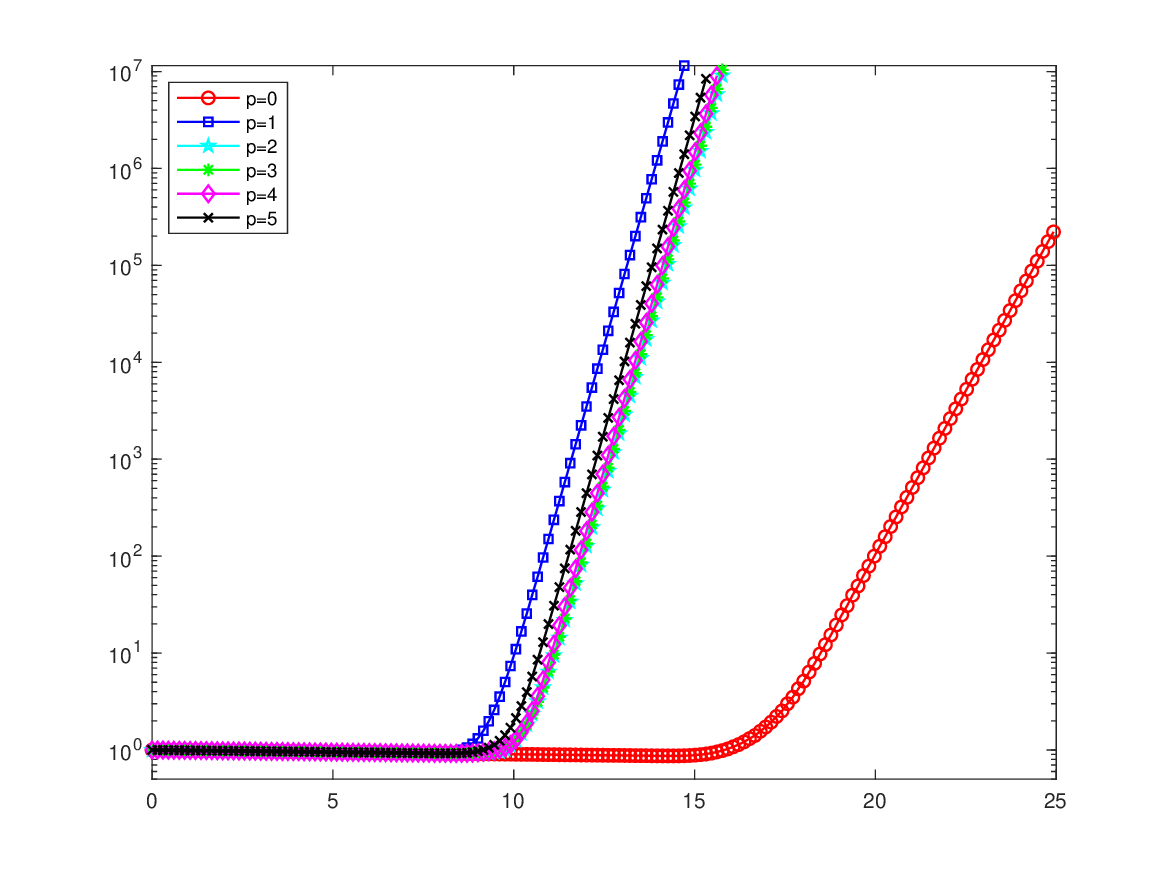}
  \caption{ETD-RK3, $\tau=1.1\times\tau_0\frac{d}{a^2}$}
 \end{subfigure}

 \begin{subfigure}[b]{0.3\textwidth}
  \includegraphics[width=\textwidth]{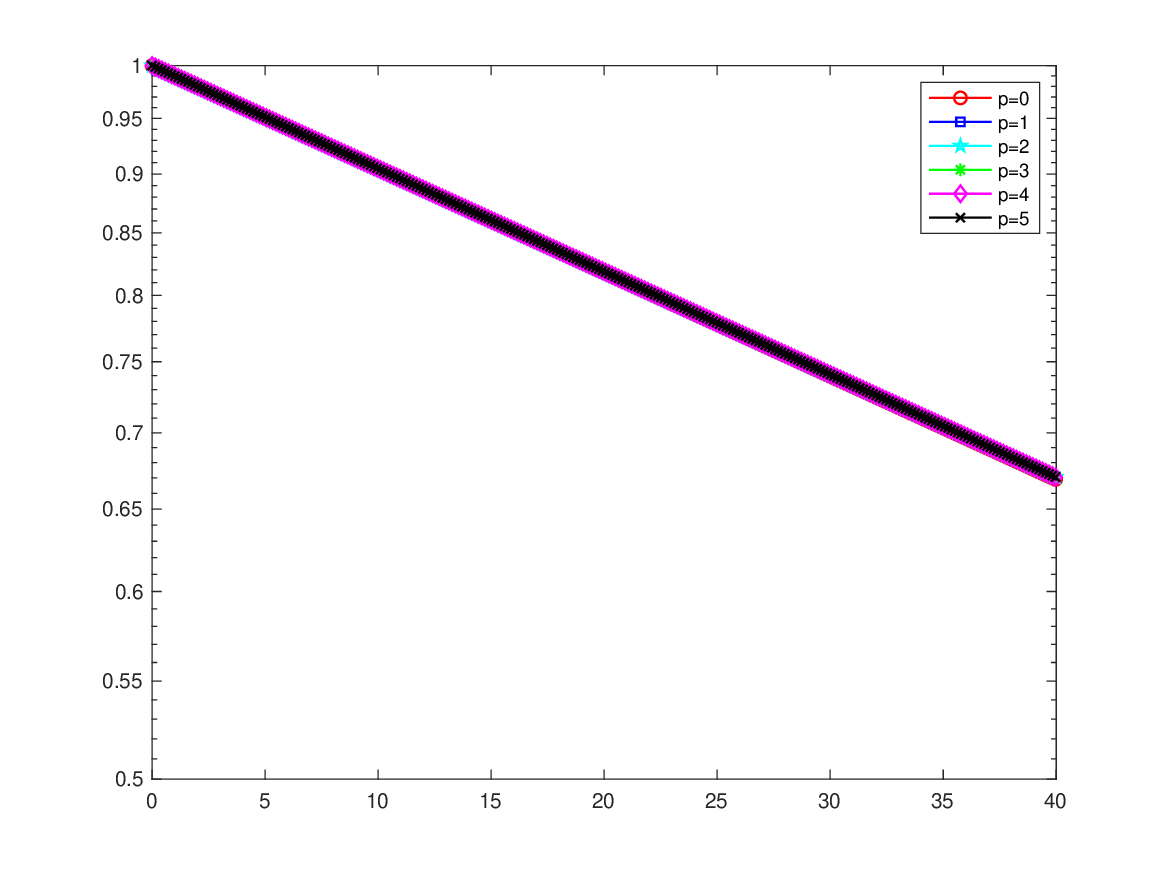}
  \caption{ETD-RK4, $\tau=\tau_0\frac{d}{a^2}$}
 \end{subfigure}
 \begin{subfigure}[b]{0.3\textwidth}
  \includegraphics[width=\textwidth]{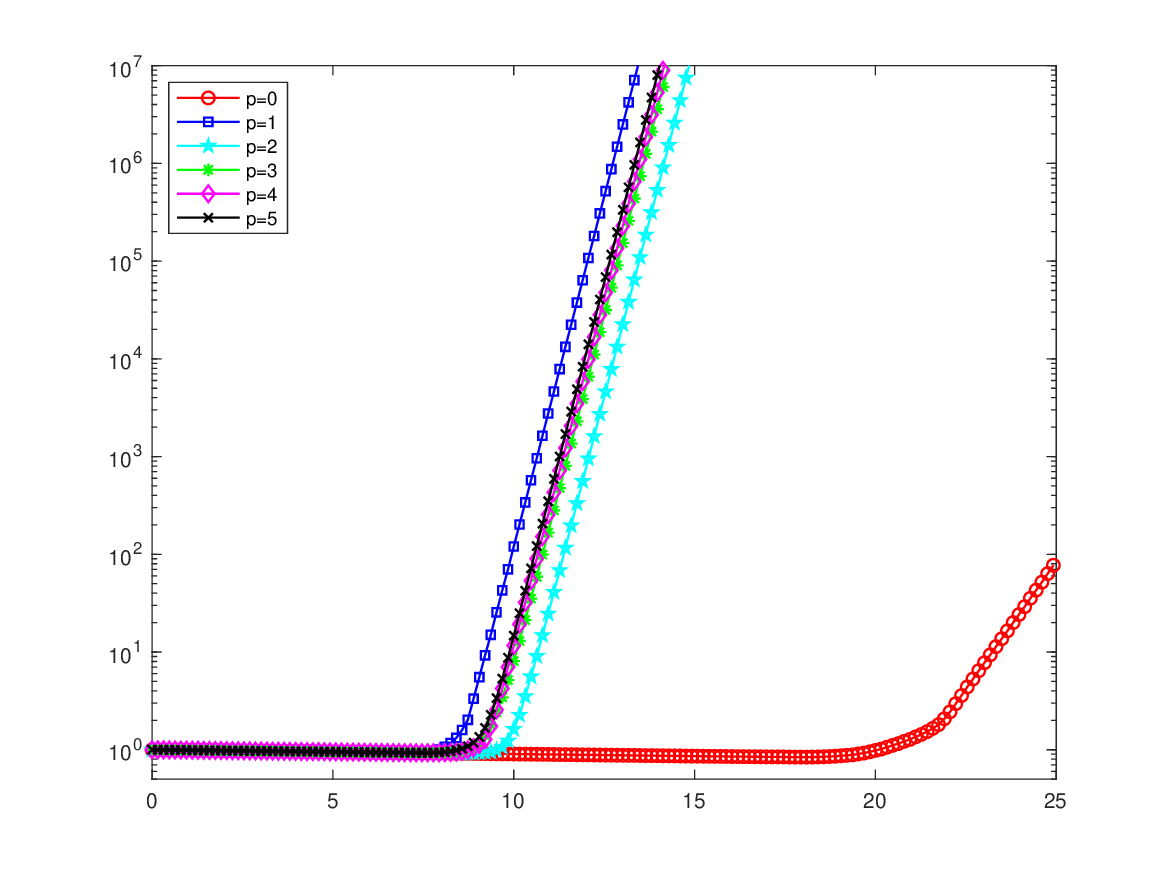}
  \caption{ETD-RK4, $\tau=1.1\times\tau_0\frac{d}{a^2}$}
 \end{subfigure}
 \caption{\textbf{Example \ref{ex:adaptive}. $h$ variation test.} 
The growth of the maximum norm of the solutions for the advection-dominated problem over time using different ETD-RKDG methods.
A nonuniform mesh with $\Delta x_{2m-1}:\Delta x_{2m}=1:9$ for $m=1,2,\ldots,1000$ is employed.
The time-step sizes are set to the stable values $\tau=\tau_0\frac{d}{a^2}$ in the left column and the unstable values $\tau=1.1\times\tau_0 \frac{d}{a^2}$ in the right column, where $\tau_0 = 2, 3.93, 4.55$, and $4.81$ for ETD-RK1, ETD-RK2, ETD-RK3, and ETD-RK4, respectively.}
 \label{fig:h_adaptive}
\end{figure}

\begin{figure}[!htbp]
 \centering
 \begin{subfigure}[b]{0.3\textwidth}
  \includegraphics[width=\textwidth]{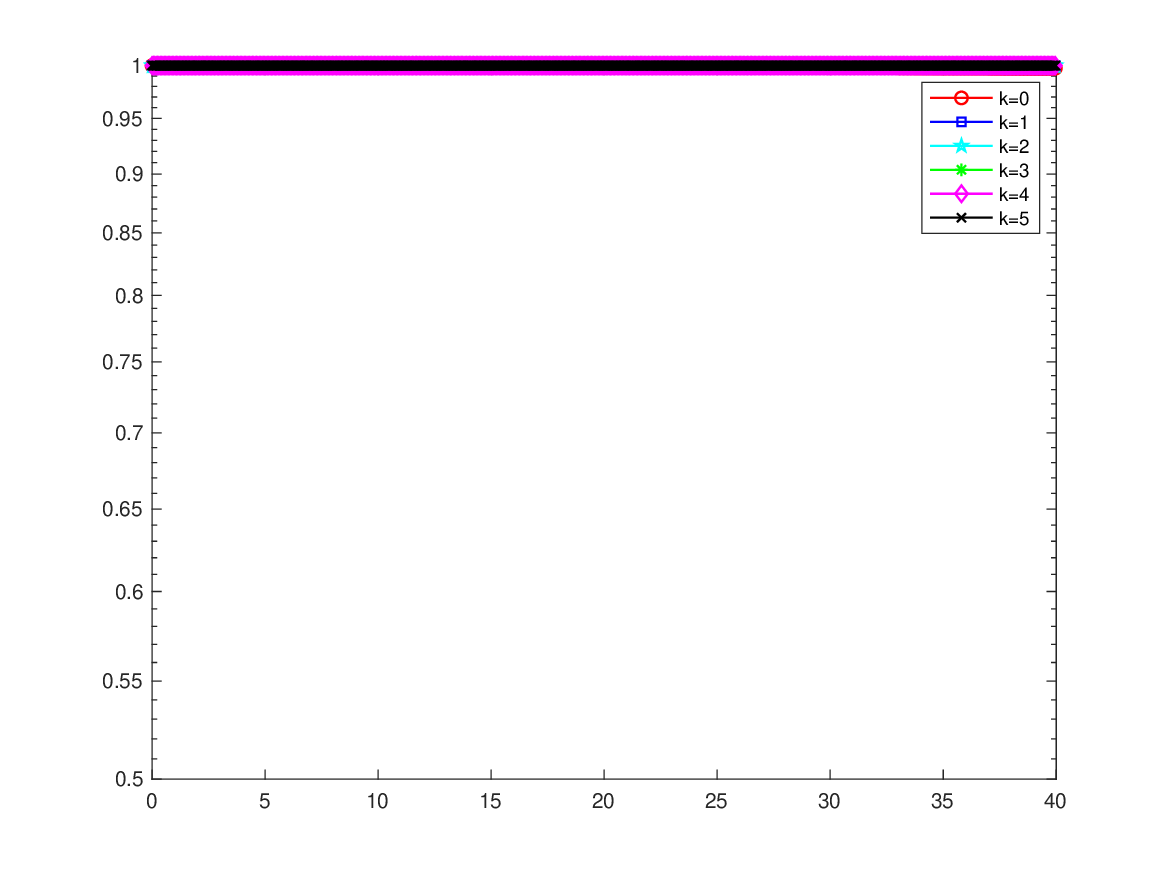}
  \caption{ETD-RK1, $\tau=\tau_0\frac{d}{a^2}$}
 \end{subfigure}
 \begin{subfigure}[b]{0.3\textwidth}
  \includegraphics[width=\textwidth]{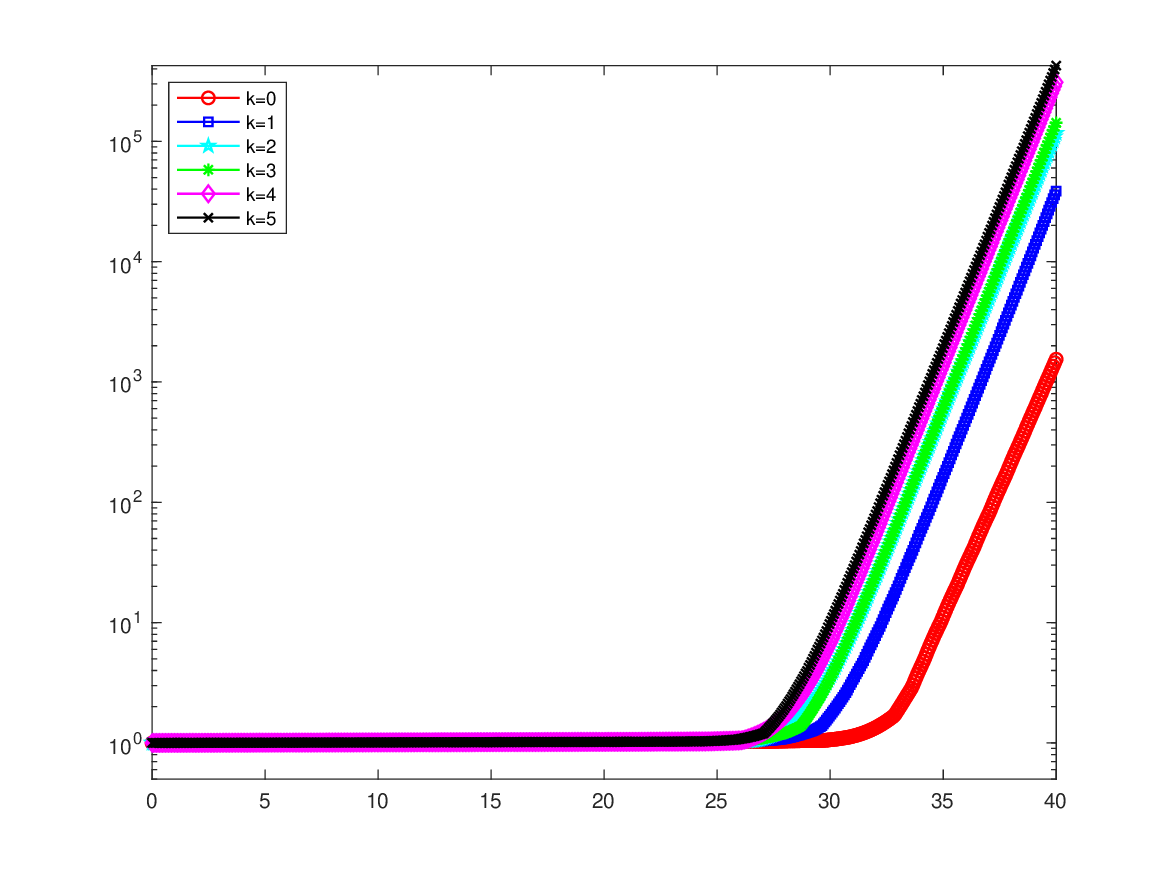}
  \caption{ETD-RK1, $\tau=1.1\times\tau_0 \frac{d}{a^2}$}
 \end{subfigure}

 \begin{subfigure}[b]{0.3\textwidth}
  \includegraphics[width=\textwidth]{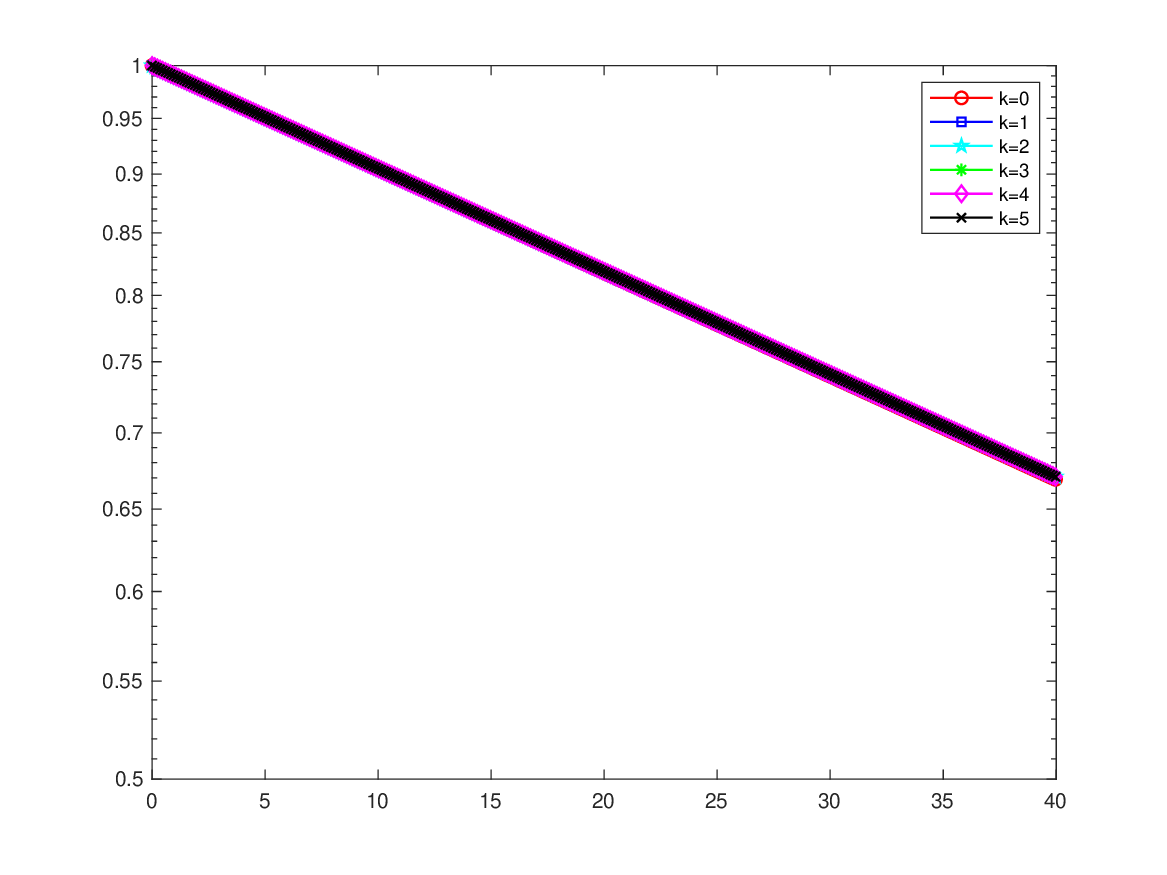}
  \caption{ETD-RK2, $\tau=\tau_0 \frac{d}{a^2}$}
 \end{subfigure}
 \begin{subfigure}[b]{0.3\textwidth}
  \includegraphics[width=\textwidth]{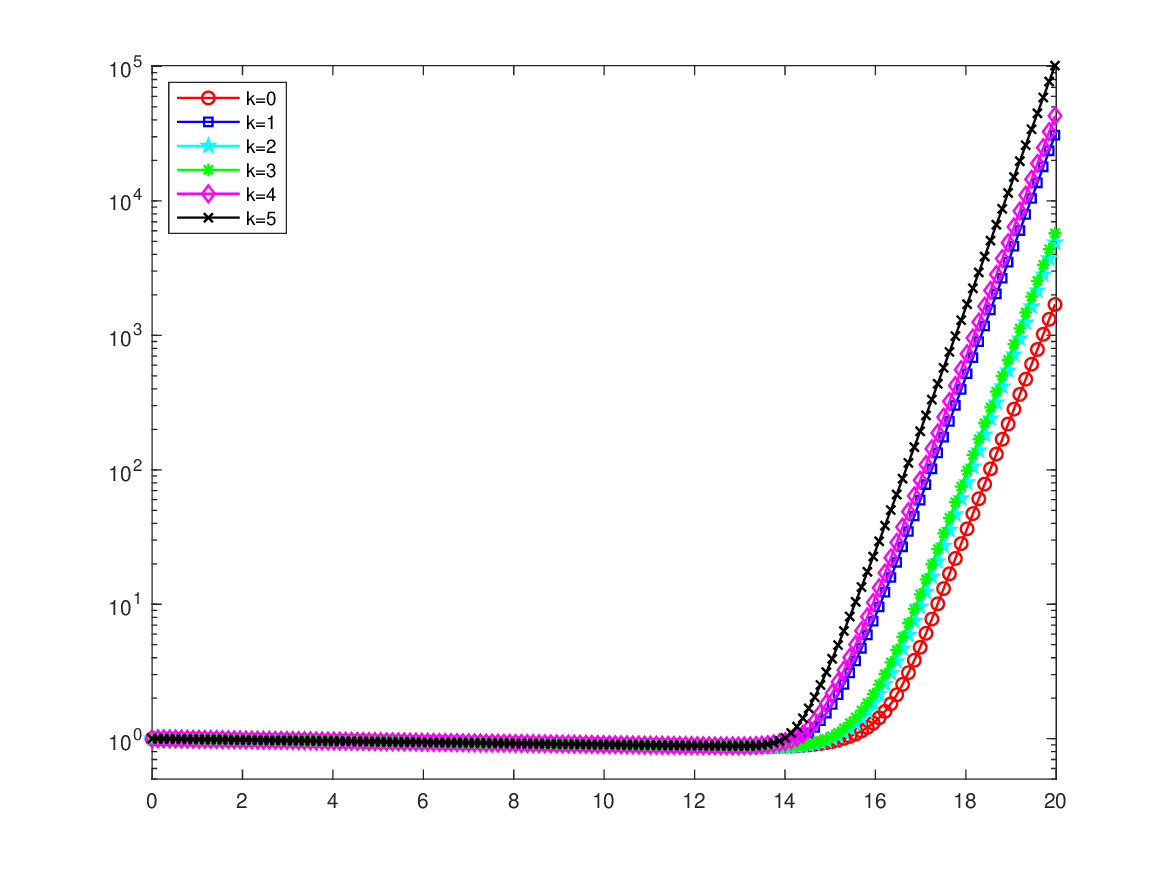}
  \caption{ETD-RK2, $\tau=1.1\times\tau_0 \frac{d}{a^2}$}
 \end{subfigure}

 \begin{subfigure}[b]{0.3\textwidth}
  \includegraphics[width=\textwidth]{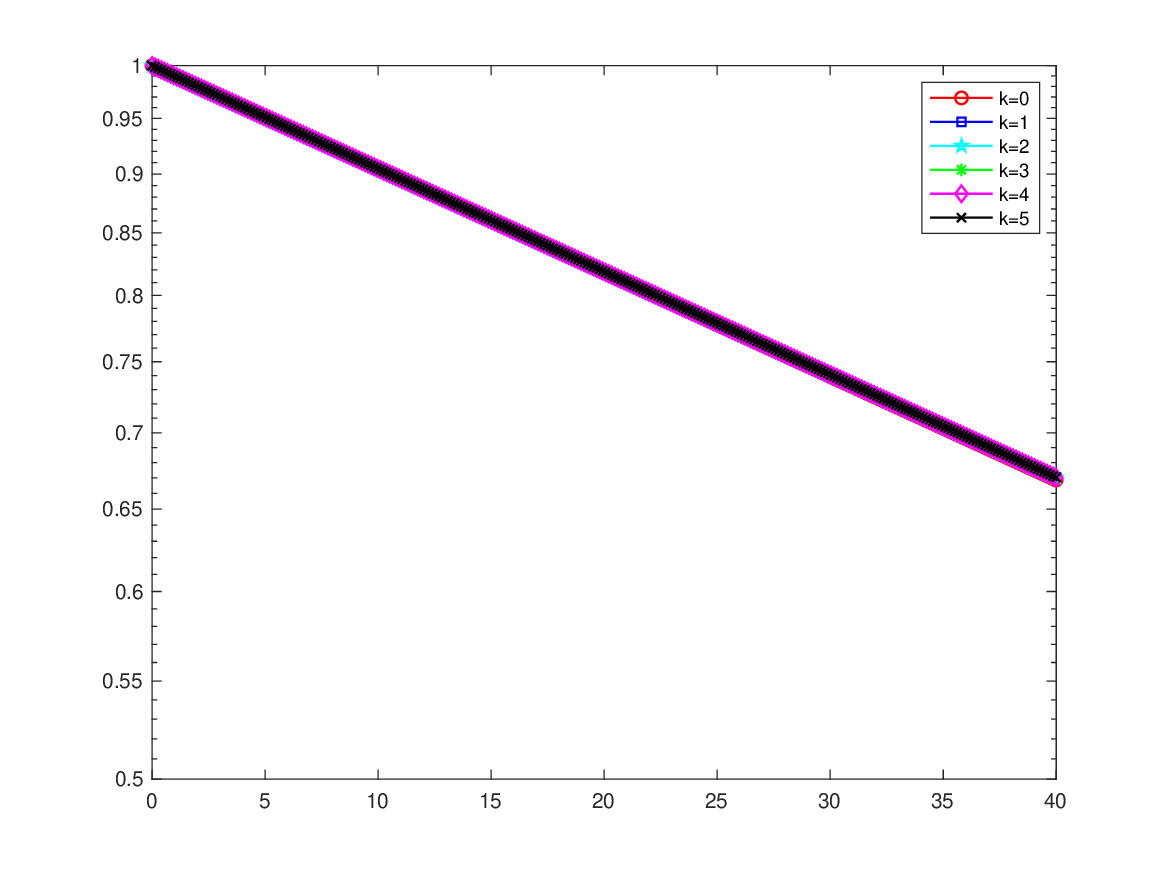}
  \caption{ETD-RK3, $\tau=\tau_0 \frac{d}{a^2}$}
 \end{subfigure}
 \begin{subfigure}[b]{0.3\textwidth}
  \includegraphics[width=\textwidth]{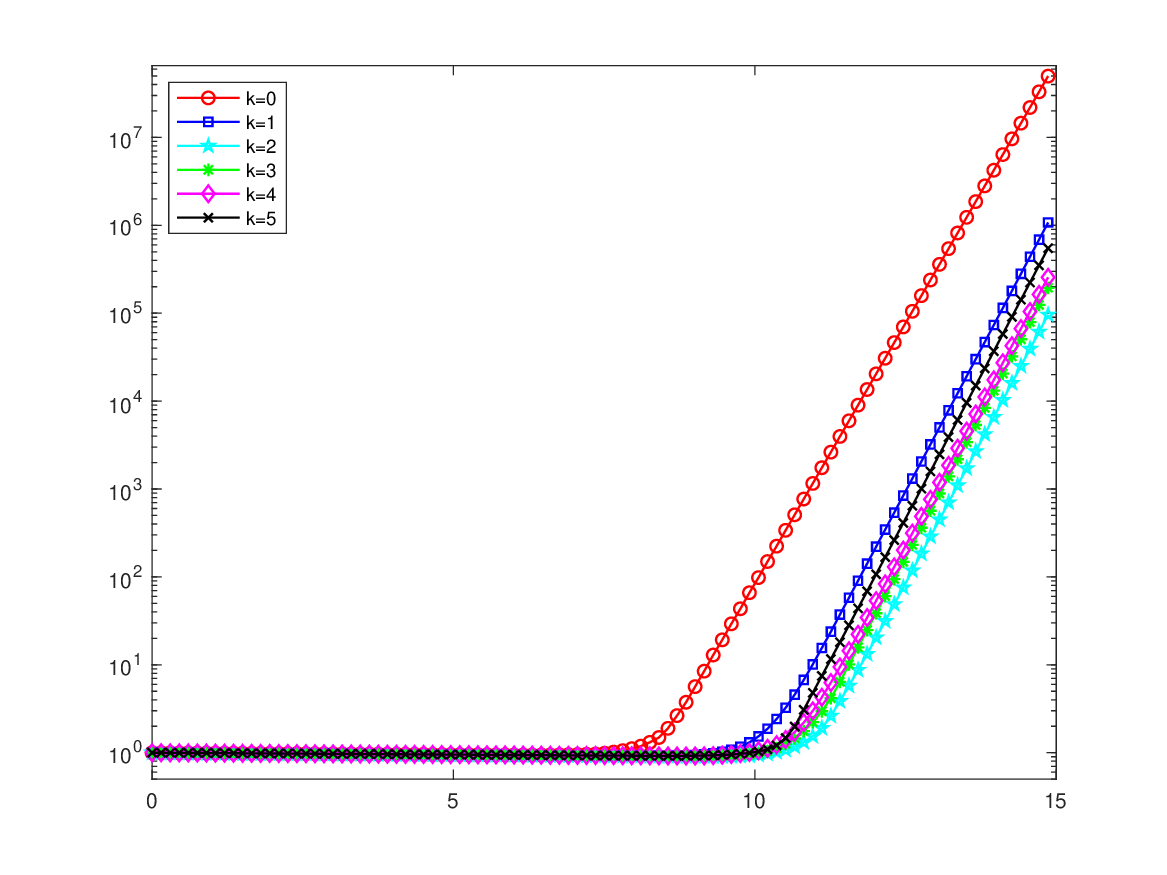}
  \caption{ETD-RK3, $\tau=1.1\times\tau_0 \frac{d}{a^2}$}
 \end{subfigure}

 \begin{subfigure}[b]{0.3\textwidth}
  \includegraphics[width=\textwidth]{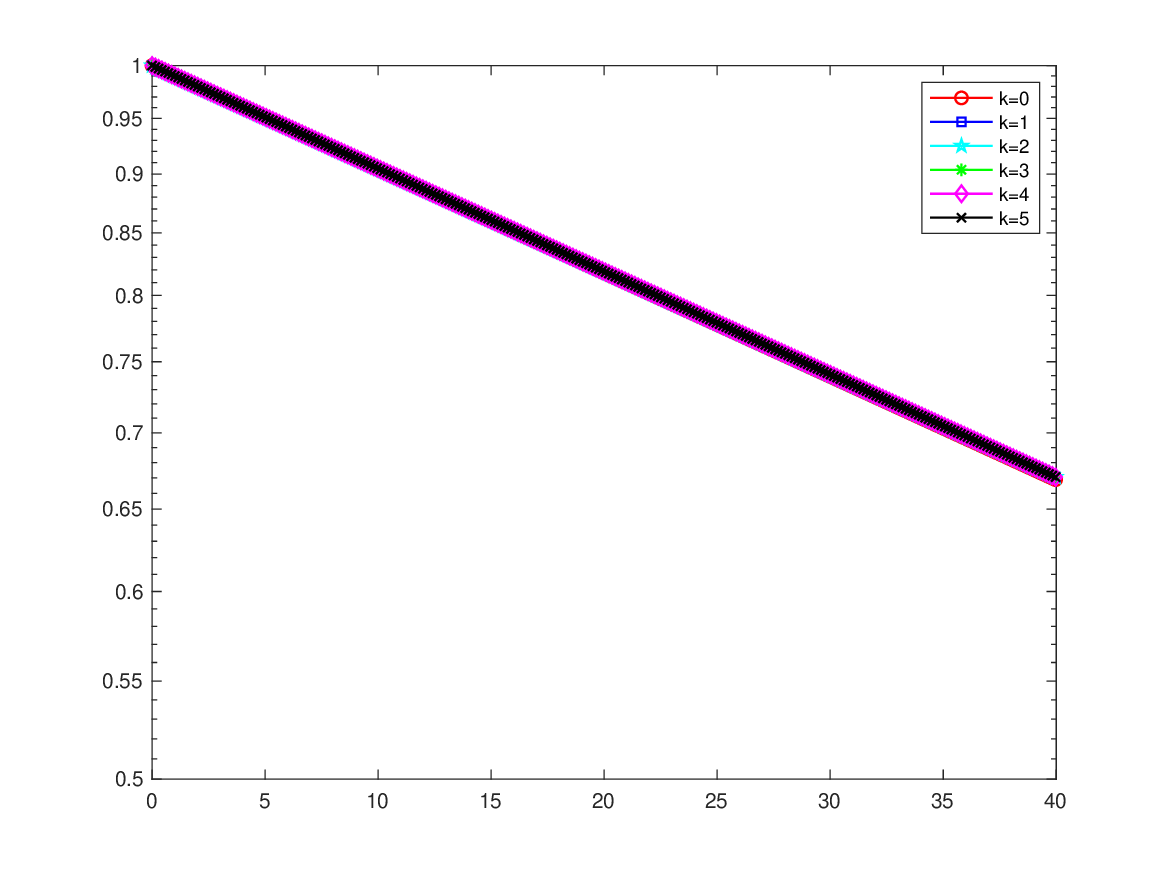}
  \caption{ETD-RK4, $\tau=\tau_0 \frac{d}{a^2}$}
 \end{subfigure}
 \begin{subfigure}[b]{0.3\textwidth}
  \includegraphics[width=\textwidth]{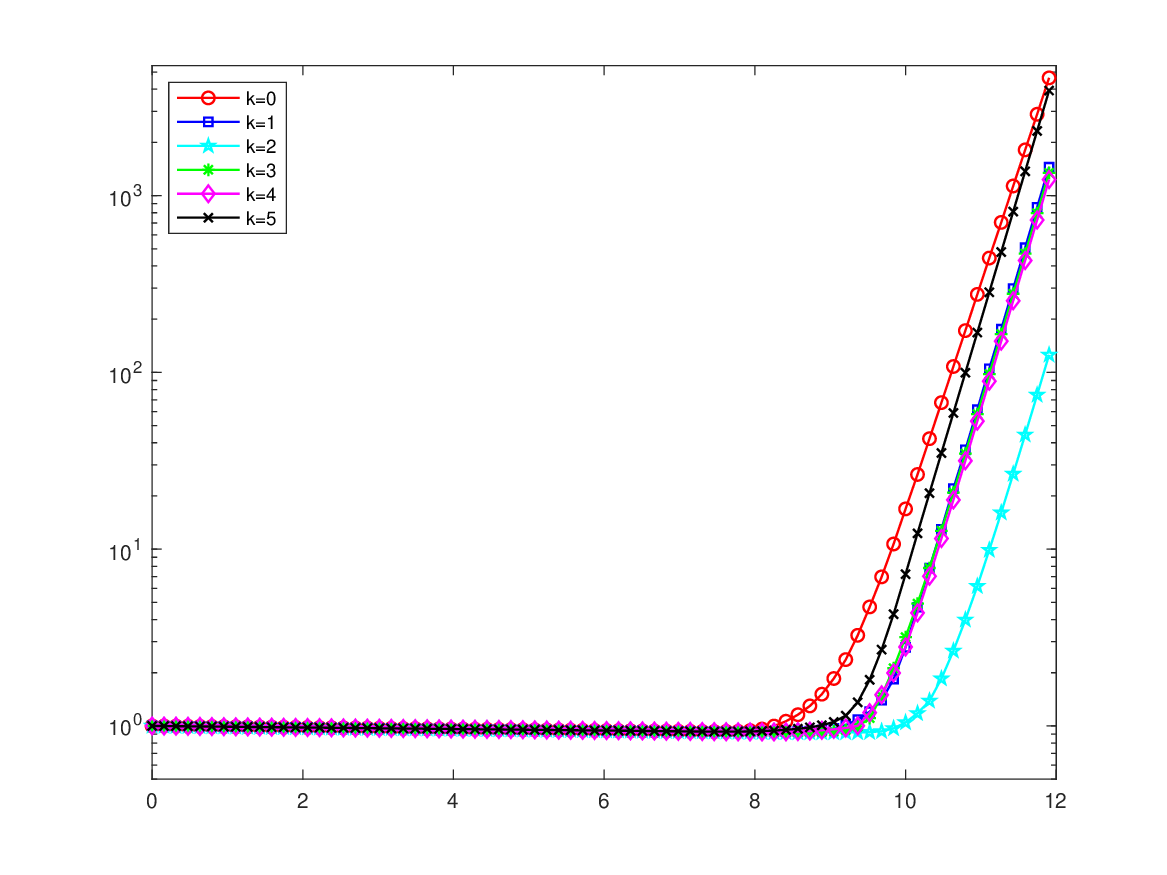}
  \caption{ETD-RK4, $\tau=1.1\times\tau_0 \frac{d}{a^2}$}
 \end{subfigure}
 \caption{\textbf{Example \ref{ex:adaptive}. $p$ variation test.} 
The growth of the maximum norm of the solutions for the advection-dominated problem over time using different ETD-RKDG methods.
A uniform mesh with $h=\frac{\pi}{1000}$ is employed. 
Nonuniform polynomials of degrees $r$ and $k$ are used on the cells $I_{2m-1}$ and $I_{2m}$, respectively, for $ m=1,2,\ldots,1000$ in the ETD-RK$r$ methods, where $r=1,2,3,4$ and $k=0,1,\ldots,5$. 
The time-step sizes are set to the stable values $\tau=\tau_0\frac{d}{a^2}$ in the left column and the unstable values $\tau=1.1\times\tau_0\frac{d}{a^2}$ in the right column, where $\tau_0 = 2, 3.93, 4.55$, and $4.81$ for ETD-RK1, ETD-RK2, ETD-RK3, and ETD-RK4, respectively.}
 \label{fig:p_adaptive}
\end{figure}

\begin{exmp}\textbf{Nonlinear equations test}\label{ex:Nonlinear}
\end{exmp}
Although the analysis in this paper focuses on linear cases \eqref{eq:AdvDiffEq}, we also test performance of the ETD-RKDG methods on nonlinear advection cases 
\begin{equation*}
u_t+f(u)_x=du_{xx}
\end{equation*}
to evaluate the generality of the conclusions obtained.
Two advection-dominated examples are tested, with the diffusion coefficient $d=0.01$, the computational domain $\Omega=[-1,1]$, and periodic boundary conditions.
The first example is the viscous Burgers equation with the convex flux function
\begin{equation*}
f(u)=\frac{u^2}{2},
\end{equation*}
and the smooth initial condition $u(x,0)=0.25+0.5\sin(\pi x)$.
The second example is the viscous Buckley--Leverett equation with the non-convex flux function
\begin{equation*}
f(u)=\frac{4u^2}{4u^2+(1-u)^2},
\end{equation*}
and the non-smooth initial condition $u=1$ in $[-\frac12,0]$ and $u=0$ elsewhere. 

Both examples are computed on a uniform grid with $N=2000$ using the ETD-RK4 and $\mathbb{P}^3$-DG method.
We adopt the same pattern for the time-step size, $\tau=\tau_0\frac{d}{a^2}$, as before, where $\tau_0=4.81$ for the ETD-RK4 method, $d=0.01$ is the diffusion coefficient, and $a=\max_{u}|f'(u)|$ represents the maximum wave speed.
Specifically, we take $a=\max_{u\in[-0.25, 0.75]}|f'(u)|=0.75$ for the Burgers equation and $a=\max_{u\in[0,1]}|f'(u)|=2.333$ for the Buckley--Leverett equation.

The results for the Burgers equation at $T=2$ and the Buckley--Leverett equation at $T=0.4$ are presented in Figure \ref{fig:Nonlinear}, alongside reference solutions obtained using the SSP-RK3 eighth-order multi-resolution A-WENO scheme \cite{jiang2021high} on the same grid with a sufficiently small time-step size.
As shown in the figure, even though the flux functions are nonlinear or even non-convex and the solutions contain large gradients, the numerical results remain stable and align well with the reference solutions, despite the vary large time-step sizes $\frac{\tau}{h}\approx 85$ for the Burgers equation and $\frac{\tau}{h}\approx 9$ for the Buckley--Leverett equation.

\begin{figure}[!htbp]
 \centering
 \begin{subfigure}[b]{0.45\textwidth}
  \includegraphics[width=\textwidth]{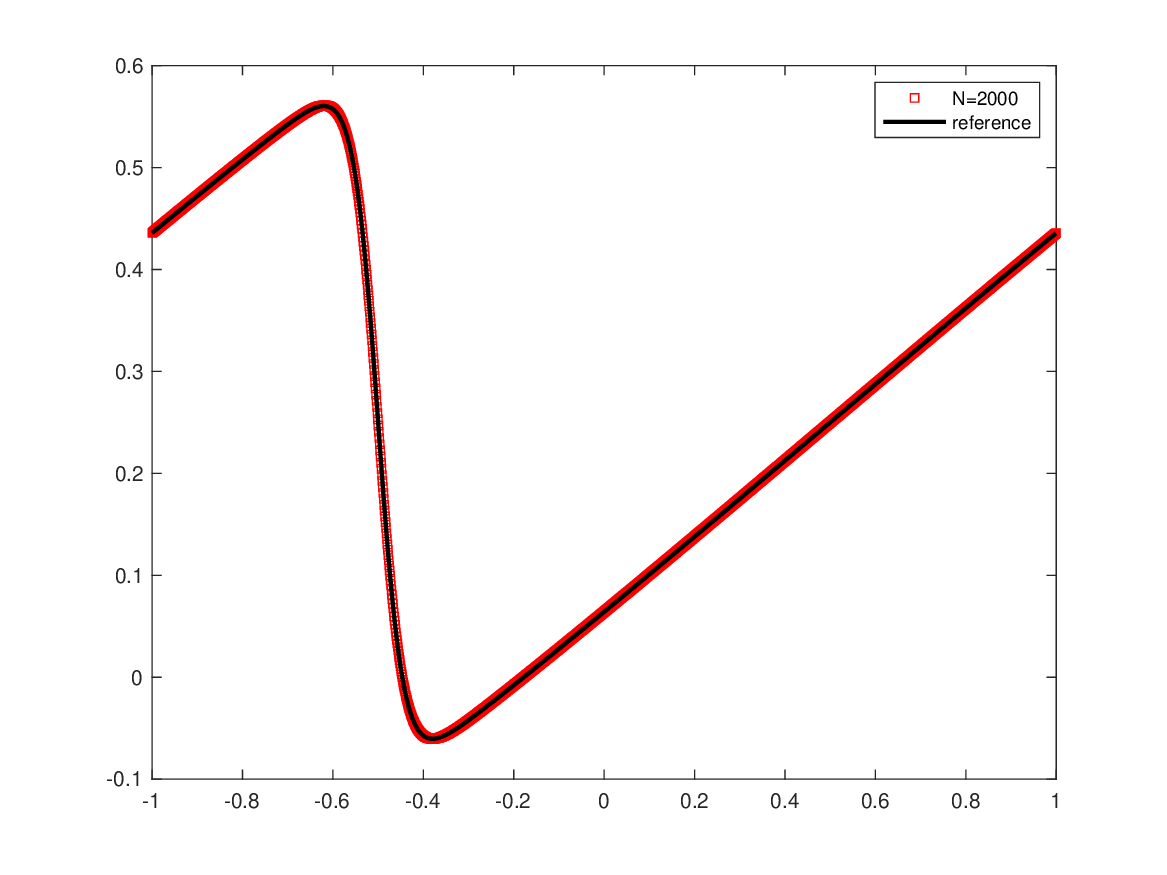}
  \caption{Burgers equation at $T=2$}
 \end{subfigure}
 \begin{subfigure}[b]{0.45\textwidth}
  \includegraphics[width=\textwidth]{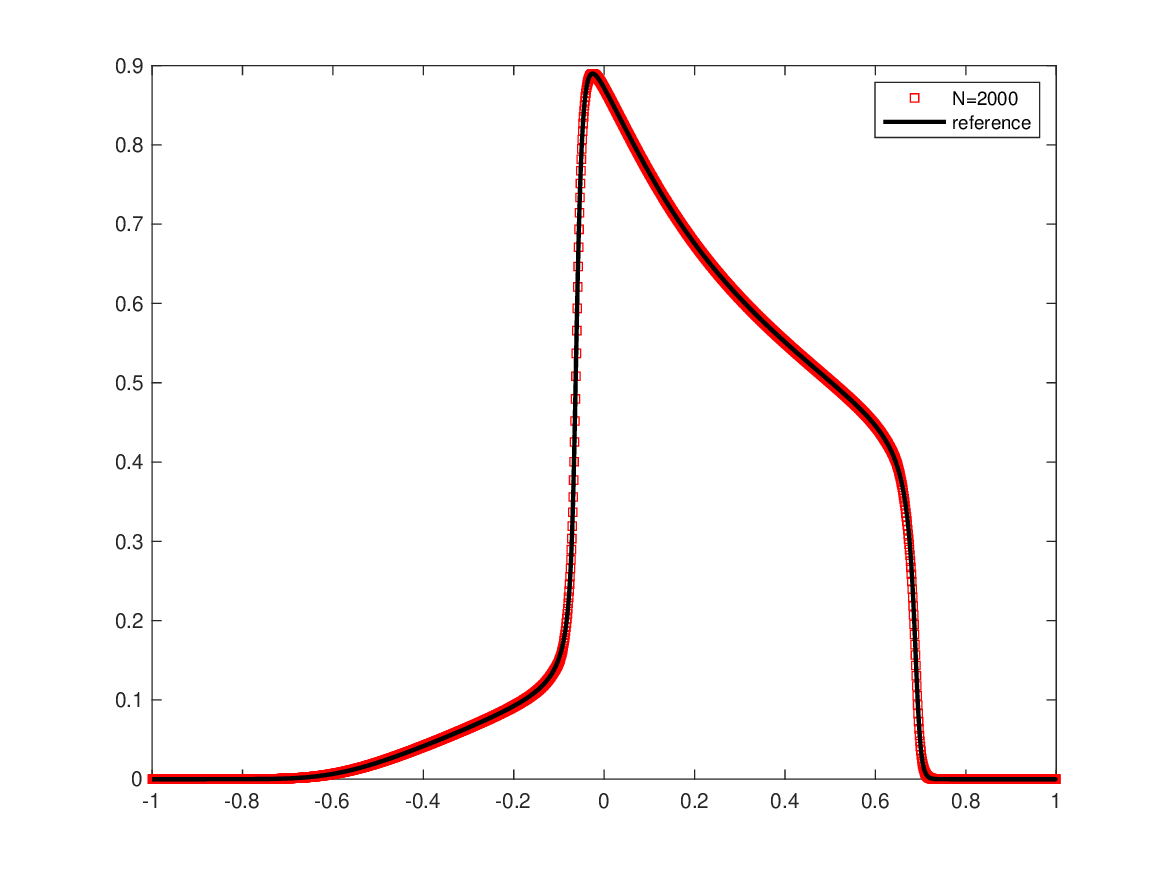}
  \caption{Buckley--Leverett equation at $T=0.4$}
 \end{subfigure}
 \caption{\textbf{Example \ref{ex:Nonlinear}. Nonlinear equations test.}
 Numerical results of the viscous Burgers and viscous Buckley--Leverett equations obtained from the ETD-RK4 and $\mathbb{P}^3$-DG method.}
 \label{fig:Nonlinear}
\end{figure}

\begin{exmp}\textbf{Two-dimensional test}\label{ex:BL2D}
\end{exmp}
In this example \cite{jiang2021high}, we test the ETD-RKDG methods on a two-dimensional viscous Buckley--Leverett equation
\begin{equation*}
u_t+f_1(u)_x+f_2(u)_y=d\Delta u,
\end{equation*}
where 
\begin{equation*}
f_1(u)=\frac{u^2}{u^2+(1-u)^2}\quand f_2(u)=\frac{u^2}{u^2+(1-u)^2}\left(1-5(1-u)^2\right)
\end{equation*}
are the non-convex fluxes representing, respectively, the directions without and with gravitational effects in two-phase porous media flows.
As before, we let $d=0.01$ so that the equation is advection-dominated.
The initial condition is given by
\begin{equation*}
u(x,y,0)=\begin{cases}
1,\quad x^2+y^2<\frac12,\\
0,\quad \text{otherwise},
\end{cases}
\end{equation*}
on the computational domain $\Omega=[-1.5,1.5]^2$ with periodic boundary conditions.

The example is computed on a uniform grid with $N\times M=600\times600$ square cells using the ETD-RK4 and $\mathbb{Q}^3$-DG method.
We adopt the time-step size $\tau=\tau_0\frac{d}{a^2}$ for this two-dimensional case, where $\tau_0=4.81$ is the constant identified for the ETD-RK4 method, $d=0.01$ is the diffusion coefficient, and $a=\sup_{u\in[0,1]}\sqrt{f_1'(u)^2+f_2'(u)^2}=\sqrt{13.37}$ represents the maximum wave speed in the domain.

The result at $T=0.5$ is shown in Figure \ref{fig:BL2D}, demonstrating stability and good agreement with the results of the well-established SSP-RK3 A-WENO method in \cite{jiang2021high}, with a significantly larger time-step size.

\begin{figure}[!htbp]
 \centering
 \begin{subfigure}[b]{0.45\textwidth}
  \includegraphics[width=\textwidth]{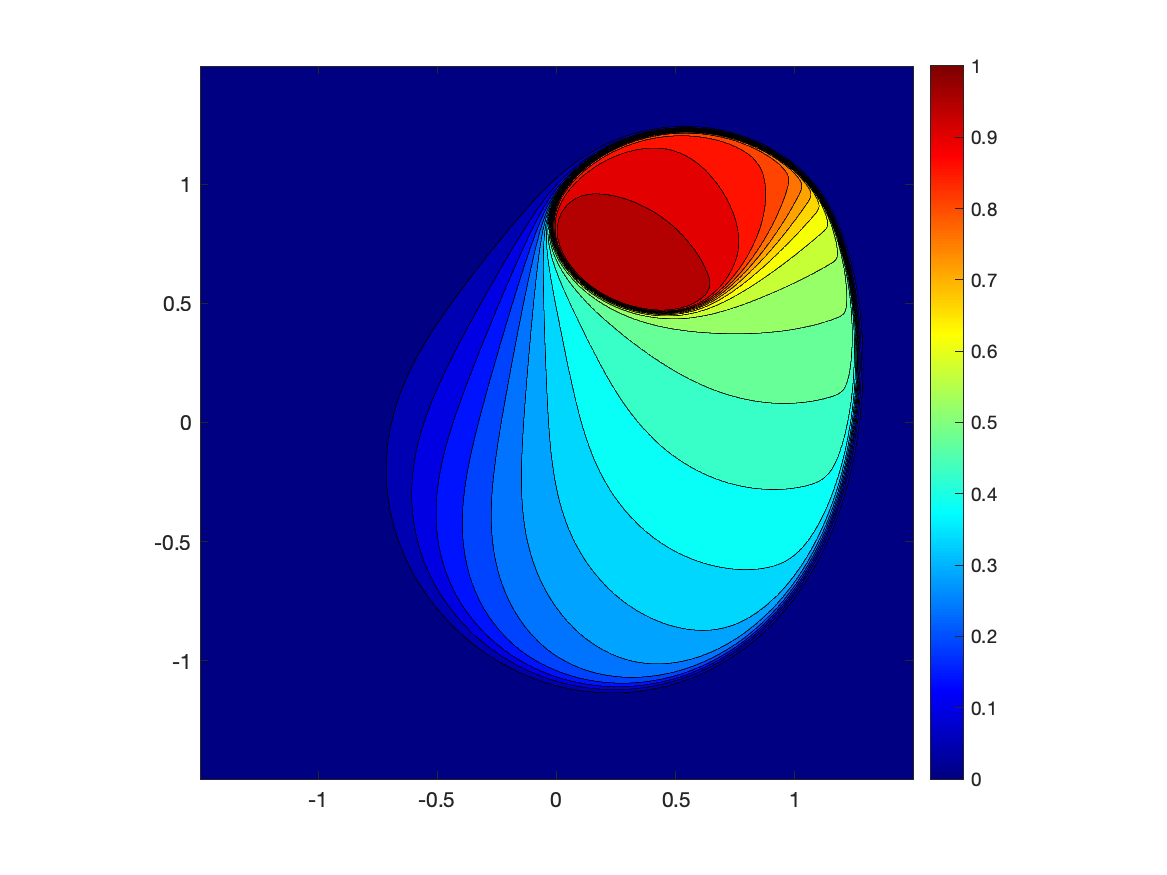}
  \caption{Contour at $T=0.5$}
 \end{subfigure}
 \begin{subfigure}[b]{0.45\textwidth}
  \includegraphics[width=\textwidth]{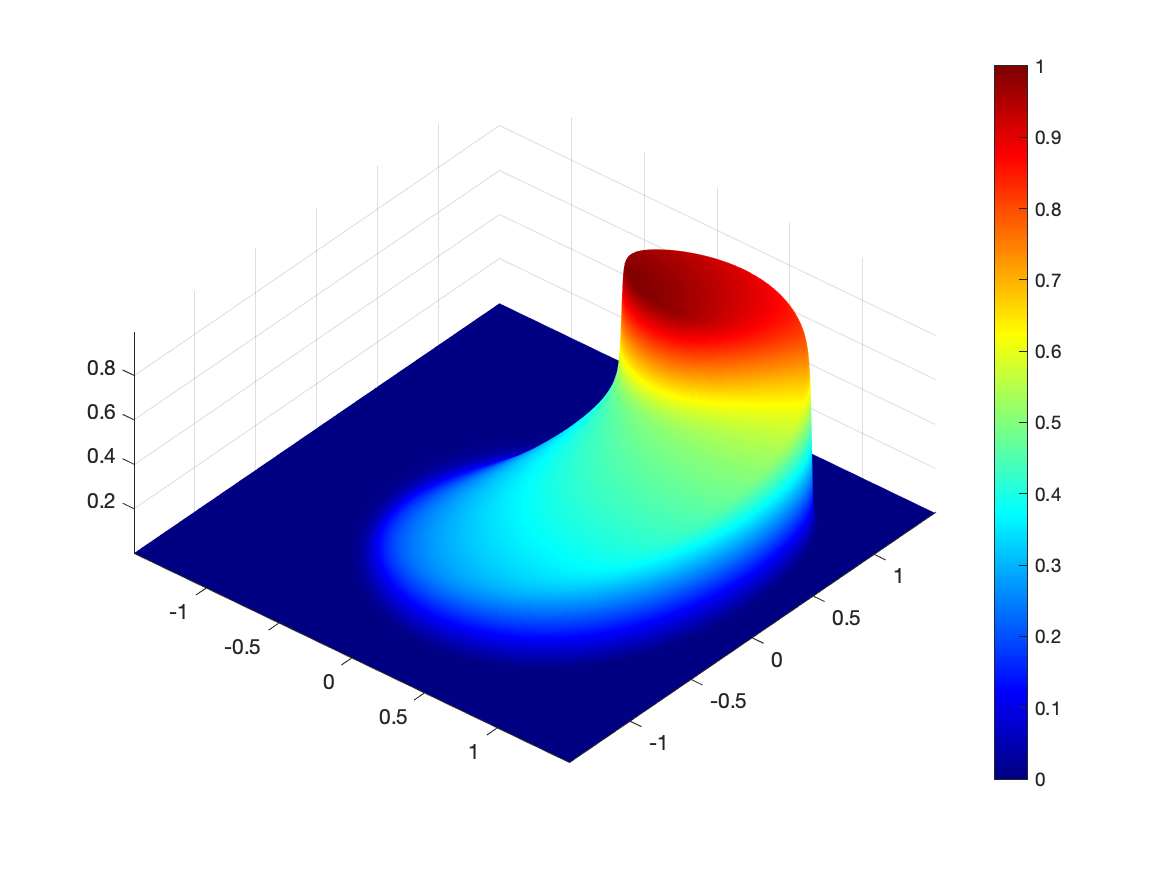}
  \caption{Surface at $T=0.5$}
 \end{subfigure}
 \caption{\textbf{Example \ref{ex:BL2D}. Two-dimensional test.}
 Numerical results of the two-dimensional viscous Buckley--Leverett equation obtained from the ETD-RK4 and $\mathbb{Q}^3$-DG method.}
 \label{fig:BL2D}
\end{figure}

\section{Conclusions}\label{sec:Remarks}
In this paper, we consider the fully discrete ETD-RKDG schemes for solving the advection-diffusion equations \eqref{eq:AdvDiffEq}. The stability and time-step constraints of the schemes are analyzed using the Fourier method. We find that for the DG method with central flux for the advection term, the ETD-RKDG schemes are stable if $\tau \leq \tau_0 \frac{d}{a^2}$, while for the DG method with upwind flux for the advection term, the ETD-RKDG schemes are stable if $\tau \leq \max\left\{\tau_0 \frac{d}{a^2}, c_0 \frac{h}{a}\right\}$, with $c_0$ being the CFL number of explicit RKDG methods for the purely advection equation. 
The value of \(\tau_0\) depends on the specific ETD-RK method. 
However, for a given ETD-RK method, \(\tau_0\) of the fully discrete ETD-RKDG scheme coincides with that of the semi-discrete ETD-RK scheme and appears to be independent of the choice of DG method and polynomial degrees in the spatial discretization.
The stability conditions are proved for the lowest-order case and are found numerically using the Fourier method for higher-order cases.

\begin{appendices}

\appendix

\section{Results for IMEX-RKDG schemes}\label{app:IMEX}
We consider the following IMEX-RK time discretization schemes of orders ranging from first to third: ARS(1,1,1), ARS(2,2,2), and ARS(4,4,3), as proposed in \cite{ascher1997implicit}. 
These schemes were analyzed in \cite{wang2023uniform}. 
However, the values of $\tau_0$ and their insensitivity to spatial discretization have not been explored.

\begin{itemize}
   \item IMEX-RK1 (first-order):
   \begin{equation}\label{eq:IMEXRK1}
   \begin{split}       \mathbf{u}^{n+1}=&(I-\tau D)^{-1}\left( \mathbf{u}^n+\tau F(\mathbf{u}^n) \right),
   \end{split}
   \end{equation}
   \item IMEX-RK2 (second-order):
   \begin{equation*}
   \begin{split}
       \mathbf{a}^n=&(I-\tau\gamma D)^{-1}\left( \mathbf{u}^n+\tau\gamma F(\mathbf{u}^n) \right),\\
       \mathbf{u}^{n+1}=&(I-\tau\gamma D)^{-1}\left( \mathbf{u}^n+\tau \delta F(\mathbf{u}^n)+\tau(1-\delta)F(\mathbf{a}^n)+\tau(1-\gamma)D \mathbf{a}^n \right),
   \end{split}
   \end{equation*}
    \item IMEX-RK3 (third-order):
    \begin{equation*}
    \begin{split}
        \mathbf{a}^n=(I-\frac{\tau}{2}D)^{-1}&\left( \mathbf{u}^n+\frac{\tau}{2}F(\mathbf{u}^n)  \right),\\
        \mathbf{b}^n=(I-\frac{\tau}{2}D)^{-1}&\left( \mathbf{u}^{n}+\tau(\frac{11}{18}F(\mathbf{u}^n)+\frac{1}{18}F(\mathbf{a}^n)+\frac{1}{6}D\mathbf{a}^n) \right),\\
        \mathbf{c}^n=(I-\frac{\tau}{2}D)^{-1}&\left( \mathbf{u}^n+\tau (\frac{5}{6}F(\mathbf{u}^n)-\frac{5}{6}F(\mathbf{a}^n)+\frac12 F(\mathbf{b}^n)-\frac12 D\mathbf{a}^n+\frac12 D\mathbf{b}^n)  \right),\\
        \mathbf{u}^{n+1}=(I-\frac{\tau}{2}D)^{-1}&\left( \mathbf{u}^n+\tau( \frac14F(\mathbf{u}^n)+\frac{7}{4}F(\mathbf{a}^n)+\frac34 F(\mathbf{b}^n)-\frac74 F(\mathbf{c}^n)\right.\\
        &\qquad\qquad\left. +\frac32 D\mathbf{a}^n-\frac32 D\mathbf{b}^n+\frac12D\mathbf{c}^n)  \right),\\
    \end{split}
    \end{equation*}

   % \item IMEX-RK3-b (third-order):
   % \begin{equation}\label{eq:IMEXRK3}
   % \begin{split}
   % \mathbf{a}^n=&(1-\tau\theta D)^{-1}\left( \mathbf{u}^n+\tau\theta F(\mathbf{u}^n) \right)\\
   % \mathbf{b}^n=&(1-\tau\theta D)^{-1}\left( \mathbf{u}^{n}+\tau((\frac{1+\theta}{2}-\alpha_1)F(\mathbf{u}^n)+\alpha_1F(\mathbf{a}^n)+\frac{1-\theta}{2}D\mathbf{a}^n) \right)\\
   % \mathbf{c}^n=&(1-\tau\theta D)^{-1}\left( \mathbf{u}^n+\tau( (1-\alpha_2)F(\mathbf{a}^n)+\alpha_2F(\mathbf{b}^n)+\beta_1 D\mathbf{a}^n+\beta_2 D\mathbf{b}^n  ) \right)\\
   % \mathbf{u}^{n+1}=&\mathbf{u}^n+\tau\left( \beta_1 F(\mathbf{a}^n)+\beta_2 F(\mathbf{b}^n)+\theta F(\mathbf{c}^n)+\beta_1 D\mathbf{a}^n+\beta_2 D\mathbf{b}^n +\theta D\mathbf{c}^n \right)
   % \end{split}
   % \end{equation}

\end{itemize}
where $\gamma=(1-\frac{\sqrt{2}}{2}), \delta=-\frac{\sqrt{2}}{2}$.
% $\theta\approx0.435866521508459$ is the middle root of $6x^3 - 18x^2 + 9x - 1 = 0$, $\beta_1 = -\frac{3}{2} \theta^2 + 4\theta - \frac{1}{4}$, $\beta_2 = \frac{3}{2} \theta^2 - 5\theta + \frac{5}{4}$, 
% $\alpha_1=-\frac{1}{4}$ and $\alpha_2 = \frac{\frac{1}{3} - 2\theta^2 - 2\beta_2\alpha_1\theta}{\theta(1-\theta)}$.

Applying \eqref{eq:IMEXRK1} for the lowest-order central scheme \eqref{eq:FD_central}, we obtain
\begin{equation}\label{eq:IMEXRK1-FD}
\mathbf{u}^{n+1}=(I-\tau D)^{-1}\left(I - \tau A \right)\mathbf{u}^n,
\end{equation}
with $D$ and $A$ defined the same as in \eqref{eq:ETD-CentralFD-D} and \eqref{eq:ETD-CentralFD-A}, respectively.
The growth factor $\widehat{G}(\tau,h,\omega)$ for the scheme is given by 
\begin{equation*}
\widehat{G}(\tau,h,\omega)=\frac{1-\ii\frac{2\tau}{h}\sin(\frac{\omega h}{2})\cos(\frac{\omega h}{2})}{1+\frac{4\tau}{h^2}\sin^2(\frac{\omega h}{2})}.
\end{equation*}
The stability and time-step constraint are stated in the following theorem.
\begin{thm}\label{thm:A1}
The scheme \eqref{eq:IMEXRK1-FD}, \eqref{eq:ETD-CentralFD-D} and \eqref{eq:ETD-CentralFD-A} is stable under the time-step constraint $\tau\leq2$, with the growth factor $|\widehat{G}(\tau,h,\omega)|\leq 1$. 
\end{thm}
\begin{proof}
One can calculate that 
\begin{equation*}
\begin{split}
|\widehat{G}(\tau,h,\omega)|^2&=\frac{1+\frac{4\tau^2}{h^2}\eta(1-\eta)}{\left( 1+\frac{4\tau}{h^2}\eta \right)^2}\\
&=:Q(\tau,h,\eta),\quad \text{where } \eta=\sin^2(\frac{\omega h}{2})\in[0,1].    
\end{split}
\end{equation*}
With direct computation, one can get
\begin{equation*}
\begin{split}
R(\tau,h,\eta)&:=1+\frac{4\tau^2}{h^2}\eta(1-\eta) - \left( 1+\frac{4\tau}{h^2}\eta \right)^2\\
&=\frac{4\tau\eta}{h^2}\left( \tau-2-\tau\eta-\frac{4\tau\eta}{h^2} \right) \leq 0\quad\forall \eta\in[0,1], h>0, \tau \leq 2.
\end{split}
\end{equation*}
Therefore, we have $Q(\tau,h,\eta)\leq 1$, which completes the proof.
\end{proof}
\begin{coro}
    The central scheme  for \eqref{eq:AdvDiffEq},    \begin{equation*}
\frac{\dd u_{j}}{\dd t} + a\frac{u_{j+1}-u_{j-1}}{2h} = d\frac{u_{j+1}-2u_{j}+u_{j-1}}{h^2},
\end{equation*}
with the IMEX-RK1 time discretization is stable under the time-step constraint $\tau\leq2\frac{d}{a^2}$.
\end{coro}

Similarly, applying \eqref{eq:IMEXRK1} for the lowest-order upwind scheme \eqref{eq:FD_upwind}, we obtain \eqref{eq:IMEXRK1-FD}
with $D$ and $A$ defined the same as in \eqref{eq:ETD-CentralFD-D} and \eqref{eq:ETD-UpwindFD-A}, respectively.
The growth factor $\widehat{G}(\tau,h,\omega)$ for the scheme is given by 
\begin{equation*}
\widehat{G}(\tau,h,\omega)=\frac{1-\frac{2\tau}{h}\sin^2(\frac{\omega h}{2})-\ii\frac{2\tau}{h}\sin(\frac{\omega h}{2})\cos(\frac{\omega h}{2})}{1+\frac{4\tau}{h^2}\sin^2(\frac{\omega h}{2})}.
\end{equation*}
The stability and time-step constraint are stated in the following theorem.
\begin{thm}\label{thm:A2}
The scheme \eqref{eq:IMEXRK1-FD}, \eqref{eq:ETD-CentralFD-D} and \eqref{eq:ETD-UpwindFD-A} is stable under the time-step constraint $\tau\leq2+h$, with the growth factor $|\widehat{G}(\tau,h,\omega)|\leq 1$. 
\end{thm}
\begin{proof}
One can calculate that 
\begin{equation*}
\begin{split}
|\widehat{G}(\tau,h,\omega)|^2&=\frac{(1-\frac{2\tau}{h}\eta)^2+\frac{4\tau^2}{h^2}\eta(1-\eta)}{\left( 1+\frac{4\tau}{h^2}\eta \right)^2}\\
&=:Q(\tau,h,\eta),\quad \text{where } \eta=\sin^2(\frac{\omega h}{2})\in[0,1].    
\end{split}
\end{equation*}
With direct computation, one can get
\begin{equation*}
\begin{split}
R(\tau,h,\eta)&:=(1-\frac{2\tau}{h}\eta)^2+\frac{4\tau^2}{h^2}\eta(1-\eta) - \left( 1+\frac{4\tau}{h^2}\eta \right)^2\\
&=\frac{4\tau\eta}{h^2}\left( \tau-2-h-\frac{4\tau\eta}{h^2} \right) \leq 0\quad\forall \eta\in[0,1], h>0, \tau \leq 2+h.
\end{split}
\end{equation*}
Therefore, we have $Q(\tau,h,\eta)\leq 1$, which completes the proof.
\end{proof}
\begin{coro}
    The upwind scheme for \eqref{eq:AdvDiffEq},    \begin{equation*}
\frac{\dd u_{j}}{\dd t} + a\frac{u_{j}-u_{j-1}}{h} = d\frac{u_{j+1}-2u_{j}+u_{j-1}}{h^2},
\end{equation*}
with IMEX-RK1 time discretization is stable under the time-step constraint $\tau \leq 2\frac{d}{a^2}+\frac{h}{a}$.
\end{coro}

\begin{rem} 
The time-step constraint $\tau \leq 2\frac{d}{a^2} + \frac{h}{a}$ was also derived in \cite{wang2023uniform}; see the end of Section 5 therein.
This condition is less restrictive compared to that of the ETD-RK1 method with the upwind flux. Our numerical tests have shown that the ETD-RK1 method exhibits $|\widehat{G}(\tau,h,\omega)| > 1$ with the time step $\tau = 2\frac{d}{a^2} + \frac{h}{a}$. 
\end{rem}
Following a similar approach as in Section \ref{sec:HigherOrder}, we can analyze the stability and time-step constraints of the semi-discrete and fully discrete high-order IMEX-RK schemes using Fourier methods. The proofs and calculations are omitted, as they follow similar steps to the analysis of the ETD-RKDG methods in Section \ref{sec:HigherOrder} and are not the  focus of this paper.

In summary, we have obtained the following results:
\begin{itemize}
    \item For the semi-discrete IMEX-RK schemes, the growth factor satisfies $|\widehat{G}(\tau,\xi)|\leq 1, \forall \xi\in\mathbb{R}$ if $\tau\leq\tau_0$. 
    The values of $\tau_0$ for various IMEX-RK methods, along with the plots of $|\widehat{G}(\tau_0,\xi)|^2$ versus $\xi$, are shown in Figure \ref{fig:G2vsxi_IMEX}.

    \item For the fully discrete IMEX-RKDG schemes, numerical experiments show that
\begin{equation}
\sup_{\xi\in[-\pi,\pi]}\rho(\widehat{G}(\tau,h,\xi))\leq 1\quad\forall h>0,
\end{equation}
under the time-step constraints $\tau\leq\tau_0\frac{d}{a^2}$ and $\tau\leq\max\{\tau_0\frac{d}{a^2},c_0 \frac{h}{a}\}$ for the central \eqref{eq:CentralFlux} and upwind \eqref{eq:UpwindFlux} DG fluxes, respectively, where the values of $\tau_0$ are provided in Table \ref{tab:tau_0_IMEX}, and $c_0$ is the CFL constant of the explicit RKDG method for the purely advection equation. 
Detailed numerical searches reveal that the exact values of \(\tau_0\) for fully discrete IMEX-RKDG methods coincide with those of the semi-discrete IMEX-RK schemes. 
%\added[id=reviewer0]{to at least the first ten decimal places.}.
\end{itemize}

\begin{figure}[!htbp]
 \centering
 \begin{subfigure}[b]{0.25\textwidth}
  \includegraphics[width=\textwidth]{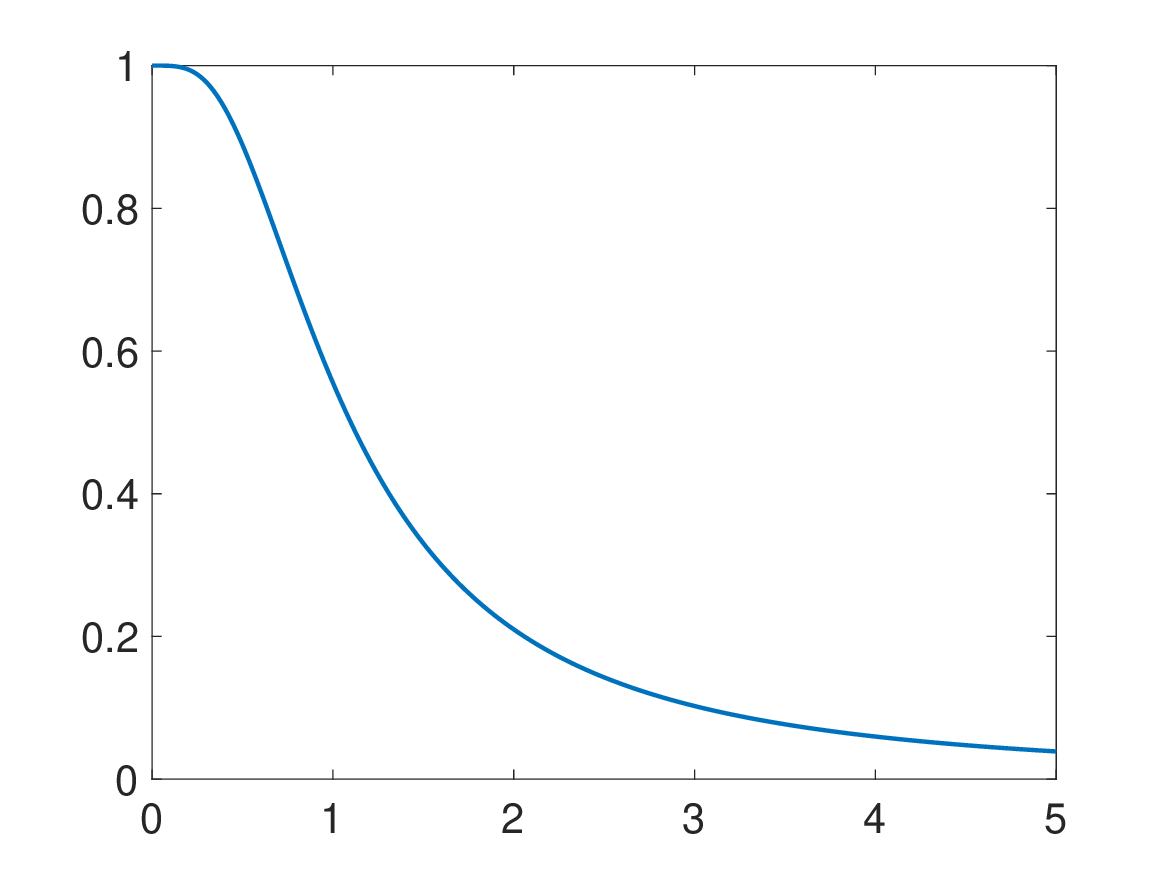}  \caption{IMEX-RK1, $\tau_0=2$}
 \end{subfigure}
  \begin{subfigure}[b]{0.25\textwidth}
  \includegraphics[width=\textwidth]{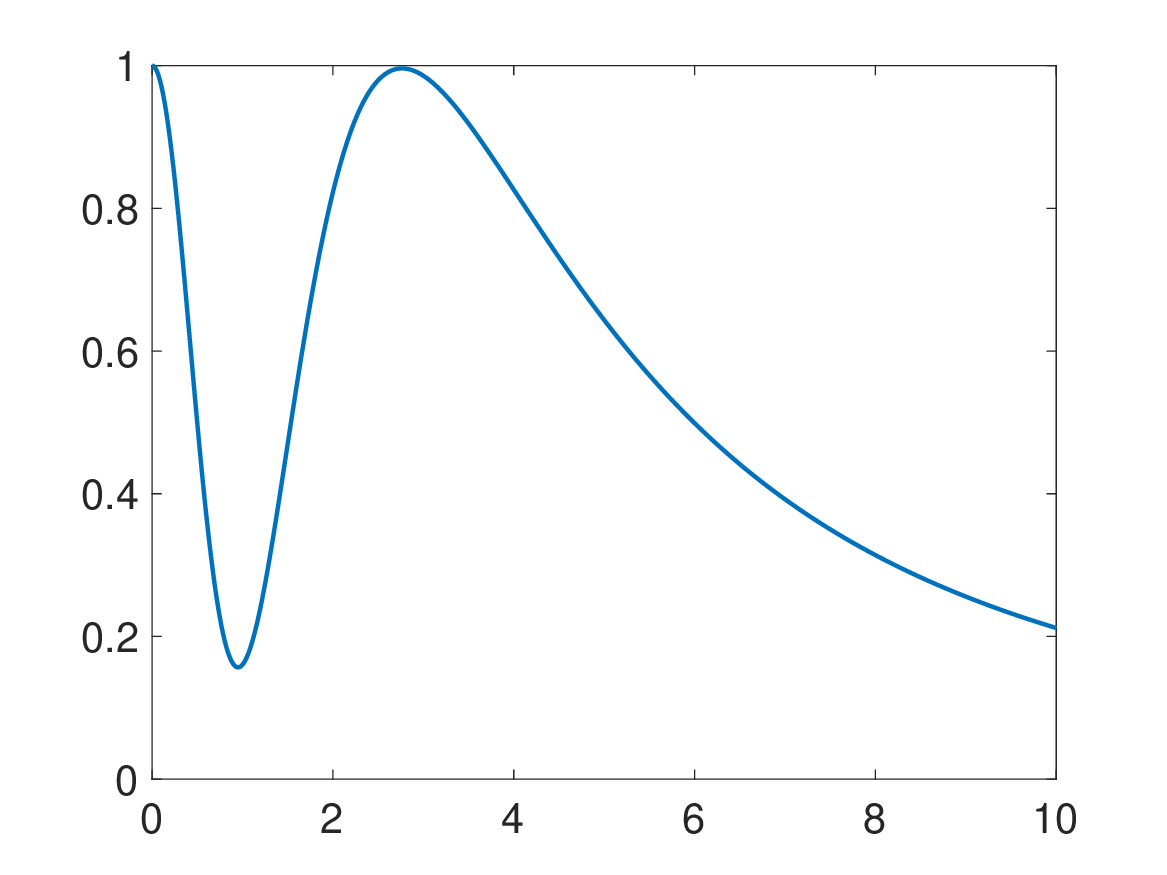}
\caption{IMEX-RK2, $\tau_0=1.38$}
 \end{subfigure}
 \begin{subfigure}[b]{0.25\textwidth}
  \includegraphics[width=\textwidth]{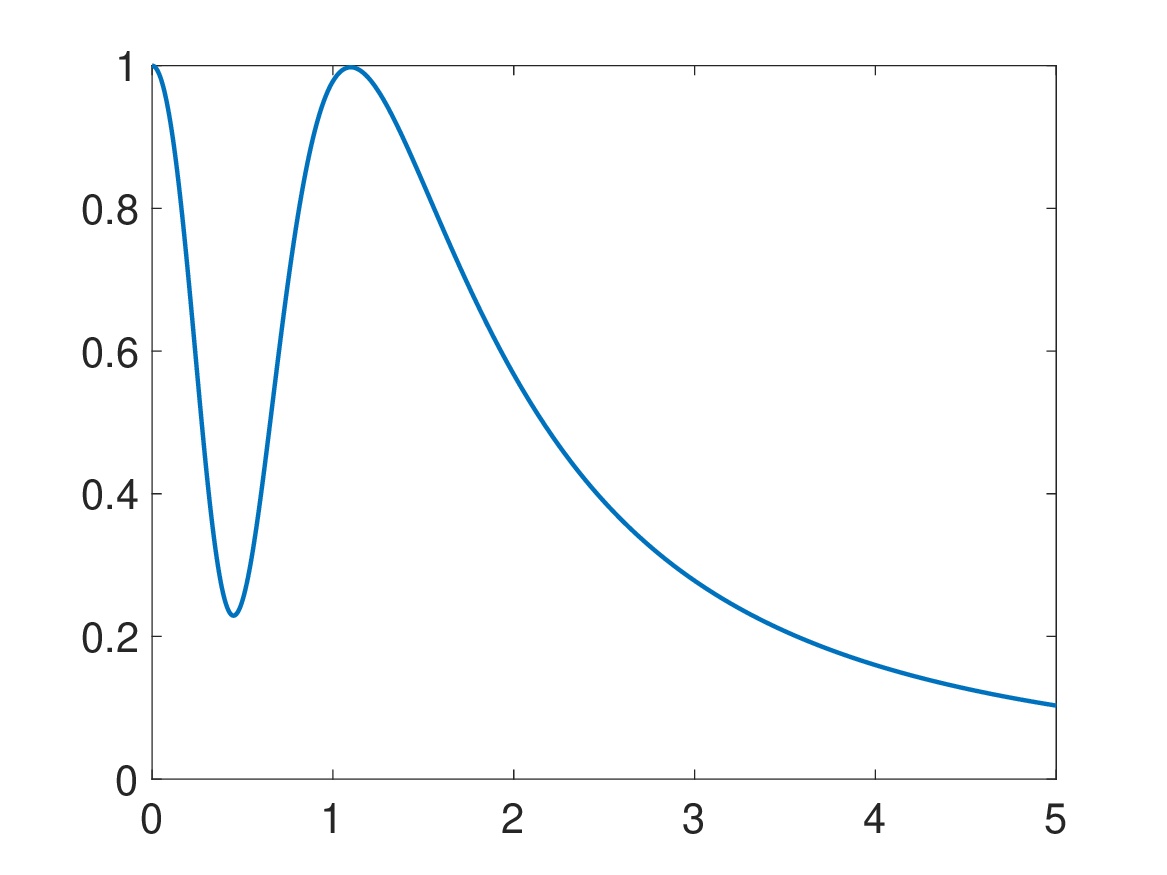}
\caption{IMEX-RK3, $\tau_0=3.89$}
 \end{subfigure}
 \caption{The square of growth factor $|\widehat{G}(\tau_0,\xi)|^2$ versus $\xi$. 
 Since $|\widehat{G}(\tau_0,\xi)|^2$ is an even function with respect to $\xi$, we present only $\xi>0$.
 The value of $\tau_0=2$ for IMEX-RK1 is sharp.
 The searched values of $\tau_0$ for IMEX-RK2 and IMEX-RK3 are valid up to the last digit shown.}
 \label{fig:G2vsxi_IMEX}
\end{figure}

\begin{table}[h!]
\centering
\begin{tabular}{|c|c|c|c|}
\hline
method & IMEX-RK1 & IMEX-RK2 & IMEX-RK3 \\ \hline
$\tau_0$ & 2 & 1.38 & 3.89  \\ \hline
\end{tabular}
\caption{Stable $\tau_0$ for $\sup_{\xi\in[-\pi,\pi]}\rho(\widehat{G}(\tau,h,\xi))\leq 1 \ \forall h>0$.}\label{tab:tau_0_IMEX}
\end{table}
Finally, we visualize the square of the growth factor, $\rho(\widehat{G}(\tau_0,h,\xi))^2$, as a function of $\xi$ for a specific spatial discretization setting: central DG flux for the advection combined with LDG discretization for diffusion, using a $\mathbb{P}^4$ polynomial space and $h = \frac{\pi}{10^{6}}$. 
The results for other spatial discretization choices are close.
From Figure \ref{fig:G2_DG_vsxi_IMEX}, we observe that the pattern of $\rho(\widehat{G}(\tau_0,h,\xi))^2$ closely resembles their semi-discrete counterparts $|\widehat{G}(\tau_0,\xi)|^2$.
\begin{figure}[!htbp]
 \centering
 \begin{subfigure}[b]{0.25\textwidth}
  \includegraphics[width=\textwidth]{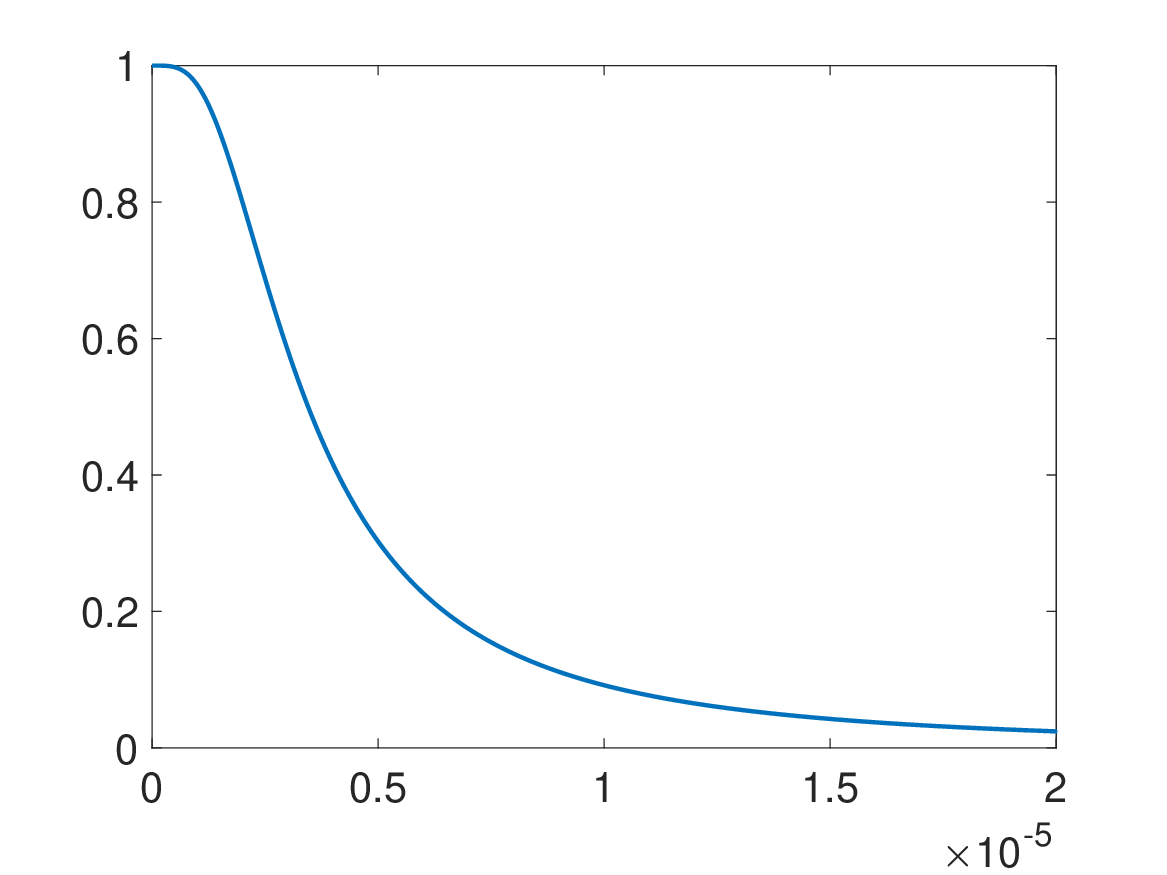}  \caption{IMEX-RK1, $\tau_0=2$}
 \end{subfigure}
  \begin{subfigure}[b]{0.25\textwidth}
  \includegraphics[width=\textwidth]{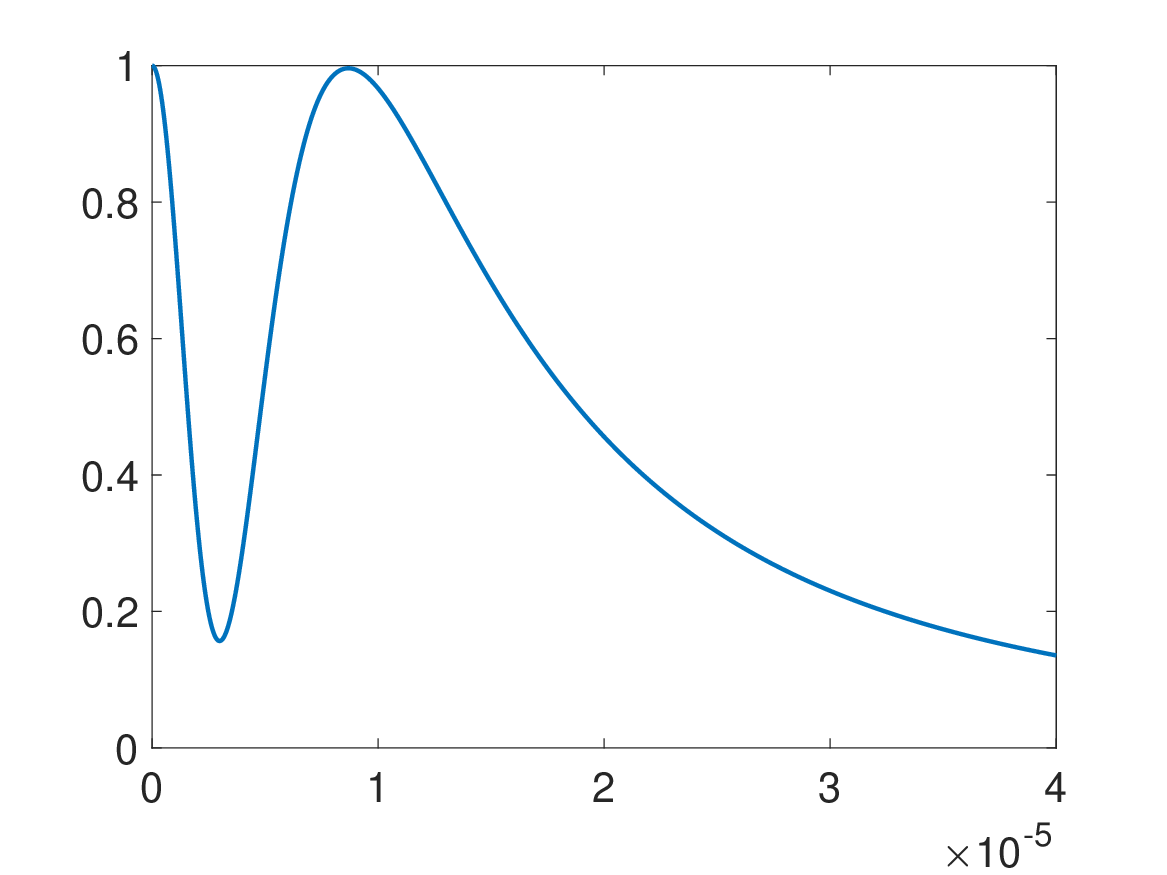}
\caption{IMEX-RK2, $\tau_0=1.38$}
 \end{subfigure}
 \begin{subfigure}[b]{0.25\textwidth}
  \includegraphics[width=\textwidth]{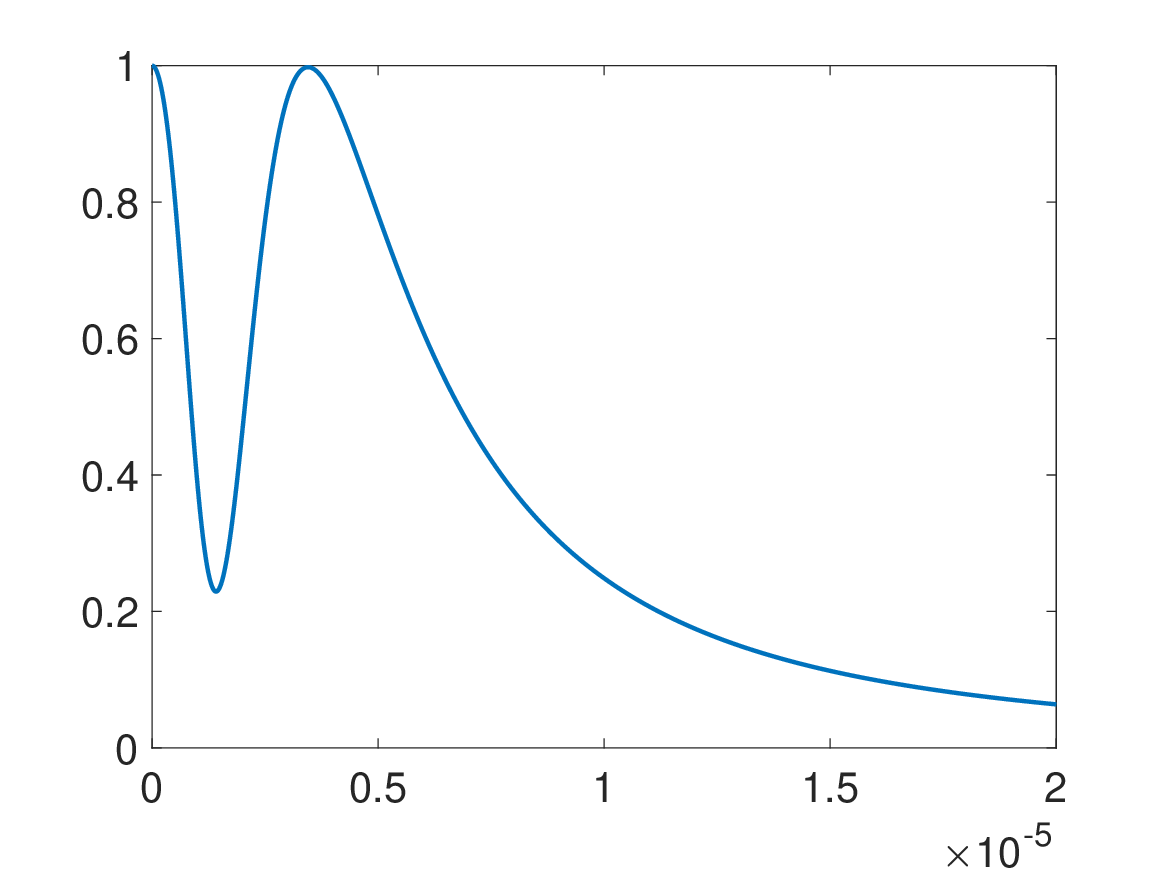}
\caption{IMEX-RK3, $\tau_0=3.89$}
 \end{subfigure}
 \caption{The square of growth factor $\rho(\widehat{G}(\tau_0,h,\xi))^2$ versus $\xi=\omega h$ for a specific spatial discretization setting: central DG flux for the advection combined with LDG discretization for diffusion, using a $\mathbb{P}^4$ polynomial space and $h = \frac{\pi}{10^{6}}$.
 The results for other spatial discretization choices are close.}
 \label{fig:G2_DG_vsxi_IMEX}
\end{figure}
\begin{rem}
A previous study \cite{wang2023uniform} on the stability of IMEX-LDG methods for the linear advection-diffusion equation derived the sufficient condition  
\begin{equation}\label{eq:IMEXCFL}  
\tau \leq \max\left\{\tau_0 \frac{d}{a^2}, \min\left\{c_0 \frac{h}{a}, \rho \frac{h^2}{d} \right\} \right\}.  
\end{equation}  
Here, \(\tau_0\) and \(c_0\) have the same meanings as in this paper, while \(\rho\) is an unknown positive constant independent of \(a\), \(d\), and \(h\). This condition includes an additional constraint and appears more restrictive than our conclusion.

To further investigate the necessity of the additional condition with the $\rho \frac{h^2}{d}$ term, we conducted additional validation using the dimensionless-form equation with \(a = d = 1\). In this case, \eqref{eq:IMEXCFL} simplifies to $\tau \leq \max\{\tau_0, \min\{c_0h,\rho h^2\}\}$, and the $\rho h^2$ condition is activated when $\tau_0 \leq \rho h^2 \leq c_0h$. Note that $\{h : \tau_0 \leq \rho h^2 \leq c_0h\} \subseteq \{h : h \geq \frac{\tau_0}{c_0}\}$. Thus, we tested a wider range of values with \(h \in \left[\frac{\tau_0}{c_0}, \infty\right)\).  
With \(\tau = c_0 h\), we did not observe any instance where \(\rho(\widehat{G}(\tau, h, \xi)) > 1\). This seems to suggest that the $\rho \frac{h^2}{d}$ condition can always be replaced by the $c_0 \frac{h}{a}$ condition for stability.

\end{rem}

\end{appendices}

\bigskip
\noindent{\bf Data Availability}

\noindent No data sets were generated during the current study of this paper.

\bigskip
\noindent{\bf \Large Declarations}

\bigskip
\noindent{\bf Conflict of interest}

\noindent The authors declare that we have no conflict of interest.

%\bibliography{refs}
%\bibliographystyle{abbrv}

\end{document}